%% file: main.tex
\documentclass[12pt, oneside]{article}
\usepackage{amsmath, amsthm, amssymb, calrsfs, wasysym, verbatim, bbm, color, xcolor, graphics, geometry, tikz-cd}
\usepackage[all,cmtip]{xy}
\usepackage{upgreek}
\usepackage{lmodern}
\usepackage[normalem]{ulem}
\usepackage{float, graphicx}
\usepackage{blindtext}
\usepackage{caption,threeparttable}
\usepackage{subcaption}
\usepackage{xcolor}
\usepackage{euscript}
\usetikzlibrary{math}
\usepackage{enumitem}
\usetikzlibrary{patterns.meta}
\usepackage{standalone}
\usepackage{mathtools}
\usepackage{sectsty}
\allsectionsfont{\scshape\color{black}}

\usepackage{url}
\usepackage[bookmarksopen,colorlinks,allcolors=orange!80!black,urlcolor=orange!80!black]{hyperref}

\usepackage{mathalpha}

\DeclareMathAlphabet{\mathdutchcal}{U}{dutchcal}{m}{n}

\DeclareFontFamily{U}{stixtwobb}{\skewchar\font=45}%
\DeclareFontShape{U}{stixtwobb}{m}{n}{<->\mathalfa@bbscaled stix2-mathbb}{}
\DeclareMathAlphabet{\mathbbst}{U}{stixtwobb}{m}{n}

\makeindex

\usepackage{scalerel,stackengine}

\newcommand{\R}{\ensuremath{\mathbbst{R}}}

\newcommand{\act}{\mathsf{act}}
\newcommand{\inert}{\mathsf{inert}}
\newcommand{\lc}{\mathsf{lc}}

\newcommand{\E}{\ensuremath{\mathbbst{E}^{\otimes}}}
\newcommand{\Eactive}{\ensuremath{\mathbbst{E}^{\otimes, \,\act}}}
\newcommand{\Cube}[2]{{\mathsf{Cube}^{\otimes}_{{#1},{#2}}}}
\newcommand{\CubeOrd}[2]{{\mathsf{Cube}_{{#1},{#2}}}}
\newcommand{\CubeActive}[2]{\mathsf{Cube}_{{#1},{#2}}^{\otimes, \,\act}}

\newcommand{\Prect}{\ensuremath{\mathsf{P}}}
\newcommand{\Drect}{\ensuremath{\mathsf{D}}}
\newcommand{\gE}{\ensuremath{\EuScript{E}}}
\newcommand{\cMfld}{\ensuremath{\,\!_{\lrcorner}\EuScript{M}\mathsf{fld}_n}}
\newcommand{\cgE}{\ensuremath{\,\!_{\lrcorner}\EuScript{E}_n}}

\newtheorem{thm}{Theorem}[section]
\newtheorem{prop}[thm]{Proposition}
\newtheorem{lemma}[thm]{Lemma}
\newtheorem{cor}[thm]{Corollary}

\newtheorem{hypothesis}[thm]{Hypothesis}
\theoremstyle{remark}
\newtheorem{remark}[thm]{Remark}
\newtheorem{remarks}[thm]{Remarks}
\theoremstyle{definition}
\newtheorem{defn}[thm]{Definition}

\newtheorem{notation}[thm]{Notation}
\newtheorem{example}[thm]{Example}
 \theoremstyle{plain}
\newtheorem{customthm}{Theorem}

 \theoremstyle{plain}
\newtheorem{customcor}{Corollary}

\DeclareMathOperator{\Rect}{\mathsf{Rect}}
\DeclareMathOperator{\Map}{\mathsf{Map}}
\DeclareMathOperator{\Emb}{\mathsf{Emb}}
\DeclareMathOperator{\Fr}{\mathsf{Fr}}
\DeclareMathOperator{\Conf}{\mathsf{Conf}}

\DeclareMathOperator{\Fin}{\mathsf{Fin}}

\DeclareMathOperator{\id}{\mathsf{id}}
\DeclareMathOperator{\const}{\mathsf{const}}
\DeclareMathOperator{\Fact}{\mathsf{Fact}}
\DeclareMathOperator{\FACT}{\normalfont{\textsc{Fact}}}
\DeclareMathOperator{\cbl}{\mathsf{cbl}}
\DeclareMathOperator{\holim}{\mathsf{holim}}
\DeclareMathOperator{\hocolim}{\mathsf{hocolim}}
\DeclareMathOperator{\Fun}{\mathsf{Fun}}
\DeclareMathOperator{\BiFun}{\mathsf{BiFun}}
\DeclareMathOperator{\Alg}{\mathsf{Alg}}
\DeclareMathOperator{\ALG}{\normalfont{\textsc{Alg}}}
\DeclareMathOperator{\open}{\mathsf{open}}
\DeclareMathOperator{\op}{\mathsf{op}}
\DeclareMathOperator{\Nerve}{\mathsf{N}}

\DeclareMathOperator{\Mon}{\mathsf{Mon}}

\DeclareMathOperator{\Cat}{\EuScript{C}\mathsf{at}}
\DeclareMathOperator{\Opd}{\EuScript{O}\mathsf{pd}}
\DeclareMathOperator{\Spc}{\EuScript{S}}

\DeclareMathOperator{\Basic}{\EuScript{B}\mathsf{sc}}

\newcommand{\brbinom}[2]{\genfrac{[}{]}{0pt}{}{#1}{#2}}

\newcommand{\vc}[1]{\color{blue} \{VC: #1\} \color{black}}

\newcommand{\Rbullet}{R_{\bullet}}
\newcommand{\Qbullet}{Q_{\bullet}}
\newcommand{\Rbulletn}{\{\Rbullet\}_{\langle n \rangle}}
\newcommand{\Qbulletm}{\{\Qbullet\}_{\langle m \rangle}}

\title{\textcolor{black}{\Large{\textsc{Additivity of constructible factorization algebras over manifolds with corners}}}}

\author{Victor Carmona, Anja Švraka
\thanks{V.C. was partially supported by PID2020-117971GB-C21 funded by MCIN/AEI. A.Š. was supported by SFB 1085: Higher Invariants from the Deutsche Forschungsgemeinschaft (DFG).}}

\date{October 2025}

\begin{document}

\maketitle

\begin{abstract}
We prove the statement in the title, solving in this way a conjecture stated by Ginot in \cite{ginot_notes_2013} for manifolds with corners. Along the way, we establish a derived Swiss-cheese additivity theorem and an alternative proof for the hyperdescent of factorization algebras over those manifolds.
 \end{abstract}
 \vspace{-2mm}

\vspace{-2mm}

\vspace{-2mm}

{\hypersetup{linkcolor=black}
\tableofcontents}

\section{Introduction}\label{sect: Intro}

Factorization algebras are mathematical objects that embody a local-to-global principle which accounts for simultaneous information around finite families of points. Despite their recent introduction, they have already shown remarkable applications in diverse fields of mathematics, including mathematical physics, algebraic and differential topology, and representation theory. Their algebro-geometric version was originally introduced by Beilinson and Drinfeld in their seminal work on vertex algebras, \cite{beilinson_chiral_2004}, and later extended by Francis and Gaitsgory to higher dimensions \cite{francis_chiral_2012}. Around the same time, the topological counterpart was developed by Lurie, Ayala--Francis, Andrade... in a series of works, e.g.\ \cite{lurie_classification_2009} and \cite{ayala_factorization_2015}, along with the closely related notion of chiral or factorization homology (see also \cite{ginot_higher_2014,horel_factorization_2017}). Recent progress by Hennion--Kapranov \cite{hennion_gelfand_2023} builds a bridge between these two worlds. Another remarkable use of these objects comes from the work of Costello--Gwilliam \cite{costello_factorization_2017,costello_factorization_2023} who identified their structure as the minimal axiom system encoding the observables of a perturbative quantum field theory (pQFT) in Riemannian signature, and systematically constructed examples of those via a striking combination of techniques. 


 More concretely, A factorization algebra $\mathdutchcal{F}$ on a topological space \(X\) in a symmetric monoidal \(\infty\)-category \(\EuScript{V}\) (satisfying mild requirements) is a rule that assigns, to every open subset \(U\subseteq X\), an object \(\mathdutchcal{F}(U) \in \EuScript{V}\), and to every inclusion of disjoint open subsets \(U_1 \sqcup U_2 \sqcup \dots \sqcup U_n \subseteq V\), a product-morphism in \(\EuScript{V}\)
    \[
    \mathdutchcal{F}(U_1) \otimes  \mathdutchcal{F}(U_2) \otimes \dots \otimes \mathdutchcal{F}(U_n) \xrightarrow[]{\;\;\quad\;\;} \mathdutchcal{F}(V),
    \]
 subject to natural associativity, unitality, and equivariance conditions. Additionally, $\mathdutchcal{F}$ is required to satisfy two properties: \emph{multiplicativity} and \emph{Weiss codescent} (see~\textsection\ref{sect: Preliminaries}, \cite{costello_factorization_2017, ginot_notes_2013} or \cite{karlsson_assembly_2024}). An obvious but crucial observation is that factorization algebras on $X$ can be seen as special algebras over an ($\infty$-)operad $\mathsf{Disj}_X$ built on ``disjoint open subsets" in $X$. We adopt this operadic perspective in the present work. 

When the underlying space $X$ comes with more structure, e.g., as is the case of interest for us, $X$ is a manifold with corners and hence it comes with a natural stratification, we can further impose an extra ``topological" property on $\mathdutchcal{F}$ called \emph{constructibility} (see Definition \ref{defn: Constructible Fact Algs}). Such a condition is just the natural generalization of local-constancy in singular or stratified settings, and it is strongly connected to TQFTs with defects (see \cite{ayala_factorization_2017, contreras_defects_2023}). Informally, it means that the value of $\mathdutchcal{F}$ on the basic building blocks of $X$, i.e.\ the open subsets $U\subseteq X$ isomorphic to $\R^{p,q}:=\R^p_{\phantom{\geq 0}}\!\!\!\!\times \R^{q}_{\geq 0}$ for some $(p,q)$ in the case of manifolds with corners, is independent of the size of such as long as they encapsulate the same singularity. As we will see, this is arguably the most powerful axiom among the three (e.g.\ it is responsible for all the main theorems in this article). Therefore, we will focus on those factorization algebras fulfilling \emph{constructibility}, and denote the symmetric monoidal $\infty$-category they span by $\FACT_{X}^{\cbl}(\EuScript{V})$ in the sequel. 

Our first key result solves a conjecture stated by Ginot, see \cite[\textsection 6.1]{ginot_notes_2013}, for manifolds with corners.\footnote{This question also appears in Berry's PhD thesis as \cite[Conjecture 1.0.5]{berry_additivity_2021}.} The problem concerns the behavior of constructible factorization algebras on a product of spaces, and our solution establishes that they function as expected even in the presence of non-trivial singularities/stratifications (see also Theorems \ref{thm: Global additivity max generality}--\ref{thm: Global additivity mixing local models}):

 \begin{customthm}
    \label{global thm: global additivity}
    (Global additivity; see Theorem \ref{thm: Global additivity})
        Let $\EuScript{V}$ be a presentable symmetric monoidal \(\infty\)-category and let $X$, $Y$ be two manifolds with corners. Then, there is an equivalence of symmetric monoidal \(\infty\)-categories 
        \[
        \FACT^{\cbl}_{X\times Y}(\EuScript{V}) \simeq \FACT^{\cbl}_{X}\big(\FACT^{\cbl}_{Y}(\EuScript{V})\big)
        \]
         induced by pushforward of factorization algebras along $\uppi_X\colon X\times Y\to X$.
    \end{customthm}

To exemplify the content of this result, let us consider a constructible factorization algebra $\mathdutchcal{F}$ on $X\times Y$. For QFT-inclined readers, $\mathdutchcal{F}$ can be seen as encoding the observables of a TQFT with defects over the canonical corner-stratification of $X\times Y$. Then, we could forget most of the structure and only keep the values of $\mathdutchcal{F}$ on open subsets of the form $U\times V$, plus the product-morphisms associated with inclusions of disjoint open subsets on each factor, called \emph{horizontal} and \emph{vertical products} (see Figure \ref{fig: figure 4}). Denoting this data by $(\mathdutchcal{F}\vert_{\mathsf{open}(X)\times \mathsf{open}(Y)},\upmu_{\mathsf{h}},\upmu_{\mathsf{v}})$, Theorem \ref{global thm: global additivity} asserts that we can reconstruct the whole factorization algebra $\mathdutchcal{F}$ from $(\mathdutchcal{F}\vert_{\mathsf{open}(X)\times \mathsf{open}(Y)},\upmu_{\mathsf{h}},\upmu_{\mathsf{v}})$ in an essentially unique way. In other words, to prescribe a constructible factorization algebra on $X\times Y$, you only need to provide a rule $U\times V\mapsto  \mathdutchcal{F}(U\times V)$ for $U\subseteq X$ and $V\subseteq Y$, equipped with compatible horizontal and vertical products, which satisfies the axioms to be a constructible factorization algebra on each variable. Moreover, our methods also show that you can just keep the values of $\mathdutchcal{F}$ over open subsets $U\times V$ with $U\cong \R^{p,q}$ and $V\cong \R^{r,t}$ for some $p,q,r,t$, and only check constructibility on each variable. This vastly simplifies the construction of TQFTs with defects on $ X\times Y$.\footnote{By no means we are referring to TQFTs in the sense of Atiyah's, or a related, definition, just what one would expect as observables of such when embracing Costello-Gwilliam's formalism (see \cite{costello_factorization_2023}).}

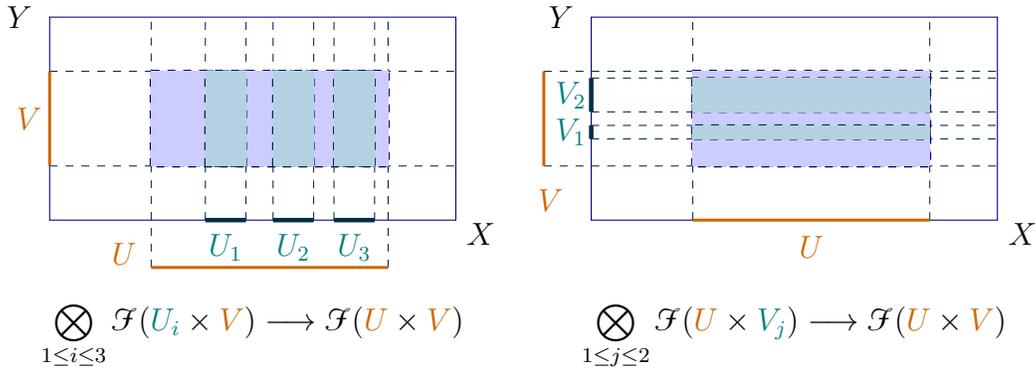
\begin{figure}[htp]
    \centering
    \input{figure4}
    \caption{An instance of a horizontal (left) and a vertical product (right) of $\mathdutchcal{F}$ and the corresponding configuration of disjoint open subsets in $X\times Y$.}
    \label{fig: figure 4}
\end{figure}

Our strategy for proving Theorem \ref{global thm: global additivity} is quite natural, as it relies on proving the claim for the basic building blocks $\R^{p,q}$ first, and then globalizing via a (hyper)descent argument. Thus, it is a suitable path to obtain variations/generalizations for which one can establish additivity for the building blocks of more general, or simply different, stratified manifolds (cf. Theorems \ref{thm: Global additivity max generality}--\ref{thm: Global additivity mixing local models}). Our motivation for choosing manifolds with corners is twofold. On the one hand, both the additivity for building blocks and the differential topology of manifolds with corners are comparatively better controllable and understood than that of, e.g., conically smooth stratified spaces (see \cite{ayala_factorization_2017,ayala_local_2017}), allowing us to provide a self-contained narrative. On the other hand, the case treated in this article will serve as the base step for the generalization to all conically smooth manifolds in the forthcoming work of the second author.

The mentioned additivity for $X=\R^{p,q}$ and $Y=\R^{r,t}$ is obtained by relating it to an extension of Lurie--Dunn's Theorem \cite[Theorem 5.1.2.2]{lurie_higher_nodate} for (generalized) Swiss-cheese operads $\mathbbst{E}_{p,q}$ (a result which holds independent interest), since $\mathbbst{E}_{p,q}$ governs constructible factorization algebras on $\R^{p,q}$ (see Proposition \ref{prop: Coincidence of little (p,q)-disc operads} and Lemma \ref{lem:cbl Fact as operadic algebras}).

\begin{customthm}\label{global thm: local additivity} (Local additivity; see Theorems \ref{thm:SwissCheeseAdditivity} and \ref{thm:SwissCheeseAdditivity Reformulated})
 For any symmetric monoidal \(\infty\)-category $\EuScript{V}$, there is an equivalence of symmetric monoidal \(\infty\)-categories
    $$
    \ALG_{\mathbbst{E}_{p+r,q+t}}(\EuScript{V})\simeq 
    \ALG_{\mathbbst{E}_{p,q}}\big(\ALG_{\mathbbst{E}_{r,t}}(\EuScript{V})\big).
    $$ 
\end{customthm}

See Theorem \ref{thm: Additivity Quadrant Flags} for a related result also generalizing Lurie--Dunn's Theorem.

The proof of Theorem \ref{global thm: local additivity} is heavily inspired by Harpaz's reimagination (see   \cite{harpaz_little_nodate}) of Lurie's ideas in  \cite[\textsection 5]{lurie_higher_nodate}. In particular, we make extensive use of weak \(\infty\)-operads and their approximations (see \textsection \ref{sect:weak operads} for an account on these concepts).  

The last ingredient in our proof of Theorem \ref{global thm: global additivity} is the (hyper)descent of constructible factorization algebras. Such a result was considered folklore for quite some time, but it was not until recently that a complete proof, due to Karlsson--Scheimbauer--Walde, was achieved (see \cite[Theorems A and B]{karlsson_assembly_2024})\footnote{Notice that in loc.cit.\ there is a strong point-set hypothesis on the conically smooth stratified manifold, the existence of ``\emph{enough good disks}", which is proven to hold here for manifolds with corners (see Proposition \ref{prop: Good supply of Weiss covers}).}. In fact, the authors of \cite{karlsson_assembly_2024} develop a rigorous and extensive homotopical toolkit for dealing with (constructible) factorization algebras on conically smooth manifolds that the literature was lacking.

Nonetheless, their proof of descent is quite involved, and relies on a number of intermediate steps. For this reason, we provide an alternative, more streamlined, proof of hyperdescent over manifolds with corners.   

\begin{customthm}\label{global thm: hyperdescent} (Hyperdescent; see Theorem \ref{thm: Hyperdescent for cbl fact})
    Let $X$ be a manifold with corners, $\mathcal{U}\hookrightarrow \open(U)$ be a hypercover of $U\in \open(X)$, and $\EuScript{V}$ be a presentable symmetric monoidal $\infty$-category. Then, the canonical comparison map 
    $$
    \FACT_{U}^{\cbl}(\EuScript{V})\xrightarrow{\;\;\quad\;\;}\underset{V\in\mathcal{U}^{\op}}{\holim}\FACT_{V}^{\cbl}(\EuScript{V})
    $$
    is an equivalence of symmetric monoidal $\infty$-categories. In other words, $\FACT^{\cbl}_{\scalebox{0.4}{$(-)$}\!}(\EuScript{V})$ is a hypersheaf of symmetric monoidal $\infty$-categories over $X$.
\end{customthm}

For establishing Theorem \ref{global thm: hyperdescent}, we generalize a ``disintegration" construction for $\infty$-operads appearing in \cite{harpaz_little_nodate} and \cite{lurie_higher_nodate} which may be of independent interest (see Definition \ref{defn: Directed Integral} and the surrounding discussion). For Harpaz--Lurie's construction, it is crucial to disintegrate over a manifold without boundaries or singularities, whereas our construction deals with this subtlety (see Remark \ref{rem: Why directed integral is needed}). It is also important to emphasize that our proof of Theorem \ref{global thm: hyperdescent} works in more generality; we restrict our treatment to manifolds with corners for coherence and self-containment. 

\paragraph{Applications.} Let us conclude this introduction by mentioning some (possible) applications of our main results.

We start with a direct consequence of Theorem \ref{global thm: global additivity}: it can be used to close the gap in understanding constructible factorization algebras on cones \(\mathsf{C} X\), where \(X\) is a manifold with corners. That is, combining \cite[Theorem 6.12]{karlsson_assembly_2024} with Theorem \ref{global thm: global additivity}, one obtains:

    \begin{customcor} 
    \cite[Question 6.11]{karlsson_assembly_2024} 
    Let \(X\) be a manifold with corners. Then, a $\EuScript{V}$-valued constructible factorization algebra on the cone \(\mathsf{C} X\) corresponds to a pair $(\mathdutchcal{A},\mathdutchcal{M})$ where  \(\mathdutchcal{A}\in \FACT_{X}^{\cbl}(\ALG_{\mathbbst{E}_1}(\EuScript{V}))\) is a constructible factorization algebra on $X$ valued in (unital) associative algebras, and  $\mathdutchcal{M}$ is a pointed module in $\EuScript{V}$ over \(\int_{X}\mathdutchcal{A}=\mathdutchcal{A}(X)\in \ALG_{\mathbbst{E}_1}(\EuScript{V})\).
    \end{customcor}

As mentioned above, another important application of Theorem \ref{global thm: global additivity} concerns a more general global additivity of constructible factorization algebras. Namely, the additivity for manifolds with corners provides the foundational case from which additivity for conically smooth manifolds \cite{ayala_local_2017} can be derived. Notice that this is quite a non-trivial task, since even the basic building blocks for those manifolds are quite complicated objects. This generalization will appear in the subsequent work of the second author.

 While we do not provide all technical details here, Theorem \ref{global thm: global additivity} (or the more general Theorem \ref{thm: Global additivity max generality}), and similar results, yield an essential step in comparing two existing models for higher Morita categories: the one developed by Scheimbauer \cite{scheimbauer_factorization_2014} and that of Haugseng \cite{haugseng_higher_2017,haugseng_remarks_2023}. We note that additivity is built into Haugseng's construction, and establishing a comparison between these two frameworks has been a longstanding goal, as both are frequently employed interchangeably in the study of dualizability and TFTs. This application will be featured in subsequent work by Scheimbauer, Steffens, and Stewart.
 In recent joint work of the first author \cite{benini_cast-categorical_2025}, additivity of locally constant factorization algebras over ordinary smooth manifolds was exploited to geometrically explain the key braided monoidal structure on the $C^*$-category of cone-localized superselection sectors associated to a ground state in suitable lattice quantum systems living over $\mathbbst{Z}^2$. Results such as Theorem \ref{global thm: global additivity}, or Theorem \ref{thm: Global additivity max generality}, open the door for considering defects in this context, and related situations.

\paragraph{Related work.} As mentioned above, Theorem \ref{global thm: local additivity} is a generalization of \cite[Theorem 5.1.2.2]{lurie_higher_nodate}, which was proven by Lurie as a derived counterpart of the celebrated Dunn additivity of little cubes operads \cite{dunn_tensor_1988}. Dunn's original proof achieves additivity with respect to the ordinary Boardman--Vogt tensor product of topological operads, which is well known for not preserving weak equivalences (see \cite{moerdijk_mysterious_2023} for a short account, and \cite{fiedorowicz_additivity_2015} for a remarkable exception). Lurie's insight was to develop a homotopical version of such tensor product at the level of $\infty$-operads, and then prove a derived version of Dunn's theorem by using locally constant factorization algebras over Euclidean spaces. Later, Ginot, building on Lurie's result, established additivity for locally constant factorization algebras (valued in chain complexes) over ordinary smooth manifolds \cite[Proposition 18]{ginot_notes_2013}. Thus, Theorem \ref{global thm: global additivity} generalizes the last result in two directions: allowing more general coefficients and singularities. 

More recently, Berry proved in his thesis \cite{berry_additivity_2021} that factorization algebras satisfying a hyper-version of Weiss codescent (see \cite[Remark 2.59]{karlsson_assembly_2024}) also satisfy additivity, without imposing local-constancy or constructibility. 

Coming back to the underived Dunn additivity, Brinkmeier extended Dunn's original result to the tensor product of finitely many little cubes operads of varying dimensions in his thesis \cite{brinkmeier_operads_2001} (result which is not, a priori, a direct consequence of \cite[Theorem 2.9]{dunn_tensor_1988} because of the bad homotopical behavior of the ordinary Boardman--Vogt tensor product). More recently, Barata and Moerdijk provided a new and streamlined proof of Dunn's theorem, simplifying and clarifying the arguments presented in \cite{brinkmeier_operads_2001, dunn_tensor_1988}. Their methods have since been used by Szczesny to prove equivariant generalizations of Dunn additivity in \cite{szczesny_equivariant_2024}, and by us to prove an underived version of Theorem \ref{global thm: local additivity}. Independently and in parallel to our work, Stewart has developed related methods to establish a derived version of Szczesny's equivariant additivity in her forthcoming paper \cite{stewart_homotopical_2025}. 

Lastly, Theorem \ref{global thm: hyperdescent} also has its precursors in the literature. We have already discussed Karlsson--Scheimbauer--Walde's work, \cite{karlsson_assembly_2024}, in this direction, but we should also mention Arakawa's \cite{arakawa_universal_2025}, where he deduces descent of locally constant factorization algebras following, like us, Lurie's ``disintegration" insights. See also the related work \cite{horel_two_2025} by Horel--Krannich--Kupers. Previously, Matsuoka also studied the descent of locally constant factorization algebras in \cite{matsuoka_descent_2017}. 

\paragraph{Outline.} 
Section~\ref{sect: Preliminaries} collects notation and some preliminary results around generalized Swiss-cheese and related operads. Section~\ref{sect: Additivity for generalized swiss-cheese operads} builds on Harpaz's theory of weak $\infty$-operads and their approximations to address the derived additivity of generalized Swiss-cheese operads (Theorem \ref{thm:SwissCheeseAdditivity Reformulated}). Section~\ref{sect: Additivity for cbl Fact} contains our alternative proof for the hyperdescent of constructible factorization algebras on manifolds with corners (Theorem \ref{thm: Hyperdescent for cbl fact}), which exploits a construction that might be of independent interest (see Definition \ref{defn: Directed Integral} and the succeeding discussion), and the main theorem of this work, the additivity of constructible factorization algebras on manifolds with corners (Theorem \ref{thm: Global additivity}). At the end of Section \ref{sect: Additivity for cbl Fact}, we have also included a short digression on a variation of Theorem \ref{thm:SwissCheeseAdditivity}, namely Theorem \ref{thm: Additivity Quadrant Flags}, and its globalization. Appendices~\ref{sect:Manifolds with Corners}--\ref{sect:weak operads} collect technical results on manifolds with corners, higher algebra, and Harpaz's weak~\(\infty\)-operads, respectively.

\paragraph{Conventions.} We adopt some conventions throughout the text. Fonts like $\EuScript{C}$, $\EuScript{D}$... will be reserved for $\infty$-categories/operads, while symbols like $\mathsf{C}$, $\mathsf{D}$... often refer to strict 1-categorical notions. For example, we denote by $\Spc$ the $\infty$-category of spaces (aka $\infty$-groupoids or anima), and $\EuScript{V}$ will always stand for a symmetric monoidal $\infty$-category where algebras take values. Also, we will not distinguish between small and large $\infty$-categories/operads, but our results do not concern any subtlety that cannot be addressed with the usual methods, e.g.\ universes. We follow Haugseng's notation in  \cite{haugseng_allegedly_nodate,haugseng_algebras_2024} regarding algebras to avoid overusing the symbol $\otimes$: i.e.\ we denote by $\ALG_{\EuScript{O}}(\EuScript{V})$ the symmetric monoidal $\infty$-category of $\EuScript{O}$-algebras, and by $\Alg_{\EuScript{O}}(\EuScript{V})$ its underlying $\infty$-category (cf.\ \cite{karlsson_assembly_2024}). Lastly, we will use the notation $\EuScript{O}\brbinom{\{x_i\}_{i}}{y}$ for spaces of multimorphisms in $\EuScript{O}$ for the sake of compactness in some formulas.  

\paragraph{\bf Acknowledgements.} 
We thank C.\ Scheimbauer for initiating and supporting this collaboration, and together with E.\ Karlsson, P.\ Steffens, and T.\ Walde for many helpful discussions. In particular, we are indebted to P.\ Steffens for contributions to Lemma \ref{lem:strict pullback}, and for help in identifying the failure in Remark \ref{rem: Why directed integral is needed}, and proposing a solution that led us to Definition \ref{defn: Directed Integral}, and to T.\ Walde for his help with Lemma \ref{lem: colims over cocart fibrations}. We also thank M.\ Barata for elaborating on details of \cite{barata_additivity_2024}. Finally, we thank B.\ Conrad and A.\ Kupers for clarifying technical questions about manifolds with corners, and F.\ Cantero-Morán for suggesting \cite{palmer_homological_2021} to unravel \cite{cerf_topologie_1961}. 

\section{Background on Swiss-cheese and related operads}\label{sect: Preliminaries} 

Let us fix an $n$-dimensional manifold with corners $X$ throughout this section, and a presentable symmetric monoidal $\infty$-category $\EuScript{V}$. In this work, we are mostly interested in the $\infty$-category of constructible factorization algebras over $X$, denoted $\Fact^{\cbl}_X(\EuScript{V})$.\footnote{In order to safely work with factorization algebras, even to define them, one is required to impose some conditions on the symmetric monoidal $\infty$-category $\EuScript{V}$ where they take values. For simplicity, we will assume that $\EuScript{V}$ is presentably symmetric monoidal whenever we invoke factorization algebras.} As sketched in the introduction, this $\infty$-category admits a slick description using operads. 
\begin{defn} The operad of \emph{disjoint open subsets} in $X$, denoted $\mathsf{Disj}_X$, is the discrete operad whose colors are the open subsets in $X$, and multimorphisms are given by
$$
\mathsf{Disj}_X\brbinom{\{U_i\}_{i}}{V}=\begin{cases}
 * & \quad\text{if } \bigsqcup_i U_i\subseteq V,\\[2mm]
 \diameter & \quad\text{otherwise}.
\end{cases} 
$$
The composition rule and the action by permutations of the inputs are the only possible ones. We will denote by $\mathsf{Bsc}_X$ the full suboperad of $\mathsf{Disj}_X$ spanned by open subsets $D$ diffeomorphic to $\R^{p,q}=\R^p_{\phantom{\geq 0}}\!\!\!\!\times \R^{q}_{\geq 0}$ for some $(p,q)$. Furthermore, we will refer to the underlying categories of these operads by $\mathsf{open}(X)$ and $\mathsf{bsc}(X)$ respectively.
\end{defn}

With these simple operads at our disposal, one can define the $\infty$-category of \emph{prefactorization algebras} over $X$ with coefficients in $\EuScript{V}$ as the $\infty$-category of $\mathsf{Disj}_X$-algebras in $\EuScript{V}$, and among those select the ones we care about.
\begin{defn}\label{defn: Constructible Fact Algs} The $\infty$-category of \emph{constructible factorization algebras} over $X$ with coefficients in $\EuScript{V}$, denoted $\Fact^{\cbl}_X(\EuScript{V})$, is the full subcategory of $\Alg_{\mathsf{Disj_X}}(\EuScript{V})$ spanned by those prefactorization algebras $\mathdutchcal{F}$ satisfying:
\begin{itemize}
    \item (\emph{multiplicativity}) $\mathdutchcal{F}$ preserves cocartesian maps, i.e.\
    $$
    \bigotimes_{i\in I}\mathdutchcal{F}(U_i)\xrightarrow{\;\;\;\sim\;\;\;} \mathdutchcal{F}\big(\bigsqcup_{i\in I}U_i\big) \quad \text{and} \quad \mathbbst{I} \xrightarrow{\;\;\;\sim\;\;\;}\mathdutchcal{F}(\diameter)    
    $$
    for any finite family $\{U_i\}_{i\in I}$ of disjoint open subsets in $X$.
    \item (\emph{Weiss codescent}) For any open subset $U\subseteq X$ and any Weiss cover $\mathcal{W}$ of $U$,\footnote{We follow the conventions on (Weiss) covers from \cite{karlsson_assembly_2024}.} the comparison map
    $
    \underset{W\in\mathcal{W}}{\hocolim}\, \mathdutchcal{F}(W)\longrightarrow \mathdutchcal{F}(U)
    $
    is an equivalence.
    \item (\emph{constructibility}) $\mathdutchcal{F}$ sends any inclusion $U\subseteq V$ in $\mathsf{bsc}(X)$, with $U$ and $V$ abstractly diffeomorphic, to an equivalence.
\end{itemize}
\end{defn}

We recommend the reader to consult \cite{karlsson_assembly_2024} for an extensive and detailed discussion of these objects and the previous list of requirements. 

The central axiom for our purposes is \emph{constructibility}. To start, it is the fundamental ingredient to identify $\Fact^{\cbl}_{X}(\EuScript{V})$ for $X=\R^{p,q}$ with the $\infty$-category of algebras over the generalized Swiss-cheese operad $\mathbbst{E}_{p,q}$ defined below. However, its power goes far beyond this result. As we will see in \textsection \ref{sect: Additivity for cbl Fact}, it is also responsible for the (hyper)descent of constructible factorization algebras (Theorem \ref{thm: Hyperdescent for cbl fact}), and therefore for the derived additivity these algebras satisfy (Theorem \ref{thm: Global additivity}), which is the main result in this paper.

Let us now change gears and focus on the other main character in this story: the generalized Swiss-cheese operad and its friends, for the remainder of this section. 

\begin{defn} Let $U, V\subseteq \R^{n}$ be connected subsets of the euclidean space.
\begin{itemize}
    \item A \emph{rectilinear map} $f\colon U\to V$ is a continuous map obtained as the restriction of a map $\R^n\to \R^n$ which is a composition of a generalized dilation and a translation, i.e.\ 
    $$
    \begin{tikzcd}
    \R^n \ar[rr] && \R^n\ar[rr] && \R^n \\[-6mm]
    (x_i)_i \ar[rr, mapsto] && (\uplambda_ix_i)_i \\[-6mm]
    && x\ar[rr, mapsto] && x+v
    \end{tikzcd}
    $$
    for certain $(\uplambda_i>0)_i$ and $v\in \R^n$.
     We will denote by $\Rect(U, V)$ the space of open rectilinear maps $f\colon U\to V$ equipped with the compact-open topology. 
    \item Let $\{U_i\subset \R^n\}_{i\in I}$ be a finite collection of connected subsets of the Euclidean space. Define the space of open rectilinear maps $\Rect(\bigsqcup_{i\in I} U_i,V)$ as follows
    $$
    \Rect\big(\bigsqcup_{i\in I} U_i,V\big)=\Big\{(f_i)_i\in \prod_{i\in I}\Rect(U_i,V)\text{ s.t. } (f_i)_i\colon \bigsqcup_{i\in I}U_i\to V \text{ is injective}\Big\}.
    $$
\end{itemize}
\end{defn}

\begin{notation}\label{notat: Cubes indexed by subsets} In analogy with our convention for $\R^{p,q}=\R^p_{\phantom{\geq 0}}\!\!\!\!\times \R^q_{\geq 0}$, we will employ the notation $\square^{\,p,q}=(-1,1)^p\times [0,1)^q$. For a given subset $Q\subseteq \underline{q}=\{1,\dots,q\}$, we will also use $\square^{\,p,Q}$ to denote the product $(-1,1)^p\times \prod_{i\in \underline{q}} \updelta_i$ where
$$
\updelta_i=
\begin{cases}
        [0,1) & \quad \text{if } i\in Q, \vspace*{2mm}\\ 
        (-1,1) & \quad \text{if } i\notin Q.
    \end{cases}
$$
\end{notation}

This notation is just a particular example of our more general conventions in \textsection \ref{sect:Manifolds with Corners}. There, we discuss the poset $\Prect(X)$ of connected components of strata associated to a manifold with corners $X$, and the local types associated with its elements. To draw the connection with Notation \ref{notat: Cubes indexed by subsets}, we have:
\begin{example}\label{examp:Poset of Strata for quadrants} For the local model/quadrant $\R^{p,q}=\R^{p}_{\phantom{\geq 0}}\!\!\!\!\times \R_{\geq 0}^{q}$, the poset $\Prect(\R^{p,q})$ can be (anti-)identified with the power set $\underline{2}^{\underline{q}}$ of $\underline{q}$ ordered by inclusion. Recalling that 
$$
(\R^{p,q})_{\leq k}=\big\{(x,y)\in \R^{p,q}\;\text{ s.t. at least }q-k\text{ coordinates of }y\text{ vanish}\big\},
$$
the map $\Prect(\R^{p,q})^{\mathsf{op}}\to \underline{2}^{\underline{q}}$ sends a connected component $K\subseteq (\R^{p,q})_{\leq k}\backslash (\R^{p,q})_{\leq k-1}$ to the subset of $\underline{q}$ given by the vanishing coordinates determining $K$. 

For instance, the poset 
$$
    \Prect(\R^{1,2})=\left\{\begin{tikzcd}
        &\R\times\{0\}\times\{0\}\ar[ld]\ar[rd]\\[-0.5cm]
       \R  \times\{0\} \times\R_{> 0}&& \R\times \R_{> 0}\times\{0\}\\[-0.5cm]
       & \R\times \R_{> 0}\times \R_{>0}\ar[ru,leftarrow]\ar[lu,leftarrow]
    \end{tikzcd}\right\}
$$
is anti-equivalent to
$$
  \underline{2}^{\underline{2}}=\left\{\begin{tikzcd}
        &\{1,2\}\ar[ld,leftarrow]\ar[rd,leftarrow]\\[-0.5cm]
       \{1\} && \{2\}\\[-0.5cm]
       & \diameter\ar[lu]\ar[ru]
    \end{tikzcd}\right\}.
$$
\end{example}

\begin{defn} 
\label{defn: Swiss-Cheese Operad} The \emph{generalized Swiss-cheese operad} $\mathbbst{E}_{p,q}$, also called \emph{little} $(p,q)$\emph{-cubes operad}, is the operad in spaces with:
\begin{itemize}
    \item set of objects $\underline{2}^{\underline{q}}\;\overset{\mathsf{op}}{\simeq}\; \Prect(\R^{p,q})$ (see Example \ref{examp:Poset of Strata for quadrants}), the power set of $\underline{q}$;
    \item spaces of multimorphisms given by
    \begin{equation*}
     \mathbbst{E}_{p,q} \brbinom{\{Q_i\}_{i\in I}}{Q}=
    \begin{cases}
        \Rect \big( \bigsqcup_{i \in I} \square^{\,p, Q_i}, \square^{\,p, Q} \big)  & \quad \text{if } Q_i \subseteq Q, \vspace*{2mm}\\ 
         \diameter &  \quad \text{otherwise.}
    \end{cases}
    \end{equation*}
    Composition products and identities are the obvious ones for rectilinear maps.
\end{itemize}

We denote by $\E_{p,q}$ the $\infty$-operad obtained as the operadic coherent nerve of $\mathbbst{E}_{p,q}$.
\end{defn}

\begin{remark}
    The operadic coherent nerve turning an operad in spaces into an $\infty$-operad can be found in \cite[Definition 2.1.1.23]{lurie_higher_2009} or in \cite[Proposition 4.1.13]{harpaz_little_nodate} and the surrounding discussion. In the case of $\E_{p,q}$, the objects in the resulting \(\infty\)-category are tuples \((R_1, \dots, R_n)\), where each \(R_j\) is a subset of \(\underline{q}\). We will often abbreviate such a tuple as \(\Rbullet\), or as \(\Rbulletn\) when we want to emphasize the number of entries. Also, the mapping space from \(\Rbulletn\) to \(\Qbulletm\) is given by
    \[
     \coprod_{\upalpha: \langle n \rangle \to \langle m \rangle} \quad \prod_{j\in \langle m\rangle^{\circ}} \Rect \big( \bigsqcup_{i \in \upalpha^{-1}(j)} \square^{\,p, R_i}, \square^{\,p, Q_j} \big),
    \]
    when \(R_i \subseteq Q_{\upalpha(i)}\) for all \(i \in \langle n\rangle^{\circ}\), or empty otherwise.
\end{remark}

To better bridge the abstract machinery of the higher category theory with our geometric setting, we introduce a simpler $\infty$-operad $\Cube{p}{q}$ that approximates $\E_{p,q}$ in a suitable sense.

\begin{defn}\label{defn: Cube p,q operads}
    Let \(\CubeOrd{p}{q}\) denote the full suboperad of \(\mathsf{Disj}_{\square^{\,p,q}}\) spanned by the non-empty open sub-cubes of \(\square^{\,p,q}\), i.e.\ the images of open rectilinear embeddings. We will denote by $\Cube{p}{q}$ its associated $\infty$-operad.
\end{defn}

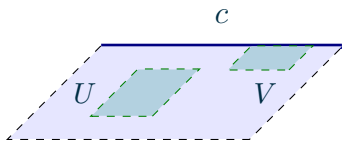
\begin{figure}[htp]
    \centering
    \input{figure1}
    \caption{Image of a multimorphism $c\in\mathbbst{E}_{1,1}\brbinom{\{\scalebox{0.8}{$\diameter$} , \{1\}\}}{\{1\}}$. The resulting collection of open subsets of $\square^{\,1,1}$ can also be viewed as an object in $\Cube{1}{1}$, denoted $(U,V)$.}
    \label{fig: figure 1}
    \end{figure}

There is a natural map of operads \(\upphi\colon\CubeOrd{p}{q} \xrightarrow{\quad} \mathbbst{E}_{p,q}\), and hence a map between the associated $\infty$-operads. For each sub-cube in \(\CubeOrd{p}{q}\), there is a unique object \(\square^{\,p,Q}\) in \(\mathbbst{E}_{p,q}\) such that the sub-cube arises as the image of \(\square^{\,p,Q}\) under a unique (open) rectilinear map. We refer to both \(\square^{\,p,Q}\) and \(Q\) as the \emph{type} of the sub-cube.

\begin{notation}
    Let \(U \subseteq V\) be sub-cubes in \(\CubeOrd{p}{q}\) with the same type, then we will denote this by writing \(U \overset{\sim}{\subseteq} V\). 
\end{notation}

\begin{remark}\label{rem:Mapping spaces in E_pq act in a compact form}
    We will focus on the mapping spaces in \(\Eactive_{p,q}\). To prepare for the proofs that follow, we rewrite these mapping spaces in a more convenient form. 

    Let \(X\) denote the disjoint union of a family of sub-cubes \(\{\square^{\,p,Q_j}\}_{j\in \langle m\rangle^{\circ}}\), seen as a subspace of \(\langle m \rangle^{\circ} \times \square^{\,p,q}\). Then, by connectivity considerations, we can absorb the coproduct over active maps  \(\upalpha: \langle n \rangle \to \langle m \rangle\) in the definition of mapping spaces in $\Eactive_{p,q} $ into a ``single" space of rectilinear embeddings; i.e.\ there is a canonical identification
    \[
\coprod_{\substack{\text{active maps} \\ \upalpha: \langle n \rangle \to \langle m \rangle}} \quad \prod_{j\in\langle m\rangle^{\circ}} \Rect \big( \bigsqcup_{i \in \upalpha^{-1}(j)} \square^{\,p,R_i}, \square^{\,p,Q_j} \big)
    \cong
    \Rect \big( \bigsqcup_{i\in \langle n\rangle^{\circ}} \square^{\,p,R_i}, X \big).
    \]
\end{remark}

\begin{figure}[htp]
    \centering
    \input{figure2}
    \caption{Image of a map in 
    \(
    \Rect \left( \bigsqcup_{i\in \langle 3\rangle^{\circ}} \square^{\,p, R_i}, X \right)
    \)
    with \((R_1 , R_2, R_3) = (\{1\}, \diameter, \diameter)\) and \((Q_1, Q_2, Q_3) = (\{1\}, \diameter, \{1\})\), both seen as objects of \(\E_{1,1}\).}
    \label{fig: figure 2}
\end{figure}
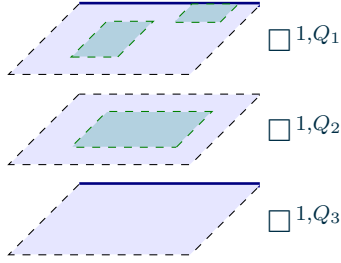

\begin{remark}
\label{remark: cubes are nice} In the sequel, we will use the following simple property of local types $\square^{Z}$
    and strata \(X_{[Z]}\) for \(X=\square^{\,p,q}\). Note that for any \(x \in X\), there exists a \emph{unique} local type \(\square^{\,p,Q}\), admitting a rectilinear embedding \(\upiota: \square^{\,p,Q} \hookrightarrow X\) such that \(x \in \upiota(\square^{\,p,Q}_{[Q]})\). Recall that the open condition we are imposing on rectilinear maps imposes restrictions on the translation components of those. In particular, an open rectilinear embedding $\upiota\colon \square^{\,p,Q_j}\hookrightarrow \square^{\,p,Q}$ must satisfy $\upiota(0)\in \square^{\,p,Q}_{[Q_j]}$.
    
    As a consequence of this, if \(\upiota_1: \square^{\,p,Q_1} \hookrightarrow X\) and \(\upiota_2: \square^{\,p,Q_2} \hookrightarrow X\) are two rectilinear embeddings in \(\E_{p,q}\) with a non-empty intersection  \( \upiota_1(\square^{\,p,Q_1}_{[Q_1]})\cap \upiota_2(\square^{\,p,Q_2}_{[Q_2]})\), then \(Q_1=Q_2\).
\end{remark}

\begin{defn}\label{defn:Little discs operad in X} Let $X$ be a manifold with corners. The \emph{little} $X$\emph{-discs operad} $\gE_{X}$ is the colored operad in spaces defined as follows:  
    \begin{itemize}
    \item its colors are pairs $\Drect=(\square^{\,p,q}, \upiota)$ where $\upiota\colon \square^{\,p,q}\hookrightarrow X$ is an open embedding;
    \item its spaces of multimorphisms are given by the homotopy pullbacks
    $$
     \hspace*{-1cm} \begin{tikzcd}
         \gE_{X} \brbinom{\{\Drect_j\}_{j\in I}}{\Drect}\ar[r]\ar[d]\ar[rd, phantom, "\overset{\mathsf{h}\;\;}{\lrcorner}" description] &   \prod_{j\in I} \Emb_{/X}(\square^{\, p_j,q_j},\square^{\, p,q})\ar[d] \ar[rd, phantom, "\overset{\mathsf{h}\;\;}{\lrcorner}" description] \ar[r] & * \ar[d,"(\upiota_j)_j"] \\
         \Emb\big( {\bigsqcup}_{j \in I} \square^{\,p_j,q_j}, \square^{\, p,q} \big)
          \ar[r] & \prod_{j\in I}\Emb(\square^{\, p_j,q_j},\square^{\, p,q}) \ar[r,"(\upiota_*)_j"'] & \prod_{j\in I}\Emb(\square^{\, p_j,q_j},X)
    \end{tikzcd}.
    $$
    Composition products and identities are the obvious ones for embeddings.\footnote{The strict composition and units on these homotopy pullbacks come from the use of appropriate point-set models; see \cite[\textsection 6]{horel_factorization_2017}. In the sequel, we will not need to manipulate these models explicitly.}
\end{itemize}
\end{defn}
\begin{remark} Note that Definition~\ref{defn:Little discs operad in X} imposes certain restrictions on the colors of $\gE_X$ for its multimorphism spaces to be non-empty. For instance, letting $\{\Drect_j\}_{j\in I}$ and $\Drect$ be a finite list and a single color of $\gE_X$ respectively, with $\upiota_j(\square^{\,p_j,q_j})\subseteq X_{[Z_j]}$ for all $j\in I$ and $\upiota(\square^{\,p,q})\subseteq X_{[Z]}$ where $Z_j,Z\in \Prect(X)$ are connected components of strata in $X$, we have $\gE_X\brbinom{\{\Drect_j\}_{j\in I}}{\Drect} =\diameter$  if $Z\leq Z_j$ does \emph{not} hold for some $j$. 
\end{remark}
\begin{remark}\label{rk:Comparing the two little (p,q)-operads}
    The operad $\gE_X$ can be seen as obtained from an \emph{operadic slice construction} out of a more elementary operad $\cgE$ with color-set $\{\square^{\,p,q}\}_{p+q=n}$, and as spaces of multimorphisms, embedding spaces $\Emb\big(\bigsqcup_{i}\square^{\,p_i,q_i},\square^{\, p,q}\big)$ (cf.\  Notation \ref{notat:Topological category of manifolds with corners}). One way to phrase this operadic slice construction is by observing that $\gE^{\otimes}_X$ can be identified with the homotopy pullback
    $$
      \begin{tikzcd}
         (\cgE)_X^{\otimes} \ar[rr]\ar[d] \ar[rrd, phantom, "\overset{\mathsf{h}\;\;}{\lrcorner}" description]
         &&  ({\cMfld}_{/X})^{\amalg} \ar[d] \\
         \cgE^{\otimes}
         \ar[rr]   && \cMfld^{\amalg}
    \end{tikzcd},
    $$
    where ${\cMfld}_{/X}$ denotes the slice $\infty$-category associated with $X\in \cMfld$, and $\EuScript{C}^{\amalg}$ is the cocartesian $\infty$-operad associated with the $\infty$-category $\EuScript{C}$ (see \cite[Example 4.1.12]{harpaz_little_nodate}).
\end{remark}

\begin{notation}\label{notat: Configuration spaces}
    Let \( X \) be a manifold with corners, and let \( I \) be a finite set. Suppose we are given a function \( Z \colon I \to \Prect(X) \). Since we will be working with configuration spaces of \( X \), we introduce the following version, which will be useful in what follows:
    \begin{equation}
    \Conf_Z(X) \coloneqq \prod_{K\in \Prect(X)} \Conf_{Z^{-1}\{K\}}\big(X_{[K]}\big),
    \label{config}
    \end{equation}
    where \(\Conf_{T}(M) \) denotes the (ordered) configuration space of 
    points in \(M\) parametrized by a finite set $T$. \\
    Another equivalent way to define this space is as follows. Let \( J \) be the quotient of \(I\) with respect to \(i \sim j\) iff \(Z(i)=Z(j)\) and \(I_t\) denotes the fiber of the canonical projection $I \to J$ over $t$. Then, 
    \begin{equation}
    \Conf_Z(X) \coloneqq \prod_{t\in J} \Conf_{I_t}\big(X_{[Z(t)]}\big).
    \label{config2}
    \end{equation}
    
\end{notation}

\begin{prop}\label{prop: Coincidence of little (p,q)-disc operads} 
    The canonical map $\mathbbst{E}_{p,q}\to\gE_{\mathbbst{R}^{p,q}}$ is an equivalence of operads.
\end{prop}

\begin{proof} Let us focus on the equivalence at the level of spaces of multimorphisms, since essential surjectivity up to equivalence is just a consequence of Proposition \ref{prop: Scanning with one basic}.

By definition of $\mathbbst{E}_{p,q}$ and $\gE_{\mathbbst{R}^{p,q}}$, we are done if the square
$$
  \begin{tikzcd}
         \Rect \big( \bigsqcup_{j \in I} \square^{\,p, Q_j}, \square^{\,p, Q} \big)\ar[rr]\ar[d] 
         && \prod_{j\in I} \Emb_{/\mathbbst{R}^{p,q}}(\square^{\, p,Q_j},\square^{\, p,Q}) \ar[d] \\
        \Emb\big( {\bigsqcup}_{j \in I} \square^{\,p,Q_j}, \square^{\, p,Q} \big) \ar[rr]  && \prod_{j\in I}\Emb(\square^{\, p,Q_j },\square^{\, p, Q }) 
    \end{tikzcd}
$$
is a homotopy pullback. Note that $\Emb(\square^{\,p,Q_j},\square^{\, p,Q})$ refers to the connected component of $\Emb\big((-1,1)^{\,p+q-\vert Q_j\vert}\times [0,1)^{\vert Q_j\vert},(-1,1)^{\,p+q-\vert Q\vert}\times [0,1)^{\vert Q\vert}\big)$ where the canonical inclusion $\square^{\,p,Q_j}\hookrightarrow\square^{\,p,Q}$ lives. We can expand this square to form
$$
  \begin{tikzcd}
         \Rect \big( \bigsqcup_{j \in I} \square^{\,p, Q_j}, \square^{\,p, Q} \big)\ar[rr]\ar[d] 
         && \prod_{j\in I} \Emb_{/\mathbbst{R}^{p,q}}(\square^{\, p,Q_j},\square^{\, p,Q}) \ar[d] \\
         \Emb\big( {\bigsqcup}_{j \in I} \square^{\,p,Q_j}, \square^{\, p,Q} \big)
         \ar[rr] \ar[d,"\wr"']  && \prod_{j\in I}\Emb(\square^{\, p,Q_j},\square^{\, p, Q}) \ar[d,"\wr"] \\
        \prod_{t\in J} \Fr^{\,i}_{[Q_t]}\big( \Conf_{I_t} \big(\square^{\, p, Q}_{[Q_t]}\big) \big)
         \ar[rr]\ar[d, "\uppi"']\ar[rrd, phantom, "\overset{\mathsf{h}\;\;}{\lrcorner}" description]  && \prod_{j\in I}\Fr^{\,i}_{[Q_j]}\big(\square^{\, p, Q}_{[Q_j]}\big)  \ar[d,"\uppi"]\\
         \prod_{t\in J} \Conf_{I_t}\big(\square^{\, p,Q}_{[Q_t]}\big) \ar[rr] && \prod_{j\in I}\square^{\, p, Q}_{[Q_j]}
    \end{tikzcd}.
$$

The lower square is a homotopy pullback since the inner frames yield a principal bundle when restricted to points in the same stratum (see \textsection \ref{sect:Manifolds with Corners}), and the square in the middle is a homotopy pullback since it is made of equivalences by Proposition \ref{prop: Scanning with several basics}. Therefore, the top square is a homotopy pullback if and only if so is the composed rectangle. Since evaluation at centers $\Rect\big( \bigsqcup_{j \in I} \square^{\,p, Q_j}, \square^{\,p, Q} \big)\to \prod_{t\in J} \Conf_{I_t}\big(\square^{\, p, Q}_{[Q_t]}\big)$ is a weak homotopy equivalence (see Lemma \ref{lem:rect-conf} below), we are reduced to check that the analogous map $\Emb_{/\mathbbst{R}^{p,q}}(\square^{\, p, Q_j},\square^{\, p, Q})\to \square^{\,p ,Q}_{[Q_j]}$ is also a weak homotopy equivalence. Plugging-in the definition of $\Emb_{/\mathbbst{R}^{p,q}}$ and arguing as above, we find a diagram
\begin{equation}\label{eqt:hopbs showing space of relative embeddings is contractible}
  \begin{tikzcd}
         \Emb_{/\mathbbst{R}^{p,q}} \big( \square^{\,p, Q_j}, \square^{\,p, Q} \big)\ar[rr]\ar[d] \ar[rrd, phantom, "\overset{\mathsf{h}\;\;}{\lrcorner}" description]
         &&  * \ar[d,"\square^{\,p,Q_j}\hookrightarrow \mathbbst{R}^{p,q}"] \\
         \Emb\big( \square^{\,p, Q_j}, \square^{\, p,Q} \big)
         \ar[rr] \ar[d,"\mathsf{ev_0}"'] \ar[rrd, phantom, "\overset{\mathsf{h}\;\;}{\lrcorner}" description]  && \Emb\big( \square^{\,p,Q_j}, \mathbbst{R}^{p,q} \big) \ar[d,"\mathsf{ev}_0"]\\
   \square^{\, p, Q}_{[Q_j]} \ar[rr] && \mathbbst{R}^{p,q}_{[Q_j]}
    \end{tikzcd},
\end{equation}
which shows that the composite left vertical map is a weak homotopy equivalence since the right composition is so.
\end{proof}

We close this section by introducing a piece of notation that will prove useful for bookkeeping purposes in the sequel (notice the minor, but important, difference with Notation \ref{notat: Configuration spaces}), and by presenting the postponed auxiliary Lemma \ref{lem:rect-conf}, which generalizes \cite[Lemma 5.1.10]{harpaz_little_nodate} to the corner setting.

\begin{notation}
   Let \(\Rbulletn\) and \(\Qbulletm\) be objects in \(\E_{p,q}\). Recall that such an object can be described as a function
    \(
    \Rbullet: \langle n \rangle^{\circ} \to \Prect(\square^{\,p,q}) \cong (\underline{2}^{\underline{q}})^{\op},
    \)
    where \(\Prect(\square^{\,p,q})\) denotes the (poset of) connected components of strata within \(\square^{\,p,q}\). Set $T=\bigsqcup_{j\in \langle m\rangle^{\circ}}\square^{\,p,Q_j}$, seen as a subspace of $\langle m\rangle^{\circ}\times \square^{\,p,q}$. By an abuse of notation, set $T_{[R_i]}=T\cap (\langle m\rangle^{\circ}\times\square^{\,p,q}_{[R_i]})$, and consider the configuration space
    \begin{equation*}
        \Conf_{\Rbullet}(T) := \underset{t \in J_{R}}{\prod} \Conf_{I_t} (T_{[R_t]}),
    \end{equation*}
    where $J_{R}$ denotes the quotient of $\langle n\rangle^{\circ}$ by the equivalence relation $i\sim i'$ iff $R_i=R_{i'}$ and $I_t$ denotes the fiber of $\langle n\rangle^{\circ}\to J_R$ over $t$.
\end{notation}

\begin{lemma}
\label{lem:rect-conf}
    Let \(\Rbulletn\) and \(\Qbulletm\) be objects in \(\E_{p,q}\). Suppose \(T\) denotes the coproduct of the sub-cubes \(\{\square^{\,p,Q_j}\}_{j \in \langle m \rangle^{\circ}}\), so that \(T \subseteq \langle m\rangle^{\circ} \times \square^{\,p,q}\). Then, evaluation at $0\in\square^{\,p,R_i}$ for all $i\in \langle n\rangle^{\circ}$ induces a weak homotopy equivalence
      \begin{align*}
        \mathsf{ev}_0\colon\Rect \big( \bigsqcup_{i\in\langle n\rangle^{\circ}} \square^{\,p,R_i} , T \big)  \xrightarrow{\;\;\;\sim\;\;\;}
         \Conf_{\Rbullet} (T)
        \end{align*} 
\end{lemma}

\begin{proof} We construct a homotopy inverse 
    $\mathsf{s} : \Conf_{\Rbullet}(T) \xrightarrow{\quad} \Rect \big( \bigsqcup_{i\in\langle n\rangle^{\circ}} \square^{\,p,R_i} , T \big)$
    as follows. First, consider a \emph{distance function} on $T$, which is only possibly different from $2$ between points $x,y\in T$ lying in the same connected component, and in that case is
    $$
    d(x,y):=\begin{cases} \mathsf{min}\left\{|\uppi_k(x)-\uppi_k(y)|  \text{ for }1\leq k\leq p+q\text{ with }\uppi_k(x)\neq \uppi_k(y)\right\} & \text{if }x\neq y,\\[2mm]
    0 & \text{otherwise},    
    \end{cases}
    $$
    where $\uppi_k$ denotes the projection into the $k^{\text{th}}$-coordinate in $\square^{\,p,q}$. As in \cite{harpaz_little_nodate}, for a point $ \underline{x}=\{x_i\}_{i\in \langle n\rangle^{\circ}}$ in $\Conf_{\Rbullet}(T)$, let us define
    \[
    \upvarepsilon'(\underline{x}) \;=\; 
    \mathsf{min}\left\{d(x_i,x_j):\;i,j\in \langle n\rangle^{\circ},\;\;i\neq j\right\}
    \]
    as the minimum distance among points conforming to the configuration $\underline{x}$.
    
    Next, we measure the minimum distance from the relevant boundary of the cubes to the points in the configuration. Since some of the points may lie in strata of \(\square^{\,p, Q_k}\), we measure the distance to the boundary of the stratum in which each point lies, taken in $\langle m\rangle^{\circ}\times[-1,1]^{p}\times[0,1]^{q}$. For each point $x_i$ in $\underline{x}$, \(R_i\) denotes its local type. 
    Now define
    \[
    \upvarepsilon''(\underline{x}) \;=\; \mathsf{min}\left\{ d\big(x_i,\,\partial\,  \square^{\,p,Q_{k_i}}_{[R_i]}\big):\;i\in \langle n\rangle^{\circ}\right\},
    \]
    where the distance is induced in the obvious sense from the one above.  
    We then set
    \[
    \upvarepsilon(\underline{x}) \;=\; \tfrac{1}{2}\,\mathsf{min}\{\upvarepsilon'(\underline{x}),\,\upvarepsilon''(\underline{x})\}.
    \]

    For each point \(x_i\) in the configuration, we build a cube \(\square_{x_i}\) of type \(R_i\), whose side length is \(\upvarepsilon(\underline{x})\). The cube \(\square_{x_i}\) is placed in the appropriate \(\square^{\,p,Q_{k_i}}\) so that \(x_i\) is the center of \( (\square_{x_i})_{[R_i]}\). In other words, \(\mathsf{ev}_0 (\square_{x_i})=x_i\). Then the open rectilinear map \(\mathsf{s}(\underline{x})\) is determined by the way the cubes are constructed, and \( \mathsf{ev}_0 \circ \mathsf{s} = \id_{\Conf}\). 
    The homotopy \(\mathsf{s} \circ \mathsf{ev}_0 \simeq \id_{\Rect}\) can be given by explicit expressions on each connected component of \(T\). We leave it to the reader to write them down.
\end{proof}

\section{Additivity for generalized Swiss-cheese operads}
\label{sect: Additivity for generalized swiss-cheese operads}

This section presents a generalization of Harpaz's proof of derived Dunn additivity for the operads of little 
$n$-cubes, originally based on Lurie's approach. For the sake of completeness, the relevant components of Harpaz’s argument are included in \textsection \ref{sect:weak operads}. In the main body of the text, we identify the specific points at which modifications are required to extend the proof to the broader setting under consideration. This will allow readers to follow the progression of ideas within this context.

Let us begin by recalling that there are two quite natural maps of operads
$$
\begin{tikzcd}
    \mathbbst{E}_{p,q}\ar[r,"\upmu_1"] & \mathbbst{E}_{p+r,q+t} \ar[r,leftarrow,"\upmu_2"] & \mathbbst{E}_{r,t}
\end{tikzcd}
$$
which satisfy an interesting \emph{interchange law} (cf.\ \cite{brinkmeier_operads_2001,dunn_tensor_1988}) that we specify below. Ultimately, both maps boil down to the canonical identification
\begin{align*}
    \square^{\,p,q}\times \square^{\,r,t} &=(-1,1)^p\times [0,1)^q\times (-1,1)^r\times [0,1)^t\\
    &\cong (-1,1)^p\times (-1,1)^r\times [0,1)^q\times [0,1)^t=\square^{\, p+r,q+t},
\end{align*}
but let us be more precise here. On objects, they are given by:
\begin{itemize}
    \item $\upmu_1\colon \square^{\,p,Q}\longmapsto \square^{\,p,Q}\times \square^{\,r,t}\cong \square^{\, p+r,Q}$, and
    \item $\upmu_2\colon \square^{\,r,T}\longmapsto \square^{\,p,q}\times \square^{\, r,T}\cong \square^{\,p+r,q+T}$,
\end{itemize}
where $Q$ on the right-hand side denotes the image of the homonymous subset along $\underline{q}\hookrightarrow \underline{q+t}$, $i\mapsto i$, while $q+T$ denotes the image of $T$ along $\underline{t}\hookrightarrow \underline{q+t}$, $i\mapsto q+i$. On multimorphisms, $\upmu_2$ sends the open rectilinear embedding $\upiota\colon \bigsqcup_{j\in I}\square^{\,r, T_j}\hookrightarrow \square^{\,r,T}$ to 
$$
\upmu_2(\upiota)\colon \bigsqcup_{j\in I}\square^{\, p+r,q+T_j} \cong \square^{\,p,q}\times \bigsqcup_{j\in I}\square^{\, r,T_j}\xhookrightarrow{\quad\square^{\, p,q}\times \upiota\quad} \square^{\, p,q}\times \square^{\,r,T}\cong \square^{\,p+r,q+T},
$$
and similarly for $\upmu_1$.

To formulate the interchange law they satisfy, the picture is slightly more complicated than the uncolored case \cite[\textsection 7]{brinkmeier_operads_2001}, since we need to consider analogues of $\upmu_1$ and $\upmu_2$ for each color $T\in \mathbbst{E}_{r,t}$ and $Q\in \mathbbst{E}_{p,q}$ respectively (the previously defined maps been those corresponding to $\underline{t}$ and $\underline{q}$); see Figure \ref{fig: figure 5}. What they are is evident from the discussion above. Once we have them, the interchange law  asserts that
$$
\begin{tikzcd}
   &[-34mm] \mathbbst{E}_{p,q}\brbinom{\{Q_i\}_{i\in I}}{Q} \!\times\! \mathbbst{E}_{r,t}\brbinom{\{T_j\}_{j\in J}}{T} \ar[rdd, "\id\times \Delta"]\ar[ldd,"(\id\times \Delta)\cdot (\mathsf{switch})"'] &[-34mm] \\\\
\mathbbst{E}_{r,t}\brbinom{\{T_j\}_{j\in J}}{T} \!\times\! \mathbbst{E}_{p,q}\brbinom{\{Q_i\}_{i\in I}}{Q}^{\times \vert J\vert} \ar[dd,"\upmu_{Q}\times \prod_j\upmu_{T_j}"'] & & \mathbbst{E}_{p,q}\brbinom{\{Q_i\}_{i\in I}}{Q} \!\times\! \mathbbst{E}_{r,t}\brbinom{\{T_j\}_{j\in J}}{T}^{\times \vert I \vert} \ar[dd,"\upmu_{T}\times \prod_{i}\upmu_{Q_i}"]\\\\
\mathbbst{E}_{p+r,q+t}\brbinom{\{Q\times T_j\}_{j\in J}}{Q\times T} \!\times\! \prod_{j}\mathbbst{E}_{p+r,q+t}\brbinom{\{Q_i\times T_j\}_{i\in I}}{Q\times T_j}\ar[rdd,"\mathsf{compose}"'] & & \mathbbst{E}_{p+r,q+t}\brbinom{\{Q_i\times T\}_{i\in I}}{Q\times T} \!\times\!\prod_i \mathbbst{E}_{p+r,q+t}\brbinom{\{Q_i\times T_j\}_{j\in J}}{Q_i\times T} \ar[ldd,"\uptau_{I,J}\cdot\,\mathsf{compose}"]\\\\
& \mathbbst{E}_{p+r,q+t}\brbinom{\{Q_i\times T_j\}_{(i,j)}}{Q\times T}
\end{tikzcd}
$$
commutes for all possible choices, where we have denoted by $\Delta$ the appropriate diagonal maps, by $Q\times T$, the color of $\E_{p+r,q+t}$ corresponding to $\square^{\,p, Q}\times \square^{\,r, T}$, and by $\uptau_{I,J}$, the obvious bijection $I\times J\cong J\times I$.

\begin{figure}[htp]
    \centering
    \input{figure5}
    \caption{Open subsets associated with the image of the open rectilinear embedding $c$ from Figure \ref{fig: figure 1} along $\upmu_{\scalebox{0.8}{$\diameter$}},\upmu_{\{1\}}\colon \mathbbst{E}_{1,1}\longrightarrow \mathbbst{E}_{1,2}$; both maps use $\square^{\,1,2}\cong \square^{\, 1,1}\times \square^{0,1}$.}
    \label{fig: figure 5}
 \end{figure}
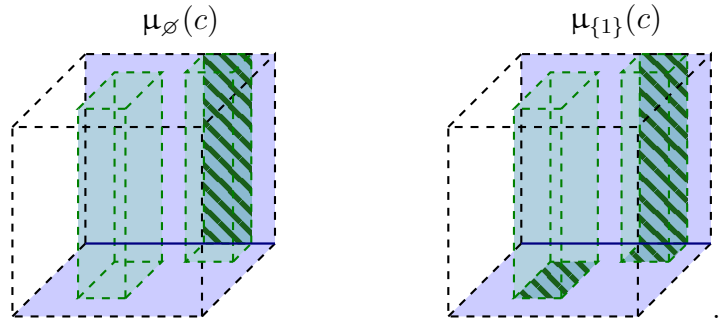

All this is precisely what one needs to build a bifunctor $\upmu\colon \E_{p,q} \times \E_{r,t} \xrightarrow{\quad}\E_{p+r,q+t}$ (see \textsection \ref{sect: Higher categorical tools}). More concretely, $\upmu$ is given on objects by
$$
\big((\langle n\rangle, Q_{\bullet}),(\langle m\rangle,T_{\bullet})\big) \xmapsto{\;\;\quad\;\;} (\langle n\rangle \wedge \langle m\rangle, Q_{\bullet}\times T_{\bullet}), 
$$
and on morphisms by 
$$
\begin{tikzcd}
   \displaystyle\left( \Big(\upalpha,\big(\upiota_j\colon \bigsqcup_{i\in \upalpha^{-1}(j)}\square^{\, p,Q_i}\hookrightarrow \square^{\, p,Q_j}\big)_{j}\Big),\Big(\upbeta,\big(\upiota'_k\colon \bigsqcup_{\ell\in \upbeta^{-1}(k)}\square^{\, r,T_\ell}\hookrightarrow \square^{\, r,T_k}\big)_{k}\Big)\right) \ar[d, mapsto]\\
  \displaystyle \Big(\upalpha\wedge \upbeta, \big(\bigsqcup_{(i,\ell)\in (\upalpha\wedge\upbeta)^{-1}(j,k)}\square^{\, p,Q_i}\times \square^{\, r, T_{\ell}}\xhookrightarrow{}\square^{\,p,Q_j}\times \square^{\, r,T_k} \big)_{(j,k)}\Big)
\end{tikzcd},
$$
where the $(j,k)$-labeled open rectilinear embedding above is explicitly defined as the operadic composition  $\upmu_{Q_j}(\upiota'_{k})\circ \big(\upmu_{T_{\ell}}(\upiota_{j})\big)_{\ell\in \upbeta^{-1}(k)}=\upmu_{T_k}(\upiota_{j})\circ \big(\upmu_{Q_{i}}(\upiota'_{k})\big)_{i\in \upalpha^{-1}(j)}$.


 \input{figure6}

The main goal of this section is to prove:
\begin{thm}
    \label{thm:SwissCheeseAdditivity}
    The bifunctor \(\upmu\colon \E_{p,q} \times \E_{r,t} \xrightarrow{\quad}\E_{p+r,q+t}\) exhibits \(\E_{p+r,q+t}\) as a tensor product (see Definition \ref{defn: Exhibits a tensor product}).
\end{thm}

Of course, this statement admits the following reformulation.
\begin{thm}\label{thm:SwissCheeseAdditivity Reformulated} For any symmetric monoidal \(\infty\)-category $\EuScript{V}$, there is an equivalence of \(\infty\)-categories
    $
    \Alg_{\mathbbst{E}_{p,q}}(\ALG_{\mathbbst{E}_{r,t}}(\EuScript{V}))\simeq \Alg_{\mathbbst{E}_{p+r,q+t}}(\EuScript{V}) 
    $
     induced by the maps of topological operads $(\upmu_T\colon\mathbbst{E}_{p,q}\to \mathbbst{E}_{p+r,q+t} \leftarrow \mathbbst{E}_{r,t}:\!\upmu_Q)_{T,Q}$. Furthermore, one can leverage this to a symmetric monoidal equivalence
    $$
    \ALG_{\mathbbst{E}_{p,q}}\big(\ALG_{\mathbbst{E}_{r,t}}(\EuScript{V})\big)\simeq \ALG_{\mathbbst{E}_{p+r,q+t}}(\EuScript{V}). 
    $$
\end{thm}

Before embarking into the proof of the derived additivity theorem, let us stress that we will extensively use the weak $\infty$-operad structure on  \(\E_{p,q}\) and \(\Cube{p}{q}\), obtained in the standard way from their $\infty$-operad structure (see Remark \ref{ex:weak structure on infinity operads}).

\begin{remarks}
\label{rmks:WeakStructure E and Cube}
    The basics of the weak $\infty$-operad \(\E_{p,q}\) are single colors over \(\langle 1 \rangle\). The inert maps between basics are open rectilinear embeddings between the same types. On the other hand, the basics of the weak $\infty$-operad \(\Cube{p}{q}\) are single sub-cubes over \(\langle 1 \rangle\), and the only inert maps between basics are the identities.
\end{remarks}

The proof of Theorem \ref{thm:SwissCheeseAdditivity} will require some preparations, but the strategy is as follows. First, due to Propositions \ref{prop:weak operad 1} and \ref{thm:wreathprod}, it is enough to show that the bifunctor $\upmu$ induces an equivalence of $\infty$-categories 
\[
\Mon_{\E_{p+r,q+t}}(\Spc) \xrightarrow{\;\;\sim\;\;} \Mon_{\E_{p,q} \times \E_{r,t} }(\Spc) \simeq \Mon_{\E_{p,q}}(\Mon_{\E_{r,t} }(\Spc)).
\]
To see that this is the case, one replaces $\E_{p',q'}$ with the discrete $\Cube{p'}{q'}$, thanks to Theorems \ref{thm:WeakApproxofE} and \ref{thm:WreathProdApprox} below, and applies Lemma \ref{lem:Harpaz weak approx}. Thus, we focus now on proving the cited theorems. Actually, the proof of Theorem \ref{thm:WeakApproxofE}, which is divided into two lemmas, should be seen as a road map for that of Theorem \ref{thm:WreathProdApprox}.

\begin{thm}
\label{thm:WeakApproxofE}
    The map \(\upphi: \Cube{p}{q} \xrightarrow{\;\quad\;} \E_{p,q}\) is a weak approximation. 
\end{thm}
\begin{proof} The first condition to be a weak approximation (see Definition \ref{defn: weak and strong approximation}) is proven in Lemma \ref{lem:First condition WeakAprroxofE}, while the second is checked in Lemma \ref{lem:Second condition WeakAprroxofE}. Notice that, due to Remark \ref{rem:WeakApproxReduction}, we are able to reduce the second condition to the case of an object \(U\) living over \(\langle 1 \rangle \) in \(\Fin_*\). 
\end{proof}

\begin{lemma}\label{lem:First condition WeakAprroxofE} The functor 
$\upphi^{-1} (\E_{p,q})_{\langle 1 \rangle}^{\inert} \xrightarrow{\;\quad\;} (\E_{p,q})_{\langle 1 \rangle}^{\inert}$ exhibits an $\infty$-localization.
\end{lemma}
\begin{proof}

  Observe that \(( \E_{p,q} )_{\langle 1 \rangle}^{\inert}\) is an $\infty$-groupoid. Consequently, the localization statement reduces to showing that \(( \E_{p+r,q+t} )_{\langle 1 \rangle}^{\inert}\) is a classifying space of \(\upphi^{-1} (\E_{p,q})_{\langle 1 \rangle}^{\inert}\). Equivalently, by Proposition \ref{prop:localization}, we will show that the homotopy fibers of the map in the statement are contractible. Let \(Q\) be any object in \(( \E_{p,q} )_{\langle 1 \rangle}^{\inert}\). Since this map satisfies the assumptions of Lemma \ref{lem:strict pullback}, we conclude that the homotopy fiber over $Q$ is equivalent to an ordinary category with sub-cubes \(U\) of type \(Q\) as objects, and morphisms given by inclusions. This category has a terminal object, namely the full sub-cube \(\square^{\,p, Q}\); therefore, it is contractible.
\end{proof}

\begin{lemma}\label{lem:Second condition WeakAprroxofE}
    The homotopy fibers of the map 
    \(
    \big(\CubeActive{p}{q}\big)_{/ U} \xrightarrow{\;\quad\;}  \big(\Eactive_{p,q}\big)_{/Q}
    \)
    are contractible for any sub-cube \(U\subseteq \square^{\,p,q}\) of type \(Q\in  \mathbbst{E}_{p,q} \).
\end{lemma}
\begin{proof}
     The following argument will be repeated multiple times in the upcoming proofs. Consider the following composition
    \[
    \big(\CubeActive{p}{q}\big)_{/ U} \xrightarrow{\;\;\quad\;\;} \big(\Eactive_{p,q}\big)_{/Q} \xrightarrow{\;\;\quad\;\;} \Eactive_{p,q}.
    \]
    Then, it suffices to show that the induced map on homotopy fibers
    \begin{equation}\label{eqt:Map between hofibers in WeakApproxofE}
    \begin{tikzcd}
         \big(\CubeActive{p}{q}\big)_{/ U,\, |  R_{\bullet}}  \ar[r]
        &
        \big(\Eactive_{p,q}\big)_{/Q,\, |  R_{\bullet}} 
    \end{tikzcd}
    \end{equation}
    is a weak homotopy equivalence for any object \(\{ R_{\bullet} \}_{\langle l \rangle}\in \Eactive_{p,q}\). Since the second map $(\Eactive_{p,q})_{/Q}\to \Eactive_{p,q}$ in the composition is a right fibration (see Proposition \ref{prop: Categorical triangle to analyze fibers}),
     the right hand side is equivalent to \( {\mathbbst{E}_{p,q}} \brbinom{{\{R_\bullet\}}}{ Q}\), and hence to a configuration space by Lemma \ref{lem:rect-conf}. Namely, 
$$
 {\mathbbst{E}_{p,q}} \brbinom{{\{R_\bullet\}}}{Q} 
       =
        \Rect \big( \bigsqcup_{ i\in \langle l\rangle^{\circ}} \square^{\,p, R_i}, \square^{\,p, Q} \big)   \cong
          \Rect \big( \bigsqcup_{ i\in \langle l\rangle^{\circ}} \square^{\,p,R_i} , U \big)\simeq \Conf_{R_{\bullet}}(U),
$$    
    where \(R_i \subseteq Q\) for all \(1 \leq i \leq l\) (otherwise it is empty). 
    Therefore, we will proceed to show that 
    \[
    \Big|   \big(\CubeActive{p}{q}\big)_{/U,\, |  R_{\bullet}}  \Big| \simeq \Conf_{R_\bullet} (U)
    \]
    if \(R_i \subseteq Q\) for all \(1 \leq i \leq l\), or empty otherwise, through a map compatible with (\ref{eqt:Map between hofibers in WeakApproxofE}).

    By virtue of Lemma \ref{lem:strict pullback}, the homotopy fiber $\big(\CubeActive{p}{q}\big)_{/U,\, |  R_{\bullet}}$ over \( \{R_{\bullet}\}_{\langle l\rangle}\) is equivalent to the poset $\Prect$ described as follows. Objects are functions $V_{\bullet}\colon \langle l\rangle^{\circ}\to \mathsf{bsc}(\square^{\,p,q})$ such that $\bigsqcup_{i\in \langle l\rangle^{\circ}}V_i\subseteq U$ and whose types match the prescribed sequence \(R_\bullet\). More precisely, there exists a bijection $\upsigma$ of $\langle l\rangle^{\circ}$
    such that \(V_i\) is of type \(R_{\upsigma(i)}\). Morphisms in $\Prect$ are just families of inclusions $\big(V_i\overset{\sim}{\subseteq} V'_i\big)_{i\in \langle l\rangle^{\circ}}$ between open subsets of the same basic-type. Note that the empty-cases are already handled, as $\Prect$ is the empty poset in those instances.
    
    The idea is to apply Lurie--Seifert--van Kampen theorem \cite[Theorem A.3.1]{lurie_higher_nodate} to the functor
   $$
\begin{tikzcd}
    \upchi\colon &[-12mm] \Prect\ar[r] & \open\big(\Conf_{R_{\bullet}} (U)\big)\\[-8mm]
    & V_{\bullet} \ar[r, mapsto] & \displaystyle{\prod_{i\in \langle l\rangle^{\circ}} \Conf_{\{i\}}\big({V_{i}}_{[R_{\upsigma(i)}]}\big)}
\end{tikzcd}.
$$
    Given a point 
    \(
    \underline{x} = \{ x_k \}_{k \in \langle l \rangle^{\circ}} \in \Conf_{R_\bullet}(U),
    \)
    define the full subcategory
    \[
    \Prect_{\underline{x}} \coloneqq \big\{ V_{\bullet}\in \Prect :\;\; {\underline{x}} \in \upchi(V_{\bullet}) \big\} \subseteq \Prect.
    \]
    Hence, we are reduced to verify:
    \begin{itemize}
        \item The classifying space \( \lvert \Prect_{\underline{x}} \rvert \) is contractible.
        \item Each open set \( \upchi\left( V_{\bullet} \right) \) is contractible (obvious by inspection).
    \end{itemize}

    Firstly, \(  \Prect_{\underline{x}}  \) is non-empty for all \({\underline{x}}\in \Conf_{R_{\bullet}}(U)\). Since \(U\) is Hausdorff and sub-cubes form a basis for the topology, around each \(x_k\) we can find sub-cubes that we can shrink enough to make them pairwise disjoint for all \(k\). As \(x_k \in U_{[R_k]}\), by Remark \ref{remark: cubes are nice}, the sub-cubes will have the appropriate types \(R_{\bullet}\), making the list of sub-cubes an element in \(\Prect_{\underline{x}}\). Secondly, as  \(\Prect\) is already a poset, \(  \Prect_{\underline{x}}  \) is also a poset. The same process of finding and shrinking cubes can be utilized to show that for any two lists of sub-cubes in \( \Prect_{\underline{x}} \), there is a third list of smaller sub-cubes in the intersection. We conclude that \( \Prect_{\underline{x}} \) is cofiltered and therefore contractible.

    Finally, combining both points, we obtain the chain of weak homotopy equivalences
   $$
   \Big|\big(\CubeActive{p}{q}\big)_{/U,\, |  R_{\bullet}} \Big| \simeq \vert \Prect\vert \simeq \underset{V_{\bullet}\in \Prect}{\hocolim}\,\upchi(V,\upsigma)\simeq \Conf_{R_{\bullet}}(U),
   $$
   where the second equivalence is due to the fact that $\upchi(V_{\bullet})\simeq *$ for any $V_{\bullet}\in \Prect$, and the third is due to Lurie--Seifert--van Kampen theorem.
\end{proof}

Now, we move towards the second technical result supporting the derived additivity theorem. To continue, we will need to use the wreath product \(\Cube{p}{q} \wr \Cube{r}{t}\) (see \cite[Construction 4.3.8]{harpaz_little_nodate}). Recall that the objects of \(\Cube{p}{q} \wr \Cube{r}{t}\) are tuples of the form
    \[
    \big( (U_1, \dots, U_n), \{V_1^{1}, \dots, V_{l_1}^{1}\}, \dots, \{V_1^{n}, \dots, V_{l_n}^{n}\} \big),
    \]
    where each \(U_i\) is a sub-cube in \(\square^{\,p,q}\), and each \(V_j^{i}\) is a sub-cube in \(\square^{\,r,t}\). We will sometimes abbreviate this as \(\big( (U_1, \dots, U_n), \{V^{1}_{\bullet}\}, \dots, \{V^{n}_{\bullet}\} \big)\), or more succinctly \((U_{\bullet}, \{V^{\bullet}_{\bullet}\})\).

    A morphism from \((U_{\bullet}, \{V^{\bullet}_{\bullet}\})\) to \(('U_{\bullet}, \{'V^{\bullet}_{\bullet}\})\) consists of a tuple \((\upalpha, \{\upbeta_k\}_{k \in \langle n \rangle^{\circ}})\), where \(\upalpha : \langle n \rangle \to \langle m \rangle\) satisfies
    \[
    \bigsqcup_{i \in \upalpha^{-1}(j)} U_i \subseteq {}'U_j \quad \text{for all } j \in \upalpha(\langle n \rangle^{\circ}),
    \]
    and each \(\upbeta_k\) is a morphism \(V^{k}_{\bullet} \to {}'V^{\upalpha(k)}_{\bullet}\) in \(\Cube{r}{t}\). We equip the wreath product with a weak \(\infty\)-operad structure where the basics are objects of the form \((U, \{V\})\). A morphism \((\upalpha, \{\upbeta_{\bullet}\})\) is inert (resp.\ active) if \(\upalpha\) seen as morphism in \(\Cube{p}{q}\) and all \(\upbeta_{\bullet}\)'s in \(\Cube{r}{t}\) are inert (resp.\ active). 
There is an obvious weak $\infty$-operad map
\[
\uppsi': \Cube{p}{q} \wr \Cube{r}{t} \xrightarrow{\;\;\quad\;\;} \Cube{p+r}{q+t}
\]
which on objects acts as
\[
\big( (U_1, \dots, U_n), \{V^{1}_{\bullet}\}, \dots, \{V^{n}_{\bullet}\} \big)
\mapsto
\underset{(U_{\bullet} \times V^{\bullet}_{\bullet})}{\underbrace{\big(\,\overset{U_1\times V^1_{\bullet}}{\overbrace{ U_1 \times V_1^{1}, \dots, U_1 \times V_{l_1}^{1}}}, \dots, U_n \times V_1^{n}, \dots, U_n \times V_{l_n}^{n}\, \big)}}
 ,
\]
and on morphisms, it sends  
\(
(\upalpha, \{\upbeta_k\}_{k \in \langle n \rangle^{\circ}})
\) to the corresponding inclusions.

\begin{remark}
\label{rem:types of components of the product}
    For every \(U \subset U' \in \Cube{p}{q}\) and \(V \subset V' \in \Cube{r}{t}\), if the products have the same type, i.e. \(U \times V \overset{\sim}{\subseteq} U' \times V'\), then necessarily opens in each component of the product have the same type, \(U \overset{\sim}{\subseteq} U'\) and \(V \overset{\sim}{\subseteq} V'\). 
\end{remark}

\begin{thm}
\label{thm:WreathProdApprox}
    The composite map
    \[
    \uppsi : \Cube{p}{q} \wr \Cube{r}{t} \xrightarrow{\;\;\;\uppsi'\;\;\;} \Cube{p+r}{q+t} \xrightarrow{\;\;\;\upphi\;\;\;} \E_{p+r,q+t}
    \]
    is a weak approximation.
\end{thm}
\begin{proof}  Recall that for an object \( (U_{\bullet}, \{V^{\bullet}_{\bullet}\}) \) in the wreath product $\Cube{p}{q} \wr \Cube{r}{t}$, where \(U_{i} \times V_{j}^{i}\) has type \(Z^{i}_{j}\subseteq \underline{q+t}\), the map \(\uppsi\) sends it to
    \[
    \big( (U_1, \dots, U_n), \{V_1^{1}, \dots, V_{l_1}^{1}\}, \dots, \{V_1^{n}, \dots, V_{l_n}^{n}\} \big) 
    \mapsto 
    \{ Z_{1}^{1}, \dots, Z_{l_1}^{1} ,\dots, Z_{1}^{n}, \dots, Z_{l_n}^{n} \}.
    \]
 That being said, the first condition to be a weak approximation is proven in Lemma \ref{lem: First condition WreathProdApprox}, while the second is checked in Lemma \ref{lem: Second condition WreathProdApprox}.
\end{proof}

\begin{lemma}\label{lem: First condition WreathProdApprox}
    The map
    \begin{equation}
    \label{eq:localization2}
    \mathsf{C}\coloneqq \uppsi^{-1} \left( \E_{p+r,q+t} \right)_{\langle 1 \rangle}^{\inert}\xrightarrow{\;\;\quad\;\;}  \left( \E_{p+r,q+t} \right)_{\langle 1 \rangle}^{\inert}
     \end{equation}
    is an $\infty$-localization.
\end{lemma}
\begin{proof}
    As in the proof of Lemma \ref{lem:First condition WeakAprroxofE}, we have to show that \(( \E_{p+r,q+t} )_{\langle 1 \rangle}^{\inert}\) is a classifying space of \(\mathsf{C}\). Notice that the category \(\mathsf{C}\) is the wide subcategory of \((\Cube{p}{q} \wr \Cube{r}{t})_{\langle 1 \rangle}\) whose morphisms are of the form
    \[ 
    \begin{tikzcd}
    \label{fig:maps in the wreath prod}
  \big(U_{\bullet},\{V\}^{(\ell)}\big) = \big((U_{1}, \dots, U_{n}), \{\}, \dots, \underbrace{\{V\}}_{\text{position } \ell}, \dots, \{\}\big) \ar[d,"\upalpha \quad"'," \quad U_{\ell} \, \overset{\sim}{\subseteq} \, U'_{k} \text{ and } V \, \overset{\sim}{\subseteq} \, V'"] \\[6mm]
   \big(U'_{\bullet},\{V'\}^{(k)}\big) = \big((U'_{1}, \dots, U'_{m}), \{\}, \dots, \underbrace{\{V'\}}_{\text{position } k} , \dots, \{\}\big)
    \end{tikzcd},
    \]
   i.e.\ $\upalpha\colon \langle n\rangle\to \langle m\rangle $ sends $\ell$ to $k$, and  $U_{\ell}\overset{\sim}{\subseteq}U'_{k}$, $V\overset{\sim}{\subseteq} V'$ are inclusions of the same type.

    Following Proposition \ref{prop:localization}, we now show that the induced functor (\ref{eq:localization2}) has contractible fibers. Let \(Z\) be an object of \(( \E_{p+r,q+t} )_{\langle 1 \rangle}^{\inert}\). Following Lemma \ref{lem:strict pullback}, the homotopy fiber over \(Z\) is the full subcategory of \(\mathsf{C}\), whose objects are given by lists $\big(U_{\bullet},\{V\}^{(\ell)}\big)$ as above, with \(U_\ell \times V\) of type \(Z\). To show contractibility, it suffices to produce a zig-zag of natural transformations between the identity functor and a constant functor:
    \[
    \begin{tikzcd}
        \mathsf{C}_{\vert Z}
         \arrow[rr, bend left=65, "\id"{name=F}]
         \arrow[rr, "\uptheta"{inner sep=0,fill=white,anchor=center,name=G}]
         \arrow[rr, bend right=65, "\const"{name=H, swap}]
         \arrow[from=G.north-|F,to=F,Rightarrow,shorten=5pt,"\,\upalpha"'] 
         \arrow[from=G,to=H.north-|G,Rightarrow,shorten=5pt,"\,\upbeta"] &&
        \mathsf{C}_{\vert Z}     
    \end{tikzcd}.
    \]
    We define the functor \(\uptheta\) on objects by
    \(
      \big(U_{\bullet},\{V\}^{(\ell)}\big) \mapsto (U_\ell,\{V\}).
    \)
    The constant functor assigns to any object the fixed pair \((\square^{\,p,Q},\{\square^{\,r,T}\})\), where $Q\subseteq \underline{q}$, $T\subseteq  \underline{t}$ are the unique types such that \(\square^{\,p,Q} \times \square^{\,r,T}\), via $\square^{\,p,q}\times\square^{\,r,t}\cong \square^{\,p+r,q+t}$, is of type \(Z\subseteq \underline{q+t}\). 
   The natural transformation \(\upalpha\) is defined, for each object, by the morphism
   \[
   \upalpha_{(U_\bullet,V^{(\ell)})}: (U_\ell,\{V\}) \xrightarrow[]{\quad} \big(U_{\bullet},\{V\}^{(\ell)}\big),
   \] 
   induced by $\langle 1\rangle\to \langle n\rangle$, $1\mapsto \ell$,  and the identity maps on \(U_\ell\) and \(V\).
   Similarly, the transformation \(\upbeta\) is given by the morphisms
   \[
   \upbeta_{(U_\bullet,V^{(\ell)})}: (U_\ell,\{V\}) \xrightarrow[]{\quad} (\square^{\,p,Q},\{\square^{\,r,T}\}),
   \] 
   induced by the inclusions \(U_\ell \overset{\sim}{\subseteq} \square^{\,p,Q} \) and \(V \overset{\sim}{\subseteq} \square^{\,r,T}. \)
\end{proof}
\begin{lemma}\label{lem: Second condition WreathProdApprox} Let \((U_{\bullet}, \{V^{\bullet}_{\bullet}\})\) be an object of the wreath product such that, for \(i \in \langle n \rangle^{\circ}\) and \(j \in \langle l_i \rangle^{\circ}\), the product \(U_i \times V^{i}_{j}\) has type \(\square^{\,p+r, Z^{i}_{j}}\), where \(Z^{i}_{j} \subseteq \underline{q+t}\). Then, the homotopy fibers of the map
    \[
    \big( \Cube{p}{q} \wr \Cube{r}{t} \big)^{\act}_{/ (U_{\bullet}, \{V^{\bullet}_{\bullet}\})} \xrightarrow{\;\;\quad\;\;} \big(\Eactive_{p+r,q+t}\big)_{/ Z_{\bullet}^{\bullet}}
    \]
    are contractible.
\end{lemma}
\begin{proof}
       To simplify the notational burden, let us consider that the object in the wreath product is of the form $(U,\{V_1,\dots, V_l\})$, where $U\times V_k$ is of type $Z_k$. The general case is analogous. Similarly to the proof of Lemma \ref{lem:Second condition WeakAprroxofE}, we apply Proposition \ref{prop: Categorical triangle to analyze fibers} to the following composition of maps:
    \[
      \begin{tikzcd}
          {\big(\Cube{p}{q} \wr \Cube{r}{t}\big)^{\act}_{/ (U, \{V_{\bullet}\})} } 
          \ar[rr] 
         &&  
         {\big(\Eactive_{p+r,q+t}\big)_{/ Z_{\bullet} } }
         \ar[rr] 
         &&  {\Eactive_{p+r,q+t}},
    \end{tikzcd}
    \]
    and we check that the induced map on homotopy fibers
    \[
    \begin{tikzcd}
        \big(\Cube{p}{q} \wr \Cube{r}{t}\big)^{\act}_{/ (U, \{V_{\bullet}\}),\, |  R_{\bullet}}  \ar[rr]
        &&
      \big(\Eactive_{p+r,q+t}\big)_{/ Z_{\bullet} ,\, |  R_{\bullet}} 
    \end{tikzcd}
    \]
    is a weak homotopy equivalence, for any object \(\{R_\bullet\}_{\langle l' \rangle}\) in \(\Eactive_{p+r,q+t}\). First, observe that the homotopy fiber on the right hand side is just the mapping space 
    \begin{align*}
        {\Eactive_{p+r,q+t}} \brbinom{{\{R_\bullet\}_ {\langle l' \rangle}}}{\{ Z_{\bullet} \} _{\langle l \rangle}} 
        & =
         \coprod_{\substack{\text{active maps} \\ \upalpha: \langle l' \rangle \to \langle l \rangle}} \quad \prod_{j\in \langle l \rangle^{\circ}} \Rect \big( \underset{i \in \upalpha^{-1}(j)}{\bigsqcup}\square^{\,p+r, R_i}, \square^{\,p+r, Z_j} \big)   \\[2mm] 
         & \cong 
          \Rect \big( \bigsqcup_{i\in\langle l' \rangle^{\circ}} \square^{\,p+r,R_i} , \bigsqcup_{j\in \langle l \rangle^{\circ}} U \times V_j \big)\\[2mm]
         & = \Rect \big( \bigsqcup_{i\in\langle l' \rangle^{\circ}} \square^{\,p+r,R_i} , T \big)\\[2mm]
         & \simeq \Conf_{R_{\bullet}}(T), 
    \end{align*} 
    where we assume \(R_i \subseteq Z_{\upalpha(i)}\) for all \(i \), we have set $T:=\bigsqcup_{j\in \langle l\rangle^{\circ}} U \times V_j$, and we have applied Lemma \ref{lem:rect-conf} for the last equivalence. As in the proof of Lemma \ref{lem:Second condition WeakAprroxofE}, there are cases where the mapping space is empty, but for those the proof follows by inspection. Hence, it suffices to show that
    \begin{align*}
        \bigm|  \big(\Cube{p}{q} \wr \Cube{r}{t}\big)^{\act}_{/ (U, \{V_{\bullet}\}),\, |  R_{\bullet}}   \bigm|
          \simeq 
        \Conf_{R_{\bullet}}(T)\;. 
    \end{align*}    
    The idea is to again apply Lurie--Seifert–-van Kampen theorem.

   By Lemma \ref{lem:strict pullback}, the homotopy fiber $\big(\Cube{p}{q} \wr \Cube{r}{t}\big)^{\act}_{/ (U, \{V_{\bullet}\}),\, |  R_{\bullet}}$ can be identified with a poset that we will denote $\Prect$ in the remainder of this proof. Its objects are given by tuples
    \[
   (U'_{\bullet},\{W^{\bullet}_{\bullet}\})_{\upzeta}\equiv \Big( \big((U'_1, \dots, U'_{n'}), \{W_1^1, \dots, W_{k_1}^1\}, \dots, \{W_1^{n'}, \dots, W_{k_{n'}}^{n'}\} \big), \upzeta \Big),
    \]
    where $(U'_{\bullet},\{W^{\bullet}_{\bullet}\})\in \Cube{p}{q} \wr\Cube{r}{t} $ and 
    \(
    \upzeta=\{\upzeta_k\}_{k\in \langle n'\rangle^{\circ}} : (U'_\bullet, \{W^\bullet_\bullet\}) \longrightarrow (U, \{V_\bullet\})
    \)
    is an active morphism in such a wreath product, specifying how the sub-cubes sit inside \((U, \{V_\bullet\})\), subject to the following condition: the collection of product cubes \(\{U'_i \times W_j^i\}_{i,j}\) has types matching the prescribed sequence \(\{R_\bullet\}_{\langle l' \rangle}\). More precisely, there exists a bijection
    \[
    \uprho : \langle l' \rangle^{\circ} \xrightarrow{\quad\simeq\quad} \coprod_{i\in \langle n'\rangle^{\circ}} \langle k_i \rangle^{\circ}
    \]
    such that each product \(U'_i \times W_j^i\) is of type \(R_{\uprho^{-1}(i,j)}\). Morphisms in $\Prect$ are given by 1-to-1 inclusions between sub-cubes of the same type. Thus, we define the functor
       $$
\begin{tikzcd}
    \upchi\colon &[-12mm] \Prect\ar[r] & \open\big(\Conf_{R_{\bullet}} (T)\big)\\[-8mm]
    & (U'_{\bullet},\{W^{\bullet}_{\bullet}\})_{\upzeta} \ar[r, mapsto] & \displaystyle{\prod_{k\in \langle l'\rangle^{\circ}} \Conf_{\{k\}}\big((U'_{i}\times W^i_j)_{[R_k]}\big)}
\end{tikzcd},
$$
where $\uprho(k)=(i,j)$. 

Given a point $\underline{x}=\{x_k\}_{k\in \langle l'\rangle^{\circ}}\in \Conf_{R_{\bullet}}(T)$, define the full subcategory
$$
 \Prect_{\underline{x}} \coloneqq \big\{ ( U'_{\bullet}, \{W^{\bullet}_{\bullet}\})_{\upzeta} \in \Prect :\;\; {\underline{x}} \in \upchi\left( U'_{\bullet},\{W^{\bullet}_{\bullet}\}\right)_{\upzeta} \big\} \subseteq \Prect.
$$
Observe that the functor $\upchi$ is well defined since $\upzeta$ yields inclusions of open subsets 
$$
\bigsqcup_{\substack{i\in \langle n'\rangle^{\circ} \\ j\in \upzeta_i^{-1}(z)}}(U'_i\times W^i_j)_{[R_k]}\subseteq (U\times V_{z})_{[R_k]}\subseteq T_{[R_k]}
$$ 
for $k=\uprho^{-1}(i,j)$. Since it is clear that each value of the functor $\upchi$ is contractible, i.e.\ $\upchi(U'_{\bullet},\{W^{\bullet}_{\bullet}\})_{\upzeta}\simeq * $ for any object in $\Prect$, it remains to check that $\vert \Prect_{\underline{x}}\vert$ is also contractible to invoke Lurie--Seifert--van Kampen theorem and conclude the proof using
$$
      \Big|  \big(\Cube{p}{q} \wr \Cube{r}{t}\big)^{\act}_{/ (U, \{V_{\bullet}\}),\, |  R_{\bullet}}   \Big| \simeq \vert \Prect\vert \simeq \underset{(U'_{\bullet},\{W^{\bullet}_{\bullet}\})_{\upzeta}\in\Prect}{\hocolim}\;\upchi(U'_{\bullet},\{W^{\bullet}_{\bullet}\})_{\upzeta}\simeq \Conf_{R_{\bullet}}(T).
$$

   To do so, we will prove that the poset $\Prect_{\underline{x}}$ is cofiltered. Let us emphasize that the additional flexibility of the wreath product is fundamental; there are configurations $\underline{x}=\{x_k\}_{k\in \langle l'\rangle^{\circ}}\in\Conf_{R_{\bullet}}(T)$ that cannot be covered by operations coming from the product $\Cube{p}{q}\times \Cube{r}{t} $ alone, but can be realized within \(\Cube{p}{q} \wr \Cube{r}{t}\). We begin by verifying that \( \Prect_{\underline{x}} \) is non-empty. Without loss of generality, we can assume that there exists $z\in \langle l\rangle^{\circ}$ such that $x_k\in U\times V_z$ for all $k$ (otherwise, we separate the points $\{x_k\}_k$ according to the subspace $U\times V_{j}$ with  $j\in \langle l\rangle^{\circ}$ where they lie and use the following procedure for those living in the same subspace). There are two cases:
   \begin{itemize}
       \item If $x_{k}$ and $x_{k'}$ have different projections along $U\times V_z\to U$, then we can choose disjoint sub-cubes $D_{k},D_{k'}\subset U$ such that $x_k\in D_k\times V_z$ (resp.\ for $k'$). 
       \item If $x_{k}$ and $x_{k'}$ have equal projection along $U\times V_z\to U$, we can choose disjoint sub-cubes $C_{k},C_{k'}\subset V_z$ such that $x_k\in U\times C_k$ (resp.\ for $k'$). 
   \end{itemize}
   Combining these two, we can group the $x_k$'s in collections of points with the same projection along $\uppi_{U}\colon U\times V_z\to U$ (producing a family of disjoint sub-cubes $D_k\subseteq U$)\footnote{By convention, we just set $D_{k}=D_{k'}$ if $x_{k}$ and $x_{k'}$ have equal projection along $U\times V_z\to U$.}, and then separate points in the same subcollection by means of the second dot above (obtaining disjoint sub-cubes $C_k\subseteq V_z$). To illustrate the process, let us consider $\{x_1,x_2,x_3\}$ with $\uppi_{U}(x_1)=\uppi_{U}(x_3)\neq \uppi_{U}(x_2)$. In this case, we find disjoint sub-cubes $D_1=D_3$ and $D_2$ of $ U$, such that $\{x_1,x_3\}\subset D_1\times V_z$ and $\{x_2\}\subset D_2\times V_z$, and then refine by choosing disjoint sub-cubes $C_1,C_3\subset V_z$ satisfying $x_i\in D_i\times C_i$ for $i=1,3$. Thus, the resulting object in the wreath product covering $\{x_1,x_2,x_3\}$ takes the form
   $$
   \big((D_1=D_3,D_2),\{C_1,C_3\},\{C_2=V_z\}\big). 
   $$

    \begin{figure}[htp]
        \centering
        \input{figure3}
        \caption{An object in \(\Cube{p}{q} \wr \Cube{r}{t}\) covering a configuration $\{x_1,x_2,x_3\}$ with $\uppi_U(x_1)=\uppi_U(x_3)\neq \uppi_U(x_2)$.}
        \label{fig: figure 3}
    \end{figure}

    We have shown that $\Prect_{\underline{x}}$ is non-empty. Let us now argue that, given two objects $(U'_{\bullet},\{W^{\bullet}_{\bullet}\})_{\upzeta}$ and $('U'_{\bullet},\{'W^{\bullet}_{\bullet}\})_{\upzeta'} $ in $\Prect_{\underline{x}}$, there is a third $(''U'_{\bullet},\{''W^{\bullet}_{\bullet}\})_{\upzeta''}\in \Prect_{\underline{x}}$ equipped with maps $(U'_{\bullet},\{W^{\bullet}_{\bullet}\})_{\upzeta}\leftarrow (''U'_{\bullet},\{''W^{\bullet}_{\bullet}\})_{\upzeta''}\rightarrow ('U'_{\bullet},\{'W^{\bullet}_{\bullet}\})_{\upzeta'}$. Observe that for each \(k \in \langle l' \rangle^{\circ}\), the point \(x_k\) lies in the intersection of deepest strata
    $$
    x_k \in \left(U'_i \times W^i_j\right)_{[R_k]} \cap \left('U'_{i'} \times {}'W^{i'}_{j'}\right)_{[R_k]},
    $$
    with \(\uprho^{-1}(i,j)=k={\uprho'}\!\,^{-1}(i',j')\). Then, there exists a sub-cube $''U'_i\times {}''W^{i}_{j}$ of type $\square^{\,p+r,R_k}$ lying in the intersection of the given cubes that contains \(x_k\) (see Remark \ref{remark: cubes are nice}). By simply ensuring that the first factor in the product decomposition of the sub-cube is the same for points $x_k$ in $\underline{x}$ with the same projection along $U\times V_z\to U$, we obtain an object in $\Prect_{\underline{x}}$.

\end{proof}

\begin{proof}[Proof of Theorem \ref{thm:SwissCheeseAdditivity}]
   Consider the following commutative diagram
    \[
  \begin{tikzcd} 
         \Mon_{\E_{p+r,q+t}}(\Spc) \ar[rr] \ar[d,"\wr"'] &&  \Mon_{\E_{p,q} \,\wr\, \E_{r,t}} (\Spc) \ar[rr,"\sim", "\text{Theorem \ref{thm:wreathprod}}" '] && \Mon_{\E_{p,q} \times \E_{r,t}} (\Spc) \ar[d,"\wr"]  \\
        \Mon_{\Cube{p}{q}\,\wr\, \Cube{r}{t}}^{\mathsf{lc}}(\Spc) \ar[rrrr] \ar[d,hook] &&&& \Mon_{\Cube{p}{q} \times \Cube{r}{t}}^{\mathsf{lc}} (\Spc) \ar[d,hook] \\
         \Mon_{\Cube{p}{q}\,\wr\, \Cube{r}{t}}(\Spc) \ar[rrrr,"\sim", "\text{Theorem \ref{thm:wreathprod}}" '] &&&& \Mon_{\Cube{p}{q} \times \Cube{r}{t}} (\Spc) \\
    \end{tikzcd}.
    \]
    The vertical arrows at the top are equivalences by Theorems \ref{thm:WeakApproxofE} and \ref{thm:WreathProdApprox}, and by Lemmas \ref{rmk:ProductOfWeakApprox} and \ref{lem:Harpaz weak approx}. It thus suffices to verify that the middle horizontal arrow is an equivalence, as this will imply that the top horizontal arrow
     \[
     \Mon_{\E_{p+r,q+t}}(\Spc) \xrightarrow[]{\quad\sim\quad} \Mon_{\E_{p,q}\,\wr\, \E_{r,t}} (\Spc),
    \]
    is an equivalence as well, thereby concluding the proof.
    
    To establish that the middle horizontal arrow is an equivalence, we show that the bottom horizontal equivalence restricts to an equivalence between locally constant monoids. It is straightforward to check that every locally constant \(\Cube{p}{q} \wr \Cube{r}{t}\)-monoid is sent to a locally constant \(\Cube{p}{q} \times \Cube{r}{t}\)-monoid. Finally, we verify that every locally constant \(\Cube{p}{q} \times \Cube{r}{t}\)-monoid lies in the essential image of the functor \[
     \Mon_{\Cube{p}{q}\, \wr\, \Cube{r}{t}}^{\mathsf{lc}} (\Spc)\xlongrightarrow[]{\quad\phantom{\sim}\quad} \Mon_{\Cube{p}{q}\times\Cube{r}{t}} (\Spc).
    \]
    By Remark~\ref{rmk:LocConstMonoid}, every locally constant \(\Cube{p}{q} \times \Cube{r}{t}\)-monoid sends every element of the collection 
    \[
    \left\{ (U,V) \to (U',V') \;\middle|\; U \overset{\sim}{\subseteq} U' \text{ and } V \overset{\sim}{\subseteq} V' \right\}
    \]
    to equivalences. 
    In light of the discussion at the beginning of the proof of Lemma \ref{lem: First condition WreathProdApprox}, it remains to verify that every monoid \( A : \Cube{p}{q} \wr \Cube{r}{t} \xrightarrow[]{} \Spc\), which restricts to a locally constant \(\Cube{p}{q} \times \Cube{r}{t}\)-monoid, also maps any morphism of the form
$$
\begin{tikzcd}
  \big(U_{\bullet},\{V\}^{(\ell)}\big) = \big((U_{1}, \dots, U_{n}), \{\}, \dots, \underbrace{\{V\}}_{\text{position } \ell}, \dots, \{\}\big) \ar[d,"\upalpha \quad"'," \quad U_{\ell} \, \overset{\sim}{\subseteq} \, U'_{k} \text{ and } V \, \overset{\sim}{\subseteq} \, V'"] \\[6mm]
   \big(U'_{\bullet},\{V'\}^{(k)}\big) = \big((U'_{1}, \dots, U'_{m}), \{\}, \dots, \underbrace{\{V'\}}_{\text{position } k} , \dots, \{\}\big)
    \end{tikzcd},
$$
    to an equivalence. Just observe that the map \(A ( \upalpha )\) sits in the commutative diagram 
     \[ 
    \begin{tikzcd}
     A \big(U_{\bullet},\{V\}^{(\ell)}\big) 
    \ar[d,"A (\upalpha) \quad"']  \ar[r,"\sim"] &  
    A (U_{\ell},\{V\})
    \ar[d,"\wr"] \\
    A \big(U'_{\bullet},\{V'\}^{(k)}\big) \ar[r,"\sim"]  & 
    A (U'_{k},\{V'\})
    \end{tikzcd}.
    \]     
    The vertical map on the right is an equivalence by the assumption that the monoid \(A\) is locally constant upon restriction. The top and bottom horizontal maps are equivalences as a consequence of \(A\) being a monoid.  
\end{proof}

\section{Additivity for constructible factorization algebras}\label{sect: Additivity for cbl Fact}

\paragraph{Yet another localization theorem.}
Let us demonstrate that $\mathsf{Bsc}_X\to \gE_X$ exhibits an \(\infty\)-localization of \(\infty\)-operads. The ideas are adaptations of the ones presented in \cite{harpaz_little_nodate} to the corner situation together with, hopefully, simplified arguments at certain points.

First, an essentially formal result that will be useful in the proof of Proposition \ref{prop: LocalizationTheorem via WeakApprox}.
\begin{lemma}\label{lem:Simplifying slices for operads} Let $X$ be an $n$-manifold with corners and $u\colon \square^{\,p,q}\hookrightarrow X$ be an open embedding. Then, there is an identification of operads $(\cgE)_X = \gE_X$ (see Remark \ref{rk:Comparing the two little (p,q)-operads}) and  equivalences of \(\infty\)-categories
$$
\big(\gE_X^{\otimes,\act}\big)_{/(\square^{\,p,q},u)} \simeq \big(\cgE^{\otimes,\act}\big)_{/\square^{\,p,q}}\simeq \big(\gE_{\square^{\,p,q}}^{\otimes,\act}\big)_{/(\square^{\,p,q},\id)}\;.
$$    
\end{lemma}
\begin{proof} 
The identification is the content of  Remarks \ref{rk:Comparing the two little (p,q)-operads}. The operadic slice construction is also explained, for instance, in \cite[\textsection 2]{carmona_model_2022}. 

Regarding the equivalences, notice that the second one is just a particular instance of the first. Thus, we focus on the first equivalence. Let us check that there is a fully-faithful and essentially surjective functor between them. Actually, we use the forgetful functor induced by $\gE_{X}=(\cgE)_{X}\longrightarrow\cgE$ for the comparison. Let us start by observing that the objects in $\big(\gE_X^{\otimes,\act}\big)_{/(\square^{\,p,q},u)} $ are tuples $\big(\{(\square^{\,p_i,q_i},v_i)\}_{i\in I}\,,\,\uppsi\big)$ where $v_i\in \Emb(\square^{\,p_i,q_i},X)$ for any $i\in I$ and $\uppsi \in \gE_X\brbinom{\{(\square^{\,p_i,q_i},v_i)\}_{i\in I}}{(\square^{\,p,q},u)}$. Recall that this last space is the homotopy fiber at $(v_i)_i$ of
$$
u_{i,*}:= u_*\circ(\mathsf{inc}_i^*)_i\colon \Emb\big({\bigsqcup}_i\square^{\,p_i,q_i},\square^{\,p,q}\big) \longrightarrow {\prod}_i \Emb\big(\square^{\,p_i,q_i},X\big)\;.
$$
That is, a point $\uppsi$ is really given by $\big(\psi,(\upeta_i\colon v_i\Rightarrow u\cdot\psi\cdot\mathsf{inc}_i)_i\big)$.
Therefore, it is clear that the forgetful functor is essentially surjective up to equivalence. It remains to show that it is fully-faithful. We will analyze mapping spaces when the target consists of just one $\Drect'=(\square^{\,p'\!,q'},v')$ since the general case can be deduced from this one. We also introduce the notation $\Drect_i=(\square^{\,p_i,q_i},v_i)$ for the sake of compactness. 

Disentangling the definition of the mapping space of $\EuScript{X}=\big(\gE_X^{\otimes,\act}\big)_{/\Drect}$ with source $(\{\Drect_i\}_{i\in I},\uppsi)$ and target $(\Drect',\uppsi')$ we find
$$
     \begin{tikzcd}
         \EuScript{X}\big((\{\Drect_i\}_{i\in I},\uppsi), (\Drect',\uppsi') \big)\ar[rr]\ar[d] \ar[rrd, phantom, "\overset{\mathsf{h}\;\;}{\lrcorner}" description]
         && \ast \ar[d,"\uppsi"] \\
         \gE_X\brbinom{\{\Drect_i\}_{i\in I}}{\Drect'}
         \ar[rr] \ar[d] \ar[rrd, phantom, "\overset{\mathsf{h}\;\;}{\lrcorner}" description, near end]  &&  \gE_X\brbinom{\{\Drect_i\}_{i\in I}}{\Drect} \ar[d] \ar[r] \ar[rd, phantom, "\overset{\mathsf{h}\;\;}{\lrcorner}" description] & \ast \ar[d,"(v_i)_i"] \\
  \Emb\big( \bigsqcup_i\square^{\,p_i, q_i}, \square^{\, p'\!,q'} \big) \ar[rr,"\psi'_*"'] \ar[rrr, bend right=15,"v'_{i,*}"']&& \Emb\big( \bigsqcup_i\square^{\,p_i, q_i}, \square^{\, p,q} \big) \ar[r,"u_{i,*}"'] & \prod_i\Emb\big(\square^{\,p_i,q_i},X\big)
    \end{tikzcd},
$$
where all the squares are homotopy pullbacks and $v_{i,*}'\simeq u_{i,*}\psi'_*$ since $\uppsi'\in\gE_X\brbinom{(\square^{\,p'\!,q'},v')}{(\square^{\,p,q},u)}$ contains an isotopy witnessing such a relation. Comparing this diagram with the one defining the mapping space of $\EuScript{Y}= \big(\cgE^{\otimes,\act}\big)_{/\square^{\,p,q}}$ with source $(\{\square^{\,p_i,q_i}\}_{i\in I},\psi)$ and target $(\square^{\,p'\!,q'},\psi') $,
$$
  \begin{tikzcd}
         \EuScript{Y}\big( (\{\square^{\,p_i,q_i}\}_{i\in I},\psi),(\square^{\,p'\!,q'},\psi')\big)\ar[rr]\ar[d] \ar[rrd, phantom, "\overset{\mathsf{h}\;\;}{\lrcorner}" description]
         &&  * \ar[d,"\psi"] \\
         \Emb\big(\bigsqcup_i \square^{\,p_i, q_i}, \square^{\, p'\!,q'} \big)
         \ar[rr,"\psi'_*"']   && \Emb\big( \bigsqcup_i\square^{\,p_i,q_i}, \square^{\,p,q} \big) 
    \end{tikzcd},
$$
we conclude the proof.
\end{proof}

\begin{remark}\label{rem: Operadic Slice} The previous lemma should be clear from an appropriate point of view; it is a ``multilinear" version of $(\EuScript{C}_{/x})_{/(y,u)}\simeq \EuScript{C}_{/y}$ for any $(y,u)\in \EuScript{C}_{/x}$.   
\end{remark}

\begin{prop}\label{prop: LocalizationTheorem via WeakApprox} For any manifold with corners $X$, the operad map $\upgamma\colon\mathsf{Bsc}_X\to \gE_X$ exhibits $\gE_X$ as the operadic \(\infty\)-localization of $\mathsf{Bsc}_{X}$ at the set of unary maps
\begin{equation}\label{eqt: set of maps to localize BscX}
    \left\lbrace
U\subseteq V \; \text{ in }\mathsf{bsc}(X)\text{ s.t. }\; U\cong \mathbbst{R}^{p,q} \text{ is abstractly diffeomorphic to }V\cong\mathbbst{R}^{p'\!,q'}
\right\rbrace\;.
\end{equation} 
\end{prop}
\begin{proof}
Due to Corollary \ref{cor: Harpaz approxs as operadic notions}, we can show that $\upgamma$ is a weak approximation. Lemma \ref{lem: First condition LocalizationTheorem via WeakApprox} below handles the first condition, while  Lemma \ref{lem: Second condition LocalizationTheorem via WeakApprox} proves the second.
\end{proof}

\begin{lemma}\label{lem: First condition LocalizationTheorem via WeakApprox}
    The functor $\upgamma^{-1}(\gE^{\otimes}_{X})^{\inert}_{\langle 1\rangle}\longrightarrow (\gE^{\otimes}_X)^{\inert}_{\langle 1\rangle}$ exhibits an $\infty$-localization.
\end{lemma}

\begin{proof}
    Noticing that $(\gE^{\otimes}_X)^{\inert}_{\langle 1\rangle}$ is an $\infty$-groupoid, we can prove that the classifying space of $\upgamma^{-1}(\gE^{\otimes}_{X})^{\inert}_{\langle 1\rangle}$ is weak homotopy equivalent to $(\gE^{\otimes}_{X})^{\inert}_{\langle 1\rangle}$ to conclude the localization statement. To do so, we apply Quillen's Theorem A: for any given $(\square^{\,p,q},u)\in (\gE^{\otimes}_{X})^{\inert}_{\langle 1\rangle}$, we need to show that the slice $\big(\upgamma^{-1}(\gE^{\otimes}_{X})^{\inert}_{\langle 1\rangle}\big)_{/(\square^{\,p,q},u)}$ is contractible. Observe that $\upgamma^{-1}(\gE^{\otimes}_{X})^{\inert}_{\langle 1\rangle}$ is the wide subcategory of $(\mathsf{Bsc}_X^{\otimes})_{\langle 1\rangle}=\mathsf{bsc}(X)$ spanned by the morphisms in (\ref{eqt: set of maps to localize BscX}), and that the same arguments provided in Lemma \ref{lem:Simplifying slices for operads} allow us to identify $\big(\upgamma^{-1}(\gE^{\otimes}_{X})^{\inert}_{\langle 1\rangle}\big)_{/(\square^{\,p,q},u)}$ with the wide subcategory of $\mathsf{bsc}(\square^{\,p,q})$ spanned by open inclusions $U\subseteq V$ where both $U$, $V$ are abstractly diffeomorphic to $\square^{\,p,q}$. This category is clearly contractible since it has a terminal object $\square^{\,p,q}$.
\end{proof}

\begin{lemma}\label{lem: Second condition LocalizationTheorem via WeakApprox}
 The homotopy fibers of $\big(\mathsf{Bsc}_X^{\otimes,\act}\big)_{/V}\longrightarrow \big(\gE_X^{\otimes,\act}\big)_{/\upgamma(V)}$ are contractible for any $V\in \mathsf{Bsc}_{X}$,.
\end{lemma}
\begin{proof} Let us begin by introducing notation for the abstract diffeomorphism type of $V$, say $\upgamma(V)=(\square^{\,p,q},u)$. Now use Lemma \ref{lem:Simplifying slices for operads} to identify $\big(\gE_X^{\otimes,\act}\big)_{/(\square^{\,p,q},u)} \simeq  \big(\gE_{\square^{\,p,q}}^{\otimes,\act}\big)_{/(\square^{\,p,q},\id)}$ and direct inspection for the discrete version $\big(\mathsf{Bsc}_X^{\otimes,\act}\big)_{/u(\square^{\,p,q})} \simeq \big(\mathsf{Bsc}_{\,\square^{\,p,q}}^{\otimes,\act}\big)_{/\square^{\,p,q}}$. 

Then, the map to analyze gets identified with the horizontal arrow in the commutative triangle
$$
\begin{tikzcd}
   \big(\mathsf{Bsc}_{\,\square^{\,p,q}}^{\otimes,\act}\big)_{/\square^{\,p,q}} \ar[rr]\ar[rd] && \big(\gE_{\square^{\,p,q}}^{\otimes,\act}\big)_{/(\square^{\,p,q},\id)}\ar[ld] \\
    & \gE_{\square^{\,p,q}}^{\otimes,\act}
\end{tikzcd}.
$$
We are in a position to apply Proposition \ref{prop: Categorical triangle to analyze fibers} to compare homotopy fibers. Choose a function $Q:\langle m\rangle^{\circ} \to \Prect(\square^{\,p,q})\cong (\underline{2}^{\underline{q}})^{\op}$ (this is the same datum as an object in $\gE_{\square^{\,p,q}}^{\otimes,\act}$) and observe that the homotopy fiber over $Q$ of the right diagonal map, $(\mathsf{Bsc}_{\,\square^{\,p,q}}^{\otimes,\act})_{/\square^{\,p,q},\,\vert Q}$, is weak homotopy equivalent to $\Emb_{/\square^{\,p,q}}\big(\bigsqcup_{i\in \langle m\rangle^{\circ}}\square^{\,p,Q_i}, \square^{\,p,q}\big)\simeq \Conf_{Q}(\square^{\,p,q})$; see Notation \ref{notat: Configuration spaces} and the proof of Proposition \ref{prop: Coincidence of little (p,q)-disc operads}. Thus, it suffices to prove that the homotopy fiber over $Q$ of the left diagonal map is weak homotopy equivalent to $\Conf_{Q}(\square^{\,p,q})$ via the map induced from the commutative triangle above\footnote{There is also a case where both homotopy fibers are empty, but it should be clear that such situation poses no problems (cf. Proof of Lemma \ref{lem:Second condition WeakAprroxofE}).}.

By means of diagram (\ref{eqt:hopbs showing space of relative embeddings is contractible}) and Lemma \ref{lem:strict pullback}, we observe that this homotopy fiber is equivalent to the ordinary category $\Prect$ with: objects, functions $U\colon \langle m\rangle^{\circ}\to \mathsf{bsc}(\square^{\,p,q})$ such that $\bigsqcup_{i\in \langle m\rangle^{\circ}}U_i\subseteq \square^{\,p,q}$ and with types matching $Q$ (see the proof of Lemma \ref{lem:Second condition WeakAprroxofE}); and morphisms, families of inclusions $\big(U_i\overset{\sim}{\subseteq} V_i\big)_{i\in \langle m\rangle^{\circ}}$ between open subsets of the same basic-type. We will conclude the proof by applying Lurie--Seifert--van Kampen theorem to the functor
\begin{align*}
  \upchi\colon \Prect &\xrightarrow{\qquad\quad} \open\big(\Conf_{Q} (\square^{\,p,q}) \big)\\
  U & \xmapsto{\qquad\quad} \prod_{i\in \langle m\rangle^{\circ}} \Conf_{\{i\}}\big({U_i}_{[Q_{\upsigma(i)}]}\big)\quad.
\end{align*} 
The hypotheses of the cited theorem are readily verified in this case, and moreover, $\upchi(U)$ is contractible for any $U$. Therefore, we obtain a chain of weak homotopy equivalences
$$
\Big|\big(\mathsf{Bsc}_{\,\square^{\,p,q}}^{\otimes,\act}\big)_{/\square^{\,p,q},\,\vert Q}\Big| \simeq \vert \Prect\vert \simeq \underset{U\in \Prect}{\hocolim}\,\upchi(U)\simeq \Conf_Q(\square^{\,p,q}),
$$
where the second equivalence comes from the contractibility of $\upchi(U)$ for any $U\in \Prect$, and the third one follows from Lurie--Seifert--van Kampen theorem.
\end{proof}

\paragraph{Hyperdescent of constructible factorization algebras.} In this subsection, we provide an alternative proof of \cite[Theorem 6.1 and Proposition 6.10]{karlsson_assembly_2024} for manifolds with corners. Since no additional effort is required, we will show that $U\mapsto \Fact^{\cbl}_{U}(\EuScript{V})$ satisfies descent for arbitrary open hypercovers.\footnote{We follow the conventions on (hyper)covers from \cite{karlsson_assembly_2024}. See \cite[\textsection 2.5]{karlsson_assembly_2024} for an explanation of the equivalence between perspectives on (hyper)covers when working with the site $\open(X)$.} 

Recall that an open inclusion $\upiota\colon V\hookrightarrow U$ induces a pair of functors of \(\infty\)-categories
$$
\upiota_{*}\colon \Fact_{V}(\EuScript{V})\longrightarrow \Fact_{U}(\EuScript{V})\quad \text{ and }\quad \upiota^*\colon \Fact_{U}(\EuScript{V})\longrightarrow\Fact_{V}(\EuScript{V})
$$
called pushforward and restriction, respectively \cite[Lemma 3.43 and Lemma 3.44]{karlsson_assembly_2024}). Also, remember that being locally constant/constructible is a local property (see \cite[Proposition 24]{ginot_notes_2013} and \cite[\textsection 5.1]{karlsson_assembly_2024}), i.e.\ $\mathdutchcal{F}\in \Fact_{X}(\EuScript{V})$ 
is locally constant (resp.\ constructible) iff so is $\upiota^*\mathdutchcal{F}\in \Fact_{U}(\EuScript{V})$, for any open subset $U\subseteq X$ in an open cover of $X$. Hence, the pushforward and restriction functors are also defined for constructible factorization algebras 
$$
\upiota_{*}\colon \Fact_{V}^{\cbl}(\EuScript{V})\longrightarrow \Fact_{U}^{\cbl}(\EuScript{V})\quad \text{ and }\quad \upiota^*\colon \Fact_{U}^{\cbl}(\EuScript{V})\longrightarrow\Fact_{V}^{\cbl}(\EuScript{V}),
$$
for any open inclusion $\upiota\colon V\hookrightarrow U$. It is important to emphasize that the pushforward of constructible factorization algebras along arbitrary maps is not well-defined, in contrast to the pushforward of plain factorization algebras. However, there are classes of maps, like locally trivial fiber bundles, for which it is so (see \cite[Proposition 25 and 26]{ginot_notes_2013} and \cite[Example 3.45 and Remark 3.46]{karlsson_assembly_2024}).

For any manifold with corners $X$, the restriction functors along open inclusions yield a presheaf of \(\infty\)-categories
$$
\begin{tikzcd}
\Fact^{\cbl}_{\scalebox{0.4}{$(-)$}\!}(\EuScript{V})\colon  &[-12mm]      \open(X)^{\mathsf{op}} \ar[rr] && \mathsf{Cat}_{\infty} \\[-8mm]
 &   U  && \Fact_{U}^{\cbl}(\EuScript{V}) \\[-2mm]
  &  V \ar[u,hookrightarrow,"\upiota", "\quad"'{name=S}] && \Fact_{V}^{\cbl}(\EuScript{V}) \ar[u, leftarrow, "\upiota^*"', "\quad"{name=T}] \ar[mapsto, from=S, to=T]
\end{tikzcd}.
$$
Our goal in this subsection is to demonstrate that this presheaf is, in fact, a hypersheaf. That is, for any hypercover $\mathcal{U}\hookrightarrow\open(U)$ of $U\subseteq X$, the functor 
$$
\Fact_{U}^{\cbl}(\EuScript{V})\xrightarrow{\;\;\quad\;\;} \underset{V\in \mathcal{U}^{\op}}{\holim}\Fact_{V}^{\cbl}(\EuScript{V})
$$
is an equivalence of \(\infty\)-categories.

\begin{remark}
    Dropping the constructibility condition, we also have a presheaf of \(\infty\)-categories $U\mapsto \Fact_{U}(\EuScript{V})$. One can ask whether this is also a sheaf or not. To the best of our knowledge, this is an important open problem. Unfortunately, this question cannot be easily addressed with the ideas presented below, and would require important additions to the strategy followed in \cite{karlsson_assembly_2024}. 
\end{remark}

Another piece of structure that will be important for our purposes is the symmetric monoidal structure on $\Fact_{X}^{\cbl}(\EuScript{V})$ obtained in \cite[\textsection 6.2]{karlsson_assembly_2024}. We will refer to it using capital letters, $\FACT^{\cbl}_{X}(\EuScript{V})\to \Fin_*$ or simply $\FACT^{\cbl}_{X}(\EuScript{V})$, following the convention in \cite{haugseng_allegedly_nodate,haugseng_algebras_2024} that we have adopted in this work (cf.\  \cite{karlsson_assembly_2024}). Recall that there is a strong symmetric monoidal inclusion $\FACT^{\cbl}_X(\EuScript{V})\hookrightarrow \ALG_{\mathsf{Disj}_X}(\EuScript{V})$, which in plain words asserts that the monoidal structure on $\Fact^{\cbl}_{X}(\EuScript{V})$ is given by pointwise tensor products.

\begin{remark} The cited construction \cite[\textsection 6.2]{karlsson_assembly_2024} requires restricting $\EuScript{V}$ to be a presentable symmetric monoidal $\infty$-category. This aligns with our general convention: we assume $\EuScript{V}$ is presentably symmetric monoidal whenever we invoke factorization algebras. 
\end{remark}

Finally, before proceeding with our proof of hyperdescent for $\Fact^{\cbl}_{\scalebox{0.4}{$(-)$}\!}(\EuScript{V})$, we want to collect a pivotal assumption that both proofs, ours and Karlsson--Scheimbauer--Walde's, use in a crucial way.  

\begin{hypothesis}\label{hyp: Good supply of Weiss covers}
    Any open subset $U$ of $X$ admits a Weiss cover by finite disjoint unions of basics whose finite intersections remain so, i.e.\ for any open $U\subseteq X$, there is a presieve $\mathcal{U}\hookrightarrow \mathsf{bsc}(U)\hookrightarrow \open(U)$ which is a Weiss cover of $U$.
\end{hypothesis}

Proposition \ref{prop: Good supply of Weiss covers} shows that Hypothesis \ref{hyp: Good supply of Weiss covers} is satisfied for any manifold with corners $X$, which is our case of study. However, we will explicitly mention the hypothesis in the following results to emphasize which parts of the argument require it.

\begin{lemma}\label{lem:cbl Fact as operadic algebras} Let $X$ be a manifold with corners and assume that Hypothesis \ref{hyp: Good supply of Weiss covers} holds. Then, the operad maps $\mathsf{Disj}_X\leftarrow \mathsf{Bsc}_X\to \gE_X$ induce equivalences of \(\infty\)-categories
$$
\begin{tikzcd}
    \Fact_{X}^{\cbl}(\EuScript{V}) \ar[r, "\sim"] & \Alg_{\mathsf{Bsc}_X}^{\cbl}(\EuScript{V})\ar[r, "\sim", leftarrow] & \Alg_{\gE_X}(\EuScript{V}).
\end{tikzcd}
$$    
Moreover, one can leverage these equivalences to symmetric monoidal equivalences
$$
\begin{tikzcd}
    \FACT_{X}^{\cbl}(\EuScript{V}) \ar[r, "\sim"] & \ALG_{\mathsf{Bsc}_X}^{\cbl}(\EuScript{V})\ar[r, "\sim", leftarrow] & \ALG_{\gE_X}(\EuScript{V}).
\end{tikzcd}
$$
\end{lemma}
\begin{proof}
The first equivalence follows from \cite[
Corollaries 4.24 and 5.24]{karlsson_assembly_2024}, if we use a basis, or from the toolkit developed in \cite{carmona_model_2022}. The second one follows from Proposition \ref{prop: LocalizationTheorem via WeakApprox}. The lifting to symmetric monoidal equivalences simply follows since the symmetric monoidal structure on the three $\infty$-categories is the pointwise tensor product.
\end{proof}

Thanks to this result, our problem translates into proving that 
$$
\Alg_{\gE_{U}}(\EuScript{V})\xrightarrow{\;\;\quad\;\;} \underset{V\in \mathcal{U}^{\op}}{\holim}\Alg_{\gE_{V}} (\EuScript{V})
$$
is an equivalence of \(\infty\)-categories. Note that this functor comes from restrictions along operad maps $\gE_{V}\to \gE_{U}$ associated to open inclusions $V\hookrightarrow U$. Showing that this is the case (see Theorem \ref{thm: Hyperdescent for cbl fact}) requires a little detour. The main idea is to prove and exploit Proposition \ref{prop: Strong approximation} below.

\begin{defn}
    Let \( X \) be an $n$-manifold with corners. We define the $\infty$-category of basics in \( X \), denoted \( \overline{X} \), as the fiber product
    \[
    \overline{X} \coloneqq {\cMfld}_{/X} \underset{\cMfld}{\times} {\Basic}_{n},
    \]
    where $\Basic_n $  denotes the full subcategory of $\cMfld $ spanned by the manifolds \(\{\square^{\,p,q}\}_{p+q=n}\) (see Notation \ref{notat:Topological category of manifolds with corners}).  Objects of \( \overline{X} \) consist of pairs $(\square^{\,p,q},\upiota)$, where $\upiota\colon \square^{\,p,q}\hookrightarrow X$ is an open embedding, and the mapping spaces are given by spaces of open embeddings over \( X \),
    \(
    \Emb_{/X}(\square^{\,p,q}, \square^{\,p'\!,q'}).
    \)
\end{defn}

\begin{remark}
    By Example \ref{ex:weak structure on infinity cat}, the $\infty$-category \({\overline{X}}\)  can be equipped with the structure of a weak \(\infty\)-operad in which every object is basic, inert morphisms are embeddings between basics of the same type, and all morphisms are regarded as active.
\end{remark}

Note that the spaces of multimorphisms in \(\gE_{X}\) come equipped with a canonical map 
\[
    \mathsf{c}\colon \gE_{X} \brbinom{\{\Drect_i\}_{i\in I}}{\Drect}
    \xlongrightarrow[]{\;\;\quad \;\;}
    \prod_{i\in I} \Emb_{/X}(\square^{\, p_i,q_i},\square^{\, p,q}),
\]
where \(\Drect=(\square^{\,p,q}, \upiota)\) with \(\upiota\colon \square^{\,p,q}\hookrightarrow X\), and analogously for $\Drect_i$. In fact, they yield an operad map  
\(
\mathsf{c}\colon\gE_{X}^{\otimes} \xrightarrow{\quad}{\overline{X}}\!\,^{\amalg}
\)
into the cocartesian $\infty$-operad associated with $\overline{X}$. 

The idea now is to use $\overline{X}$ to ``disintegrate" $\gE_X$ using a directed version of \(\int_{\overline{X}} \E_x\) \cite[Definition 6.2.8]{harpaz_little_nodate} (cf.\ Remarks \ref{rem: Operadic Slice} and \ref{rem: Why directed integral is needed}). For this purpose, we will denote by $(\overline{X}\,\!^{\amalg})^{\Delta^1}_{\mathsf{p}=\id}$ the full subcategory of the arrow $\infty$-category $(\overline{X}\,\!^{\amalg})^{\Delta^1}$ spanned by maps in $\overline{X}$ whose projection along $\mathsf{p}\colon \overline{X}\,\!^{\amalg}\longrightarrow \Fin_*$ is an identity. We also need the diagonal map 
\[
\Delta: \Fin_{*} \times \overline{X} \xrightarrow{\;\;\quad\;\;} {\overline{X}}\!\,^{\amalg}, \quad (\langle m\rangle , \square^{\,p,q})\longmapsto (\square^{\,p,q},\dots,\square^{\,p,q}).
\]

\begin{defn}\label{defn: Directed Integral} We define the $\infty$-category $\vec{\int}\!_{x\in \overline{X}}\,\E_x$ as the homotopy limit of
$$
\begin{tikzcd}
 \gE_X^{\otimes} \ar[rd,"\mathsf{c}"']  && (\overline{X}\,\!^{\amalg})^{\Delta^1}_{\mathsf{p}=\id} \ar[ld,"\mathsf{ev}_0"']\ar[rd,"\mathsf{ev}_1"] && \Fin_*\times \overline{X}\ar[ld,"\Delta"]\\
  & \overline{X}\,\!^{\amalg} && \overline{X}\,\!^{\amalg}
\end{tikzcd}.
$$
As objects, it has tuples $\big(\langle n\rangle, (\upxi_i\colon (\square^{Z_i}, u_i)\hookrightarrow (\square^{Z},v) )_{i\in \langle n \rangle^{\circ}} \big)$, where $\{(\square^{Z_i},u_i)\}_{i\in \langle n \rangle^{\circ}}$ is the associated object in $\gE_X^{\otimes}$, analogously for $(\square^{Z},v)$ and $\overline{X}$, and each $\upxi_i$ is a morphism in $\overline{X}$. To alleviate the notation, we will refer to such an object by $(\langle n \rangle, (\upxi_i)_i)$ or even $(\upxi_i)_i$. Its mapping spaces are given by the following homotopy pullbacks
$$
\begin{tikzcd}
\vec{\int}\!_{\overline{X}}\,\E_x\big((\upxi_i)_i,(\upxi'_j)_j\big) \ar[rr]\ar[d] \ar[rrd, phantom, "\overset{\mathsf{h}\;\;}{\lrcorner}" description, near start]
         &&  \displaystyle \coprod_{\upbeta\colon \langle n\rangle\to \langle n'\rangle}\prod_{j\in \langle n'\rangle^{\circ}}\Emb\big(\bigsqcup_{i\in \upbeta^{-1}(j)}\square^{Z_i},\square^{Z'_j}\big) \ar[d,"(\upxi'_{j})_*\,\cdot\, (\mathsf{inc}_i)^*"] \\
 \displaystyle \coprod_{\upbeta\colon \langle n\rangle\to \langle n'\rangle} \Emb_{/X}(\square^{Z},\square^{Z'}) 
         \ar[rr,"(\upxi_i)^*"']   && \displaystyle \coprod_{\upbeta\colon \langle n\rangle\to \langle n'\rangle}\;\prod_{i\in \upbeta^{-1}(\langle n'\rangle^{\circ})} \Emb(\square^{Z_i},\square^{Z'})
\end{tikzcd}.
$$    
\end{defn}

\begin{remark} We equip $\vec{\int}\!_{\overline{X}}\,\E_x$ with a weak $\infty$-operad structure as follows: First, note that the homotopy limit which defines this $\infty$-category yields the following diagram 
\begin{equation}\label{eqt: weak opd structure on directed Integral}
\begin{tikzcd}
  \vec{\int}\!_{\overline{X}}\,\E_x \ar[r]\ar[dd]\ar[rdd, phantom, "\overset{\mathsf{h}\;\;}{\lrcorner}" description]  &[6mm] \EuScript{T}\ar[r]\ar[d]\ar[rd, phantom, "\overset{\mathsf{h}\;\;}{\lrcorner}" description] & \Fin_*\times \overline{X}\ar[d,"\Delta"]\\
   &  (\overline{X}\,\!^{\amalg})^{\Delta^1}_{\mathsf{p}=\id}\ar[r, "\mathsf{ev}_1"']\ar[d,"\mathsf{ev}_0"] & \overline{X}\,\!^{\amalg}\\
   \gE_X^{\otimes}\ar[r,"\mathsf{c}"'] & \overline{X}\,\!^{\amalg}
\end{tikzcd} .
\end{equation}
Hence, by Remark \ref{rem: weak operad structure on the pullback}, since $\gE_X^{\otimes}$, $\overline{X}\,\!^{\amalg}$, and $\Fin_*\times \overline{X}$ have natural weak $\infty$-operad structures, it remains to equip $(\overline{X}\,\!^{\amalg})^{\Delta^1}_{\mathsf{p}=\id}$ with such a structure. For that purpose, we set the basics as those maps in $\overline{X}$ lying over $\id_{\langle 1\rangle}$, and we say that a morphism 
$$ 
\begin{tikzcd}
 \{(\square^{Z_i},u_i)\}_{i\in \langle n\rangle^{\circ}} \ar[rrrr,"\big(\upbeta\text{,}(\uppsi_i)_{i\in\upbeta^{-1}(\langle m\rangle^{\circ})}\big)"] \ar[d, shift right=4,"\big(\id_{\langle n\rangle}\text{,}(\upxi_i)_i\big)"'] &[6mm]&&&   \{(\square^{\widetilde{Z}_j},\widetilde{u}_j)\}_{j\in \langle m\rangle^{\circ}} \ar[d, shift right=4,"\big(\id_{\langle m\rangle}\text{,}(\widetilde{\upxi}_j)_j\big)"]\\[8mm]
  \{(\square^{Z'_i},u'_i)\}_{i\in \langle n\rangle^{\circ}} \ar[rrrr,"\big(\upbeta\text{,}(\uppsi'_i)_{i\in\upbeta^{-1}(\langle m\rangle^{\circ})}\big)"'] &&&&  \{(\square^{\widetilde{Z}'_i},\widetilde{u}'_j)\}_{j\in \langle m\rangle^{\circ}}
\end{tikzcd}
$$
in $(\overline{X}\,\!^{\amalg})^{\Delta^1}_{\mathsf{p}=\id}$ is inert (resp.\ active) if $\upbeta$ is inert and $\uppsi_i$, $\uppsi'_i$ are equivalences for any $i$ (resp.\ $\upbeta$ is active, and no further restriction on the components $(\uppsi_i)_i$, $(\uppsi'_i)_i$). By definition, $\mathsf{ev}_0$ and $\mathsf{ev}_1$ are morphisms of weak $\infty$-operads.
\end{remark}

The next set of results tackles the formal requirements we would need about $\vec{\int}\!_{\overline{X}}\,\E_x$ to prove Proposition \ref{prop: Strong approximation}.

\begin{lemma}\label{lem: Directed integral is cocartesian fibration}
The canonical map $\vec{\int}\!_{\overline{X}}\,\E_x\xrightarrow{\quad}\Fin_*\times \overline{X}\xrightarrow{\quad} \overline{X}$ is a cocartesian fibration.    
\end{lemma}
\begin{proof} Although a more abstract proof is possible, we prefer the following hands-on approach, since it will allow us to see why this directed version of \cite[Definition 6.2.8]{harpaz_little_nodate} is required when $\overline{X}$ is not an $\infty$-groupoid in Remark \ref{rem: Why directed integral is needed}. See \cite{mazel-gee_users_2019} for a short and illustrative note on cocartesian fibrations.  

Let $(\langle n\rangle, (\upxi_i)_i)$ be an object in $\vec{\int}\!_{\overline{X}}\,\E_x$ and $\upzeta\colon (\square^{Z},v)\hookrightarrow (\square^{Z'},v')$ be a morphism in $ \overline{X}$, whose source matches the image of $(\langle n\rangle, (\upxi_i)_i)$ along $\vec{\int}\!_{\overline{X}}\,\E_x\xrightarrow{\quad} \overline{X}$. Then, the candidate cocartesian lift of $\upzeta$ with source $(\langle n\rangle, (\upxi_i)_i)$ is given by 
$$
(\id_{\langle n\rangle }, (\id_{\square^{Z_i}})_{i\in \langle n\rangle^{\circ}},\upzeta)\colon (\langle n\rangle, (\upxi_i)_i) \xrightarrow{\;\;\quad\;\;} \underset{ (\langle n\rangle,\,(\upzeta\cdot\upxi_i)_i)}{\underbrace{\upzeta_!(\langle n\rangle, (\upxi_i)_i)}}\;.
$$
To prove that this candidate is actually cocartesian, we consider the following commutative diagram
$$
\begin{tikzcd}
    \vec{\int}\!_{\overline{X}}\,\E_x\big(\upzeta_!(\upxi_i)_i,(\upxi_j'')_j\big) \ar[rr]\ar[d] && \vec{\int}\!_{\overline{X}}\,\E_x\big((\upxi_i)_i,(\upxi_j'')_j\big) \ar[d]\\
   \displaystyle \coprod_{\upbeta\colon \langle n\rangle \to \langle n''\rangle}\Emb_{/X}(\square^{Z'},\square^{Z''})\ar[rr,"\upzeta^*"']\ar[d] &&  \displaystyle \coprod_{\upbeta\colon \langle n\rangle \to \langle n''\rangle}\Emb_{/X}(\square^{Z},\square^{Z''})\ar[d]\\
   \Emb_{/X}(\square^{Z'},\square^{Z''})\ar[rr,"\upzeta^*"'] && \Emb_{/X}(\square^{Z},\square^{Z''})
\end{tikzcd}.
$$
To see that the square on top is a homotopy pullback, apply the cancellation property of such to the squares from the definition of mapping spaces in $\vec{\int}\!_{\overline{X}}\,\E_x$. The lower square is clearly a homotopy pullback, and hence we deduce that the previously defined lift is cocartesian.
\end{proof}

\begin{remark}\label{rem: Why directed integral is needed} Observe that a more naive variant of Definition \ref{defn: Directed Integral} using simply a homotopy pullback
   \[
        \begin{tikzcd}
           \int_{\overline{X}}\E_x \ar[r] \ar[d] \ar[rd, phantom, "\overset{\mathsf{h}\;\;}{\lrcorner}" description] & \Fin_{*} \times \overline{X} \ar[d, "\Delta"] \\
           \gE_{X}^{\otimes} \ar[r] & {\overline{X}}\!\,^{\amalg}
        \end{tikzcd},
    \]
    in line with \cite[Definition 6.2.8]{harpaz_little_nodate}, is doomed to fail for our purposes. Namely, the canonical map $\int_{\overline{X}}\E_x\longrightarrow \overline{X}$ is not a cocartesian fibration when $X$ is a honest manifold with corners (even a manifold with boundary), in contrast to the cases studied by Lurie and Harpaz where $\overline{X}$ is an $\infty$-groupoid. For the easiest counterexample, consider $X=\square^{\,0,1}=[0,1)$ and the  inclusion $(0,1)\subset [0,1)$ as the map $\upzeta$ in $\overline{X}$. In the data of the homotopy pullback, instead of arbitrary maps $\upxi_i\colon (\square^{Z_i},u_i)\hookrightarrow (\square^{Z},v)$ in $\overline{X}$, we would have embeddings between basics of the same type. This fixes the domain types to be $\square^{\,1,0}$ and the target types to be $\square^{\,0,1}$. Then, looking at maps in $\int_{\overline{X}}\E_x$ lying over the unique active map $\langle 2\rangle \to \langle 1\rangle$, one would need the square
    $$
    \begin{tikzcd}
      \displaystyle \Emb\big(\bigsqcup_{i=1,2}\square^{\, 0,1}, \square^{\,0,1}\big) \ar[rr, "\upzeta^*"]\ar[d] &&  \displaystyle \Emb \big(\bigsqcup_{i=1,2}\square^{\, 1,0}, \square^{\,0,1}\big)\ar[d]\\
      \displaystyle \prod_{i=1,2} \Emb(\square^{\, 0,1}, \square^{\,0,1})\ar[rr, "\upzeta^*"'] && \displaystyle \prod_{i=1,2}  \Emb(\square^{\, 1,0}, \square^{\,0,1})
    \end{tikzcd}
    $$
    to be a homotopy pullback in order for $\int_{\overline{X}}\E_x\longrightarrow \overline{X}$ to be cocartesian. This is clearly not the case, since the lower horizontal map is a weak homotopy equivalence by Proposition \ref{prop: Scanning with one basic}, while the upper horizontal map is manifestly not so (one vertex is empty and the other is not).
\end{remark}

\begin{lemma}\label{lem: Homotopy fibers of directed integral} Given an open embedding $v\colon \square^{\,p,q}\hookrightarrow X$, the weak $\infty$-operad $\E_{p,q}$ sits in the following homotopy pullback diagram
    \[
        \begin{tikzcd}
         \E_{p,q}\ar[rr]\ar[d]\ar[rrd, phantom, "\overset{\mathsf{h}\;\;}{\lrcorner}" description] &[14mm]&   \vec{\int}\!_{\overline{X}}\,\E_x \ar[d]   \\
          \Fin_*\ar[rr,"\id\times (\square^{\,p,q}\text{,}v)"'] && \Fin_*\times\overline{X}  
        \end{tikzcd}.
    \]
\end{lemma}
\begin{proof} The objects of the homotopy pullback, that we denote $\EuScript{X}$ during the proof, can be identified with those $(\upxi_i)_i\in \vec{\int}\!_{\overline{X}}\,\E_x$ such that the target of $\upxi_i$ is $(\square^{\,p,q},v)$ for any $i\in \langle n\rangle^{\circ}$. Regarding mapping spaces, we have homotopy pullback squares
$$
\begin{tikzcd}
    \EuScript{X}\big((\upxi_i)_i,(\upxi'_j)_j\big) \ar[r] \ar[d] \ar[rd, phantom, "\overset{\mathsf{h}\;\;}{\lrcorner}" description] &[10mm] \displaystyle \coprod_{\upbeta}\prod_{j\in \langle n'\rangle^{\circ}}\Emb\big(\bigsqcup_{i\in\upbeta^{-1}(j)}\square^{Z_i},\square^{Z'_j}\big) \ar[d,"(\upxi'_j)_*\,\cdot\, (\mathsf{inc}_i)^*"]\\[4mm]
    \displaystyle \hspace{-8mm}\coprod_{\upbeta\colon \langle n\rangle \to \langle n'\rangle}\!\!\!\!* \ar[r,"(\upxi_i)_i"'] & \displaystyle \coprod_{\upbeta}\prod_{i\in\upbeta^{-1}(\langle n'\rangle^{\circ})}\Emb(\square^{Z_i},\square^{\,p,q})
\end{tikzcd}.
$$
Therefore, inspection shows that $\EuScript{X}$ is nothing but $\gE_{\square^{\,p,q}}^{\otimes}$, which is equivalent to $\E_{p,q}$ by Proposition \ref{prop: Coincidence of little (p,q)-disc operads}.
\end{proof}

\begin{remark}\label{rem: Unstraightening of diagram of operads} 
    Writing $\uppi\colon \vec{\int}\!_{\overline{X}}\,\E_x\xrightarrow{\quad} \overline{X}$ to denote the canonical map in Lemma \ref{lem: Directed integral is cocartesian fibration}, we obtain, by straightening of $\uppi$ and Lemma \ref{lem: Homotopy fibers of directed integral}, a diagram 
    $$
    \E_{\;\!\scalebox{0.4}{$(-)$}}\colon \overline{X}\xrightarrow{\;\;\quad\;\;}\Cat_{\infty}, \qquad (\square^{\,p,q},u)\xmapsto{\;\;\quad\;\;} \E_{p,q}.$$ 
    Let us argue that $\E_{\;\!\scalebox{0.4}{$(-)$}}$ factors through the composition of forgetful functors
    $$
    \Opd_{\infty} \xrightarrow{\;\;\quad\;\;} (\Cat_{\infty})_{/\Fin_*}\xrightarrow{\;\;\quad\;\;} \Cat_{\infty}.
    $$
    In other words, what we really have is a diagram of $\infty$-operads  $\E_{\;\!\scalebox{0.4}{$(-)$}}\colon \overline{X}\xrightarrow{\quad}\Opd_{\infty}$. First, each value of the diagram comes with a natural projection into $\Fin_*$, as it is manifest from Lemma \ref{lem: Homotopy fibers of directed integral}.  By observing the $\uppi$-cocartesian lifts provided in the proof of Lemma \ref{lem: Directed integral is cocartesian fibration}, we deduce:
    \begin{itemize}
        \item The functor $\upzeta_!\colon\E_{p,q}\xrightarrow{\quad}\E_{p'\!,q'} $ induced between homotopy fibers of $\uppi$ by a morphism $\upzeta\colon (\square^{\,p,q},u)\hookrightarrow (\square^{\,p'\!,q'},u')$ in $\overline{X}$ is compatible with the projections into $\Fin_*$; note that the projection to $\Fin_*$ of the $\uppi$-cocartesian lift of $\upzeta$ is $\id_{\langle n\rangle}$.
        \item The functor $\upzeta_!\colon\E_{p,q}\xrightarrow{\quad}\E_{p'\!,q'}$ induced by $\upzeta\colon (\square^{\,p,q},u)\hookrightarrow (\square^{\,p'\!,q'},u')$ preserves inert maps (recall that Lemma \ref{lem: Homotopy fibers of directed integral} computes the homotopy fibers with their weak $\infty$-operad structures). This follows from the commutative squares in $\vec{\int}\!_{\overline{X}}\,\E_x$
        $$ 
        \begin{tikzcd}
        (\upxi_i)_{i\in \langle n\rangle^{\circ}} \ar[rrrr,"(\id_{\langle n\rangle}\text{,}\,(\id_{i})_{i}\text{,}\,\upzeta)"]\ar[rrrr,"\uppi\text{-cocart.lift}"'] \ar[d, shift right=4,"(\uprho\text{,}\,(\id_{\scalebox{0.7}{$\uprho^{-1}(j)$}})_j\text{,}\, \id_{u})"'] &[6mm]&&&   (\upzeta\cdot\upxi_i)_{i\in \langle n\rangle^{\circ}} \ar[d, shift right=4,"(\uprho\text{,}\,(\id_{\scalebox{0.7}{$\uprho^{-1}(j)$}})_j\text{,}\, \id_{u'})"]\\[8mm]
        (\upxi_{\uprho^{-1}(j)})_{j\in \langle m\rangle^{\circ}} \ar[rrrr,"(\id_{\langle m\rangle}\text{,}\,(\id_{\scalebox{0.7}{$\uprho^{-1}(j)$}})_{j}\text{,}\,\upzeta)"']\ar[rrrr,"\uppi\text{-cocart.lift}"] &&&&   (\upzeta\cdot\upxi_{\uprho^{-1}(j)})_{j\in \langle m\rangle^{\circ}}
        \end{tikzcd},
        $$ 
        where the horizontal maps are $\uppi$-cocartesian lifts of $\upzeta$, the vertical maps are inert morphisms over an inert map $\uprho\colon \langle n\rangle\to \langle m\rangle$, and we have denoted $\id_{i}\equiv\id_{(\square^{Z_i},u_i)}$, $\id_{u}=\id_{(\square^{\,p,q},u)}$ and $\id_{u'}=\id_{(\square^{\,p'\!,q'},u')}$.
    \end{itemize}
    This means that $\E_{\;\!\scalebox{0.4}{$(-)$}}$ sends morphisms in $\overline{X}$ to morphisms of $\infty$-operads, thus concluding our claim. Furthermore, the map of weak $\infty$-operads $\vec{\int}\!_{\overline{X}}\,\E_x\xrightarrow{\quad} \gE_X^{\otimes}$ induces a canonical cocone 
    $$
    \begin{tikzcd}
        \overline{X} \ar[rrrd, bend left=20,"\E_{\;\!\scalebox{0.4}{$(-)$}}"]\ar[rd, hookrightarrow]\\[-4mm]
        & \overline{X}\,\!^{\triangleright}\ar[rr] &&\Opd_{\infty}\\[-4mm]
        \{\infty\}\ar[ru, hookrightarrow] \ar[rrru, bend right=20,"\gE^\otimes_X"']
    \end{tikzcd}.
    $$ 
\end{remark}

\begin{prop} 
\label{prop: Integral is a weak approx} 
The canonical map  $\vec{\int}\!_{\overline{X}}\,\E_x\xrightarrow{\quad}\gE_X^{\otimes}$ is a weak approximation.
\end{prop}
\begin{proof}
In light of Diagram (\ref{eqt: weak opd structure on directed Integral}) and Remark \ref{rem: weak operad structure on the pullback}, the conclusion holds if we can prove: $(i)$ the diagonal map $\Delta\colon \Fin_*\times \overline{X}\longmapsto \overline{X}\,\!^{\amalg}$ is a strong approximation, and $(ii)$ the evaluation functor $\mathsf{ev}_0\colon (\overline{X}\,\!^{\amalg})^{\Delta}_{\mathsf{p}=\id} \longrightarrow \overline{X}\,\!^{\amalg}$ is a weak approximation. As for $(i)$, essentially the same argument provided in the proof of \cite[Proposition 6.2.11]{harpaz_little_nodate} works in this case as well. Just note that $(\overline{X}\!\,^{\amalg})^{\inert}_{\langle 1 \rangle} :=\overline{X}\!\,^{ \simeq}=: \overline{X}\!\,^{\inert}\) and $\Delta^{-1}( \overline{X}\!\,^{\amalg})^{\inert}_{\langle 1 \rangle}\simeq \overline{X}\!\,^{ \simeq}$ by construction, and that $\overline{\square}\,\!^{\, p,q}$ is weakly contractible. Regarding $(ii)$, let us check the two conditions in Definition \ref{defn: weak and strong approximation} for the morphism of weak $\infty$-operads $\mathsf{ev}_0$.

First, we demonstrate that $\mathsf{ev}_0^{-1}(\overline{X}\,\!^{\simeq})\longrightarrow \overline{X}\,\!^{\simeq}$ is an $\infty$-localization. Since the target is an $\infty$-groupoid, it suffices to apply Quillen's theorem A to show that this is the case. Letting $(\square^{Z},v)\in \overline{X}\,\!^{\simeq}$, we conclude this part by just observing that there is an obvious identification $\mathsf{ev}_0^{-1}(\overline{X}\,\!^{\simeq})_{\square^{Z}/} \simeq \overline{X}_{\square^{Z}/}$, and that the latter $\infty$-category is contractible.

Finally, let us show that, for any $\big(\widetilde{\upxi}_j\colon (\square^{\widetilde{Z}_j},\widetilde{u}_j)\hookrightarrow (\square^{\widetilde{Z}\,\!_j'},\widetilde{u}\,\!'_j)\big)_{j\in \langle m\rangle^{\circ}} $ in $(\overline{X}\,\!^{\amalg})^{\Delta^1}_{\mathsf{p}=\id}$, the homotopy fibers of 
\begin{equation}\label{eqt: Second condition approx for modified arrow cat}
   \begin{tikzcd}
        \big((\overline{X}\,\!^{\amalg})^{\Delta^1}_{\mathsf{p}=\id}\big)^{\act}_{/(\widetilde{\upxi}_j)_j} \ar[rr] && \big(\overline{X}\,\!^{\amalg,\,\act}\big)_{/\{\square^{\widetilde{Z}_j}\}_j}
    \end{tikzcd}
\end{equation}
are contractible. Consider an object $(\upbeta,(\uppsi_i)_i)\colon\{(\square^{Z_i},u_i)\}_i\longrightarrow \{(\square^{\widetilde{Z}_j},\widetilde{u}_j)\}_j $ of the $\infty$-category on the right hand side, where $\uppsi_i\colon (\square^{Z_i},u_i)\hookrightarrow (\square^{\widetilde{Z}_{\upbeta(i)}},\widetilde{u}_{\upbeta(i)})$ is a morphism in $\overline{X}$ for each $i$. We can handle one component at a time, and hence we can assume without loss of generality that $m=1$. In this case, the homotopy fiber of (\ref{eqt: Second condition approx for modified arrow cat}) over $(\upbeta,(\uppsi_i)_i)$ can be identified with $\prod_i {\overline{\square}\!\,^{\widetilde{Z}'}}_{\square^{Z_i}/}$, since $\overline{X}\,\!_{/\square^{Z}}\simeq \overline{\square}\!\,^{Z}$, and hence it is contractible.
\end{proof}

We are finally in position to prove Proposition \ref{prop: Strong approximation}.

\begin{prop}\label{prop: Strong approximation} For any \(\infty\)-operad $\EuScript{P}$, the canonical cocone $\upgamma\colon \E_{\;\!\scalebox{0.4}{$(-)$}}\Longrightarrow \gE^{\otimes}_X$ from Remark \ref{rem: Unstraightening of diagram of operads} induces an equivalence of $\infty$-categories
$$
\upgamma^* \colon \Alg_{\gE_{X}}(\EuScript{P})\xrightarrow{\;\;\;\sim\;\;\;} \underset{x\in \overline{X}\!\,^{\op}}{\holim}\Alg_{\mathbbst{E}_{x}}(\EuScript{P}).
$$    
\end{prop}
\begin{proof}
   We will argue that this is the case in two steps: first, by showing that the analogous statement holds for monoids, and then by deducing the general case.

  For monoids, we can factor the map in the statement along $\vec{\int}\!_{\overline{X}}\,\E_x\longrightarrow \gE^{\otimes}_X$ obtaining:
  $$
  \begin{tikzcd}
      \Mon_{\gE^{\otimes}_X}(\Spc) \ar[rr, "(i)"] && \Mon^{\lc}_{\scalebox{0.7}{$\vec{\int}$}_{\overline{X}}\,\E_x}(\Spc) \ar[rr,"(ii)"] && \underset{x\in \overline{X}\,\!^{\op}}{\holim}\,\Mon_{\E_x}(\Spc),
  \end{tikzcd}
  $$
  where the local constancy condition here means that we are considering monoids over $\vec{\int}\!_{\overline{X}}\,\E_x$ which send $\uppi$-cocartesian arrows to equivalences. The map $(i)$ is an equivalence due to Proposition \ref{prop: Integral is a weak approx} and Lemma \ref{lem:Harpaz weak approx}. For $(ii)$, we observe that there is an equivalence of $\infty$-categories 
  $$
  \textstyle \Fun^{\uppi\text{-}\mathsf{cocart}}\big(\vec{\int}\!_{\overline{X}}\,\E_x,\Spc\big)\xrightarrow{\;\;\;\sim\;\;\;} \underset{x\in\overline{X}\!\,^{\op}}{\holim}\, \Fun(\E_x,\Spc)
  $$
  coming from the fact that $\uppi\colon\vec{\int}\!_{\overline{X}}\,\E_x\longrightarrow \overline{X}$ is the unstraightening, and hence the oplax colimit, of $\E_{\;\!\scalebox{0.4}{$(-)$}}\colon \overline{X}\longrightarrow \Cat_{\infty}$ (see \cite[Theorem 7.4]{gepner_lax_2017}). Recall that the homotopy colimit of a diagram $F\colon \EuScript{J}\longrightarrow \Cat_{\infty}$ can be computed as the $\infty$-localization at the cocartesian morphisms of the unstraightening of $F$. Therefore, it suffices to check that this last equivalence restricts to monoids. This holds because the last equivalence of $\infty$-categories is induced by restriction along the morphisms of weak $\infty$-operads $(\E_x\longrightarrow \vec{\int}\!_{\overline{X}}\,\E_x)_x$ and those restriction maps preserve monoids.

  Moving to the general case, we observe that the cocone $\upgamma\colon \E_{\;\!\scalebox{0.4}{$(-)$}}\Longrightarrow \gE^{\otimes}_X$ induces a map of $\infty$-operads 
  \begin{equation}\label{eqt: Assembly map for gEX at the level of oo-operads}
  \underset{x\in \overline{X}}{\hocolim}\,\E_x\xrightarrow{\;\;\quad\;\;} \gE^{\otimes}_X.
  \end{equation}
  Thanks to \cite[Proposition 4.1.23]{harpaz_little_nodate} and the monoid case above, we are reduced to check that the map (\ref{eqt: Assembly map for gEX at the level of oo-operads}) is essentially surjective. This is clearly the case since, for any color $x=(\square^{Z},u)$ of $\gE_X$, the composition (of maps of weak $\infty$-operads)
  $
  \textstyle \E_{x} \xrightarrow{\quad}\vec{\int}\!_{\overline{X}}\,\E_x\xrightarrow{\quad}\gE_X^{\otimes}
  $
  sends the color $\square^{Z}$ of $\mathbbst{E}_x$ to $x=(\square^{\,Z},u)$.
\end{proof}

\begin{lemma}\label{lem: Seifert van Kampen for barX} Let $X$ be a manifold with corners and let $\mathcal{U}\hookrightarrow \open(X)$ be a hypercover of $X$. Then, there is an equivalence 
$
\underset{V\in \mathcal{U}}{\hocolim}\,\overline{V}\simeq \overline{X}
$
of right fibrations over $\Basic_n$.
    
\end{lemma}

\begin{proof}
Through the equivalence $\mathsf{RFib}_{\Basic_n}\simeq \Fun(\Basic_n^{\op},\Spc)$, and since homotopy colimits in functor categories are computed pointwise, we are reduced to showing that the canonical map 
$$
\underset{V\in\mathcal{U}}{\hocolim}\,\Emb(\mathbbst{R}^{Z},V)\longrightarrow\Emb(\mathbbst{R}^{Z},X)
$$
is a weak homotopy equivalence for any $Z\in \Prect(X)$. Recall that $\overline{X}\to \Basic_n$ is the right fibration associated with the presheaf $\Emb(-,X)\colon \Basic_n^{\op}\to \Spc$. Due to Proposition \ref{prop: Scanning with one basic}, one can equivalently check that $\hocolim_V\Fr^{\,i}_{[V\cap Z]}(V_{[V\cap Z]}) \to \Fr^{\,i}_{[Z]}(X_{[Z]})$ is a weak homotopy equivalence. Using the commutative square
$$
\begin{tikzcd}
    \underset{V\in\mathcal{U}}{\hocolim}\,\Fr^{\,i}_{[V\cap Z]}(V_{[V\cap Z]}) \ar[rr]\ar[d] && \Fr^{\,i}_{[Z]}(X_{[Z]})\ar[d] \\
     \underset{V\in\mathcal{U}}{\hocolim}\,V_{[V\cap Z]} \ar[rr] && X_{[Z]}
\end{tikzcd}
$$
we will see that this is the case. The bottom horizontal map is a weak homotopy equivalence by \cite[Theorem 1.3]{dugger_topological_2004} and hence, it suffices to check that the map between homotopy fibers over any point is a weak homotopy equivalence. On the one hand, it is clear that the homotopy fiber of $\Fr^{\, i}_{[Z]}(X_{[Z]})\to X_{[Z]}$ at any point is the fiber group $\mathsf{GL}(\mathbbst{R}^Z)$. On the other hand, to compute the homotopy fiber of the left vertical map, observe that $\Fr^{\, i}_{[Z']}(X_{[Z']})\to X_{[Z']}$ exhibits $X_{[Z']}$ as the homotopy quotient $\Fr^{\, i}_{[Z']}(X_{[Z']})//\mathsf{GL}(\mathbbst{R}^{Z'})$. Therefore, by a commutation of homotopy colimits, we obtain a fiber sequence
$$
\begin{tikzcd}
   \mathsf{GL}(\mathbbst{R}^Z)\ar[d]\ar[rr]\ar[rrd, phantom, "\overset{\mathsf{h}\;\;\;}{\lrcorner}"] && \underset{V\in\mathcal{U}}{\hocolim}\,\Fr^{\,i}_{[V\cap Z]}(V_{[V\cap Z]}) \ar[d] \\
    *\ar[rr] && \big(\underset{V\in\mathcal{U}}{\hocolim}\,\Fr^{\,i}_{[V\cap Z]}(V_{[V\cap Z]})\big)//\mathsf{GL}(\mathbbst{R}^Z)
    \ar[rr,"\sim"] &&[-1.3cm] \underset{V\in\mathcal{U}}{\hocolim}\,V_{[V\cap Z]}
\end{tikzcd}
$$
for any choice of basepoint.  
\end{proof}

\begin{thm}\label{thm: Hyperdescent for cbl fact} Let $X$ be a manifold with corners and let $\mathcal{U}\hookrightarrow \open(U)$ be an hypercover of $U\in \open(X)$. Then, the diagram of operads $\mathcal{U}^{\triangleright}\to \Opd_{\infty}$, $V\mapsto \gE_V$, induces an equivalence of \(\infty\)-categories
$$
\Alg_{\gE_U}(\EuScript{V})\xrightarrow{\;\;\sim\;\;}\underset{V\in\mathcal{U}^{\op}}{\holim}\Alg_{\gE_{V}}(\EuScript{V}).
$$
In other words, $\Alg_{\gE_{\scalebox{0.4}{$(-)$}\!}}(\EuScript{V})\colon \open(X)^{\op}\to \Cat_{\infty}$ is an hypersheaf. The same holds for its symmetric monoidal enhancement $\ALG_{\gE_{\scalebox{0.4}{$(-)$}\!}}(\EuScript{V})$. 

Assuming Hypothesis \ref{hyp: Good supply of Weiss covers}, $\Fact^{\cbl}_{\scalebox{0.4}{$(-)$}\!}(\EuScript{V})\colon \open(X)^{\op}\to \Cat_{\infty}$ is also a hypersheaf, as well as its symmetric monoidal enhancement $\FACT^{\cbl}_{\scalebox{0.4}{$(-)$}\!}(\EuScript{V})$.
\end{thm}
\begin{proof}
Just consider the following chain of equivalences
\begin{align*}
 \underset{V\in\mathcal{U}^{\op}}{\holim}\Alg_{\gE_V}(\EuScript{V}) &\overset{(i)}{\simeq} \underset{V\in \mathcal{U}^{\op}}{\holim}\,\underset{x\in \overline{V}\!\,^{\op}}{\holim}
\Alg_{\mathbbst{E}_x}(\EuScript{V})\\
&\overset{(ii)}{\simeq} \holim\big((\underset{V\in \mathcal{U}}{\hocolim}\,\overline{V})^{\op}\xrightarrow{\; \Alg_{\mathbbst{E}_{\;\!\scalebox{0.4}{$(-)$}\!}}(\EuScript{V})\;}\Cat_{\infty}\big)\\
&\overset{(iii)}{\simeq} \holim\big(\overline{U}\!\,^{\op}\xrightarrow{\; \Alg_{\mathbbst{E}_{\;\!\scalebox{0.4}{$(-)$}\!}}(\EuScript{V})\;}\Cat_{\infty}\big)\\
&\overset{(iv)}{\simeq} \Alg_{\gE_{U}}(\EuScript{V}) 
 ,
\end{align*}
where we have applied: Proposition \ref{prop: Strong approximation} in $(i)$ and $(iv)$; the nested homotopy limit formula of Lemma \ref{lem: colims over cocart fibrations} applied to $F\colon \mathcal{U} \longrightarrow \Cat_{\infty}$, $V \longmapsto \overline{V}$, and the diagram
$$
\overline{D}\coloneqq \Alg_{\mathbbst{E}_{\;\!\scalebox{0.4}{$(-)$}\!}}(\EuScript{V}) \colon (\underset{V\in \mathcal{U}}{\hocolim}\,\overline{V})^{\op} \xrightarrow{\;\;\quad \;\;} \Cat_{\infty}
$$ 
in $(ii)$; and Lemma \ref{lem: Seifert van Kampen for barX} in $(iii)$.  

Combine hyperdescent for $\Alg_{\gE_{\scalebox{0.4}{$(-)$}\!}}(\EuScript{V})$ with Lemma \ref{lem:cbl Fact as operadic algebras} to conclude the claim about constructible factorization algebras.
\end{proof}

\paragraph{From local to global additivity.} Finally, bringing together the derived additivity for (generalized) Swiss-cheese operads and (hyper)descent of constructible factorization algebras over manifolds with corners, we obtain the main result of this work:

 \begin{thm}\label{thm: Global additivity} Let $\EuScript{V}$ be a presentable symmetric monoidal \(\infty\)-category and let $X$, $Y$ be two manifolds with corners satisfying Hypothesis \ref{hyp: Good supply of Weiss covers}. Then, there is an equivalence of \(\infty\)-categories 
    $
    \Fact^{\cbl}_{X\times Y}(\EuScript{V}) \simeq \Fact^{\cbl}_{X}(\FACT^{\cbl}_{Y}(\EuScript{V}))
    $
     induced by pushforward of factorization algebras along $\uppi_X\colon X\times Y\to X$. Furthermore, one can leverage this to a symmetric monoidal equivalence
    $$
   \FACT^{\cbl}_{X\times Y}(\EuScript{V}) \simeq \FACT^{\cbl}_{X}\big(\FACT^{\cbl}_{Y}(\EuScript{V})\big).
    $$
 \end{thm}

\begin{proof}
Let $\mathcal{V}\hookrightarrow \open(X) $ be an open cover of $X$ by local models $\{\R^{p,q}\}_{p+q=n}$ whose finite intersections are also local models (its existence is guaranteed by the proof of Proposition \ref{prop: Good supply of Weiss covers}), and $\mathcal{W} \hookrightarrow \open(Y)$ be an open cover of $Y$ with the same property. Now, using hyperdescent of constructible factorization algebras, we can reduce the problem to a local case where the additivity of factorization algebras over $\R^{p,q}$ applies. Concretely, one has
\begin{align*}
    \Fact^{\cbl}_{X} \big( \FACT^{\cbl}_{Y}( \EuScript{V}) \big) & \overset{(i)}{\simeq} \Fact^{\cbl}_{X} \big( \underset{W\in \mathcal{W}^{\op}}{\holim} \FACT^{\cbl}_{W}( \EuScript{V}) \big) \\
   & \overset{(ii)}{\simeq} \underset{W\in\mathcal{W}^{\op}}{\holim}  \Fact^{\cbl}_{X} \big( \FACT^{\cbl}_{W} ( \EuScript{V}) \big)  \\
    & \overset{(iii)}{\simeq} \underset{W\in\mathcal{W}^{\op}}{\holim} \big( \underset{V\in\mathcal{V}^{\op}}{\holim}  \Fact^{\cbl}_{V} 
 \big( \FACT^{\cbl}_{W} ( \EuScript{V}) \big)\big)   \\
    & \overset{(iv)}{\simeq} \underset{W\in\mathcal{W}^{\op}}{\holim} \big( \underset{V\in\mathcal{V}^{\op}}{\holim}
  \Fact^{\cbl}_{V \times W}( \EuScript{V}) \big) \\
  & \overset{(v)}{\simeq} \underset{V\times W}{\holim}\Fact^{\cbl}_{V\times W}(\EuScript{V})\\
  & \overset{(vi)}{\simeq} \Fact^{\cbl}_{X\times Y}(\EuScript{V}),
\end{align*}
where: in $(i)$, $(iii)$ and $(vi)$ we use hyperdescent of constructible factorization algebras (Theorem \ref{thm: Hyperdescent for cbl fact}) (note that $\mathcal{V}\times \mathcal{W}\hookrightarrow \open(X\times Y)$ is an open cover); $(ii)$ follows from the general fact
$
\ALG_{\EuScript{O}}(\holim_t \EuScript{V}_t)\simeq \holim_t\ALG_{\EuScript{O}}(\EuScript{V}_t)
$
which holds for any \(\infty\)-operad $\EuScript{O}$ and any diagram of presentable symmetric monoidal \(\infty\)-categories (see Lemma \ref{lem:ALG commutes with lims}); for $(iv)$ we apply the local additivity result (Theorem \ref{thm:SwissCheeseAdditivity}) together with Lemma \ref{lem:cbl Fact as operadic algebras}; $(v)$ is just Fubini for homotopy limits of \(\infty\)-categories. Note that the equivalence we have found between $\Fact^{\cbl}_{X\times Y}(\EuScript{V})$ and $\Fact^{\cbl}_X(\FACT^{\cbl}_{Y}(\EuScript{V})) $ is nothing but the one induced by pushforward along the projection $\uppi_X\colon X\times Y\to X$.    
\end{proof}

\begin{remark} We believe that an analogous hyperdescent result to Theorem \ref{thm: Hyperdescent for cbl fact} can be established, following our strategy, if one relaxes Hypothesis \ref{hyp: Good supply of Weiss covers} at the expense of considering constructible factorization algebras satisfying codescent with respect to Weiss hypercovers. Replicating the proof of Theorem \ref{thm: Global additivity} with such alternative hyperdescent will lead to a global additivity theorem without Hypothesis \ref{hyp: Good supply of Weiss covers}. For manifolds with corners, as we mentioned right after introducing the mentioned hypothesis, we will gain nothing, but this perspective might be interesting for more general conically smooth stratified manifolds. 
\end{remark}

\paragraph{Digression: a word on additivity for flag/linear singularities.} Applying our arguments verbatim, one can deduce simple variations of Theorems \ref{thm:SwissCheeseAdditivity}--\ref{thm:SwissCheeseAdditivity Reformulated}. One such scenario is given by linear stratifications in the sense of \cite[Definition 3.1]{calaque_algebras_2025}, also called flag stratifications in the literature, which satisfy cubical restriction. That is, linear stratifications 
$$
\upchi=\big(X_{0}\subsetneq X_{1}\subsetneq \cdots\subsetneq X_{d}=\R^{n}\big)
$$
of $\R^n$ such that, for any $0\leq j\leq d$, the linear subspace $X_j$ is spanned by a subset of the cartesian basis $\{e_1,\dots,e_n\}$, where $e_k=(0,\dots,0,1^{(k)},0,\dots,0)$. Let us refer to these as \emph{quadrant flags} on $\R^n$ in this digression. 

Replacing euclidean discs by cubes, one can easily modify \cite[Definition 3.8]{calaque_algebras_2025} to obtain the \emph{little} $(\upchi,n)$-\emph{cubes operad} $\mathbbst{E}_{n;\,\upchi}$ associated to the (quadrant) flag $\upchi$ on $\R^n$.\footnote{Notice that the open condition on rectilinear embeddings in this framework would not be enough to isolate the  embeddings respecting the stratifications, in contrast to what happens for corner singularities.} 
Furthermore, we can also modify Definition \ref{defn: Cube p,q operads} to obtain its discrete version $\mathsf{Cube}_{n;\,\upchi}$, which comes with an operad map $\mathsf{Cube}_{n;\,\upchi}\xrightarrow{\quad}\mathbbst{E}_{n;\,\upchi}$. Repeating the proof of Theorem \ref{thm:WeakApproxofE}, one shows that this map exhibits $\mathbbst{E}_{n;\,\upchi}$ as an operadic $\infty$-localization of $\mathsf{Cube}_{n;\,\upchi}$. 

Lastly, there is a canonical bifunctor $\upmu\colon \E_{n;\,\upchi}\times \E_{m;\,\upeta}\xrightarrow{\quad} \E_{n+m;\,\upchi\times \upeta}$ that can be constructed as its analogue $\E_{p,q}\times \E_{r,t}\xrightarrow{\quad}\E_{p+r,q+t} $ using the following trivial observation: two quadrant flags
$$
\R^n_\upchi=\big(X_0\subsetneq \cdots\subsetneq X_d=\R^n\big)\quad \text{and}\quad \R^m_\upeta=\big(Y_0\subsetneq \cdots \subsetneq Y_r=\R^m\big)
$$
induce a \emph{double} quadrant flag $\upchi\times \upeta$ on $\R^{n+m}=\R^n\times \R^m$ given by
$$
\upchi\times \upeta=\big(X_{i}\times Y_{j}\subseteq X_{i'}\times Y_{j'}\big)_{\substack{0\leq i\leq i'\leq d\\0\leq j\leq j'\leq r}} .
$$

Once we adapt the operads $\mathbbst{E}_{n;\,\upchi}$ and $\mathsf{Cube}_{n;\,\upchi}$ to include the double quadrant flags, Theorem \ref{thm:WeakApproxofE} and Theorem \ref{thm:WreathProdApprox} can be redone to cover those cases using the methods we already presented. Therefore, we have all the ingredients to deduce the analogue of Theorem \ref{thm:SwissCheeseAdditivity}. Following its steps verbatim, one arrives at:
\begin{thm}\label{thm: Additivity Quadrant Flags} Let $\R^n_\upchi$ and $\R^m_\upeta$ be quadrant flags. Then, the canonical bifunctor $\upmu\colon \E_{n;\,\upchi}\times\E_{m;\,\upeta}\xrightarrow{\quad}\E_{n+m;\,\upchi\times\upeta}$ exhibits $\E_{n+m;\,\upchi\times \upeta}$ as a tensor product. In particular, $\upmu$ induces a symmetric monoidal equivalence
$$
\ALG_{\mathbbst{E}_{n;\,\upchi}}\big(\ALG_{\mathbbst{E}_{m;\,\upeta}}(\EuScript{V})\big)\simeq \ALG_{\mathbbst{E}_{n+m;\,\upchi\times\upeta}}(\EuScript{V})
$$
for any symmetric monoidal $\infty$-category $\EuScript{V}$.
\end{thm}

\begin{remark}\label{rem: Local additivity mixing local models} 
It is also possible to combine corner singularities and quadrant flags together to equip the local models $\R^{p,q}=\R^{p}_{\phantom{\geq 0}}\!\!\!\!\times \R_{\geq 0}^{q}$ with singularities also depending on the factor $\R^{p}$. Let us denote those stratified spaces by  $\R^{p,q}_{\upchi}$, where $\upchi$ is a quadrant flag on $\R^p$, and by $\mathbbst{E}_{p,q;\,\upchi}$ the corresponding operad. The previous reasoning applies to these generalized operads as well: there is a canonical bifunctor 
$$
\E_{p,q;\,\upchi}\times \E_{r,t;\,\upeta}\xrightarrow{\;\;\quad\;\;}\E_{p+r,q+t;\,\upchi\times \upeta}
$$
which exhibits the target $\infty$-operad as a tensor product.

\end{remark}

This line of argument raises two natural questions: 
\begin{itemize}
    \item[$(a)$] Is the \emph{quadrant} hypothesis on the linear stratifications a real restriction? 
    \item[$(b)$] Can one globalize Theorem \ref{thm: Additivity Quadrant Flags} to obtain a version of Theorem \ref{thm: Global additivity}?
\end{itemize}

For $(a)$, the answer is: not really, it is just a convenient point-set condition to allow the same arguments as in the corner case to carry over. Since Lurie's Boardman-Vogt tensor product of $\infty$-operads preserves equivalences (unlike the underived/ordinary Boardman-Vogt tensor product), it suffices to show that two linear stratifications $\upchi_0$ and $\upchi_1$ of $\R^n$ with the same number of subspaces, which also share dimensions, yield equivalent ($\infty$-)operads $\mathbbst{E}_{n;\,\upchi_0}\simeq \mathbbst{E}_{n;\,\upchi_1}$. We can see that this holds due to the following reasons. Notice that there is natural equivalence of $\infty$-operads $\mathbbst{E}_{n;\, \upchi_i}\simeq \gE_{\R^n_{\upchi_i}}$, where we use spaces of open embeddings of conically smooth manifolds \cite{ayala_local_2017} to build $\gE_{\R^n_{\upchi_i}}$ thanks to Lemma \ref{lem: Flags are conical-basics} below (see the proof of Theorem \ref{thm: Global additivity max generality}). Using that there exists a linear map yielding $\R^n_{\upchi_0}\cong \R^n_{\upchi_1}$, we obtain $\gE_{\R^n_{\upchi_0}}\simeq \gE_{\R^{n}_{\upchi_1}}$, and so $\mathbbst{E}_{n;\,\upchi_0}\simeq \mathbbst{E}_{n;\,\upchi_1}$.

For a more down-to-earth reason for the existence of an equivalence $\mathbbst{E}_{n;\,\upchi_0}\simeq \mathbbst{E}_{n;\upchi_1}$, observe that the posets of strata associated to $\upchi_0$ and $\upchi_1$ are isomorphic under the previous combinatorial restriction. Additionally, the spaces of multimorphisms of both $\mathbbst{E}_{n;\,\upchi_0}$ and $\mathbbst{E}_{n;\,\upchi_1}$ have the homotopy type of configuration spaces of points living in the strata of $\upchi_0$ and $\upchi_1$ respectively, and, under the previous conditions on $\upchi_0$ and $\upchi_1$, the strata of both stratifications match. However, this argument is not complete, since we have not constructed an operad map (or zigzag of those) exhibiting this equivalence.

Question $(b)$ is more interesting, and we will present two results in this direction (see Theorems \ref{thm: Global additivity max generality}--\ref{thm: Global additivity mixing local models}). To start, let us have a closer look at (quadrant) flags $\R^{n}_{\upchi}$.

\begin{lemma}\label{lem: Flags are conical-basics} 
    Any euclidean space $\R^n$ equipped with a linear stratification $\upchi$ is a conically smooth manifold. More concretely, $\R^n_{\upchi}$ is of the form $\R^k\times \mathsf{C}Z$, where $Z$ is a conically smooth $(n-k-1)$-sphere.  
\end{lemma}
\begin{proof} It would suffice for our purposes to simply observe that $\R^n_{\upchi}$ is a Whitney stratified space (see \cite[Definition 2.5]{nocera_whitney_2023}), and hence a conically smooth manifold by \cite[Proposition 2.10 and Theorem 1.1]{nocera_whitney_2023}. However, these stratified spaces are simple enough that one can apply an induction on the depth to check the claim.

First, notice that, by applying a linear transformation, we can assume that $\R^n_{\upchi}$ is a quadrant flag. Denoting $X_0$ the first subspace of $\upchi$, we need to analyze two cases:
\begin{itemize}
    \item $\mathsf{dim}X_0=0$: In this case, $\R^n_{\upchi}$ can be manifestly seen as $\mathsf{C}Z$, where $Z$ is the unit sphere at the origin $\mathbbst{S}^{n-1}$ with the stratification obtained by intersecting the linear subspaces of $\upchi$ with $\mathbbst{S}^{n-1}$. By induction on the depth of strata, $\mathsf{C}Z$ is the cone of a compact conically smooth manifold.  
    
    \item $\mathsf{dim}X_0>0$: By applying another linear transformation to organize the euclidean factors at our convenience, we can assume that the linear subspaces in the quadrant flag $\R^n_\upchi$ are spanned by subsets of the cartesian basis $\{e_1,\dots,e_n\}$ which include the first $k$-vectors $\{e_1,\dots,e_k\}$, the ones spanning $X_0$. In other words, $\R^n_{\upchi}=\R^k\times \R^{n-k}_{\upchi'}$, where $\R^{n-k}_{\upchi'}$ is a quadrant flag falling into the case above. 
\end{itemize}    
\end{proof}

\begin{remark} Let us observe that $\R^{n+m}_{\upchi\times \upeta}$, for $\R^n_{\upchi}$ and $\R^m_{\upeta}$ two linear stratifications (e.g.\ two quadrant flags), is also a conically smooth manifold of the form $\R^k\times \mathsf{C}Z$ with $Z$ a $(n-k-1)$-sphere. One can either exploit the properties of the open cone construction $\mathsf{C}(-)$, or repeat the argument in Lemma \ref{lem: Flags are conical-basics}. Building on this fact and our previous arguments, one can easily show: $(1)$ there is an equivalence of $\infty$-operads $\mathbbst{E}_{n+m;\,\upchi_0\times\upeta_{0}}\simeq \mathbbst{E}_{n+m;\,\upchi_1\times \upeta_1}$ for any tuple of linear stratifications $(\R^n_{\upchi_i},\R^m_{\upeta_i})_{i=0,1}$, and $(2)$ in the scenario from Remark \ref{rem: Local additivity mixing local models}, there is an equivalence of $\infty$-categories
$$
\Fact^{\cbl}_{\R^{p+r,q+t}_{\upchi\times\upeta}}(\EuScript{V})\simeq \Alg_{\mathbbst{E}_{p+r,q+t;\,\upchi\times\upeta}}(\EuScript{V})
$$
coming from \cite[Corollaries 4.24 and 5.24]{karlsson_assembly_2024}.
\end{remark}

Combining this fact with the ideas employed to obtain Theorem \ref{thm: Global additivity}, we arrive at:
\begin{thm}\label{thm: Global additivity max generality} Let $\EuScript{V}$ be a presentable symmetric monoidal $\infty$-category. Let $X$, $Y$ be two conically smooth stratified spaces with enough good discs in the sense of \cite[Definition 3.34]{karlsson_assembly_2024} admitting open covers by subspaces of the form $\R^{p,q}_{\upchi}$, for some $(p,q)$ and with $\upchi$ a quadrant flag on $\R^p$, whose finite intersections remain so. Then, pushforward along $\uppi_X\colon X\times Y\to X$ induces a symmetric monoidal equivalence
$$
\FACT_{X\times Y}^{\cbl}(\EuScript{V})\simeq \FACT_{X}^{\cbl}\big(\FACT_{Y}^{\cbl}(\EuScript{V})\big).
$$    
\end{thm}
\begin{proof} We follow the same strategy as in the proof of Theorem \ref{thm: Global additivity}, with slight modifications. We reproduce the argument for the reader's convenience. 

Letting $\mathcal{V}$ (resp.\ $\mathcal{W}$) be an open cover of $X$ (resp.\ $Y$) as in the statement, we consider the chain of equivalences:
\begin{align*}
    \FACT^{\cbl}_{X} \big( \FACT^{\cbl}_{Y}( \EuScript{V}) \big) & \overset{}{\simeq} \FACT^{\cbl}_{X} \big( \underset{W\in \mathcal{W}^{\op}}{\holim} \FACT^{\cbl}_{W}( \EuScript{V}) \big) \\[2mm]
   & \overset{}{\simeq} \underset{W\in\mathcal{W}^{\op}}{\holim}  \FACT^{\cbl}_{X} \big( \FACT^{\cbl}_{W} ( \EuScript{V}) \big)  \\[2mm]
    & \overset{}{\simeq} \underset{V\times W}{\holim}  \FACT^{\cbl}_{V} 
 \big( \FACT^{\cbl}_{W} ( \EuScript{V}) \big) \\[2mm]
     & \overset{}{\simeq} \underset{V\times W}{\holim}  \ALG_{\mathbbst{E}_{p,q;\,\upchi}} 
 \big( \ALG_{\mathbbst{E}_{r,t;\,\upeta}} ( \EuScript{V}) \big) \\[2mm]
    & \overset{}{\simeq} \underset{V\times W}{\holim}
  \ALG_{\mathbbst{E}_{p+r,q+t;\,\upchi\times \upeta}}( \EuScript{V}) \\[2mm]
  & \overset{}{\simeq} \underset{V\times W}{\holim} \FACT^{\cbl}_{V\times W}(\EuScript{V}) \\[2mm]
  & \overset{}{\simeq} \FACT^{\cbl}_{X\times Y}(\EuScript{V}),
\end{align*}
 where we have assumed $V\cong \R^{p,q}_{\upchi}$, $W\cong \R^{r,t}_{\upeta}$, and the only modifications with respect to Theorem \ref{thm: Global additivity} are: $(1)$ instead of Theorem \ref{thm: Hyperdescent for cbl fact} to glue constructible factorization algebras, we use \cite[Theorems A and B]{karlsson_assembly_2024}; $(2)$ in the fourth and sixth equivalences, we use \cite[Corollary 4.24 and 5.24]{karlsson_assembly_2024} together with the natural equivalence of operads $\mathbbst{E}_{p,q;\,\upchi}\simeq \gE_{\R^{p,q}_{\upchi}}$ which relies on \cite[Theorem 8.2.6]{ayala_local_2017} and the proof of \cite[Lemma 2.21]{ayala_factorization_2017} (cf.\ Proposition \ref{prop: Coincidence of little (p,q)-disc operads}); $(3)$ for the fifth equivalence we apply Remark \ref{rem: Local additivity mixing local models}.  
\end{proof}

Recall that Whitney stratified spaces are conically smooth manifolds by \cite{nocera_whitney_2023}, yielding a good supply of potential examples to apply Theorem \ref{thm: Global additivity max generality}. Also, notice that (the proof of) Lemma \ref{lem: Flags are conical-basics} can be employed to identify the basic building blocks we admit.

A particular instance of Theorem \ref{thm: Global additivity max generality}, with a cleaner statement, is:
\begin{thm}\label{thm: Global additivity mixing local models} Let $\EuScript{V}$ be a presentable symmetric monoidal $\infty$-category, $X$ be a manifold with corners and let $\R^n_\upchi$ be a quadrant flag. Then, there are equivalences of symmetric monoidal $\infty$-categories
    $$
    \FACT^{\cbl}_{\R^n_\upchi}\big(\FACT^{\cbl}_{X}(\EuScript{V})\big)\simeq \FACT_{\R^n_{\upchi}\times X}^{\cbl}(\EuScript{V})\simeq  \FACT^{\cbl}_{X}\big(\FACT^{\cbl}_{\R^n_\upchi}(\EuScript{V})\big).
    $$
\end{thm}
\begin{proof} 
We apply Theorem \ref{thm: Global additivity max generality} by noticing the following easy facts. On the one hand, any manifold with corners has enough good discs and admits an open cover by $\R^{p,q}$'s, whose finite intersections remain so, by Proposition \ref{prop: Good supply of Weiss covers}. On the other hand, there is no need to cover $\R^n_{\upchi}$, as it is already a building block, and it admits enough good discs since we can consider the full subcategory of $\open(\R^n_{\upchi})$ spanned by convex open subsets in $\R^n$ that can be included in euclidean balls which are basics for the stratification $\R^n_{\upchi}$, and their finite disjoint unions. The condition on euclidean balls simply avoids some pathological examples.
\end{proof}

\begin{remark} The use of descent of constructible factorization algebras from \cite{karlsson_assembly_2024} in Theorem \ref{thm: Global additivity max generality} could be replaced by feeding our arguments in \textsection \ref{sect: Additivity for cbl Fact} with appropriate differential topological results for conically smooth manifolds from \cite{ayala_factorization_2017}. In fact, Theorem \ref{thm: Global additivity mixing local models} might be handled more explicitly by combining embedding spaces for manifolds with corners and $\upchi$-admissible rectilinear embeddings, without appealing to the powerful machinery of conically smooth manifolds.    
\end{remark}

\appendix

\section{Manifolds with corners}\label{sect:Manifolds with Corners}

 This appendix fixes notation and presents adaptations of classical results that we require and could not find in the literature. References on manifolds with corners are relatively sparse and often employ different conventions depending on the authors' objectives. For instance, look at the various notions of maps between such manifolds in \cite{francis-staite_cinfty-algebraic_2024, melrose_calculus_1992}. In our setting, the simplest choices will suffice. We primarily follow \cite{cerf_topologie_1961,michor_manifolds_1980}, as well as the lecture notes \cite{conrad_differential_nodate}, where more detailed discussions can be found.

Let us begin by introducing the version of differentiability for functions between local models (also called quadrants \cite{michor_manifolds_1980} or sectors \cite{conrad_differential_nodate}) that we will use.

\begin{defn} Let $A$, $B$ be arbitrary subsets of $\R^n$, $\R^m$ respectively. A function $f\colon A\to B$ is \emph{(weakly) smooth} if it is continuous and it admits a smooth extension $F\colon U\to \R^m$, i.e. $F\vert_{A}=f$, where $U\subseteq \R^n$ is an open subset containing $A$. A \emph{diffeomorphism} $g\colon A\to B$ is a homeomorphism such that $g$ and $g^{-1}$ are (weakly) smooth.
\end{defn}

Given a (weakly) smooth function $f\colon A\to B$ and a point $x\in A$, we might define its differential $\mathsf{d}_xf\colon \R^n\to \R^m$ as $\mathsf{d}_xf:=\mathsf{d}_xF$ for a smooth extension $F\colon U\to \R^m$ of $ f$. This $\mathsf{d}_xf$ might depend on the choice of a smooth extension. However, if there is an open subset $V\subseteq \R^n$ such that $V\subseteq A \subseteq \overline{V}$, the map $\mathsf{d}_xf$ only depends on $f$ and $x$. This is the case for (weakly) smooth functions between quadrants $\R^{p,q}\to \R^{s,t}$.  

Slightly more generally, a \emph{quadrant} $Q$ is a subset of a finite-dimensional vector space $\mathsf{V}$ of the form 
$$
Q=\left\lbrace x\in\mathsf{V}:\;\;\ell_1(x)\geq0,\dots,\ell_r(x)\geq 0\right\rbrace,
$$
where $\{\ell_i\}_i$ is a linearly independent subset of $\mathsf{Hom}_{\R}(\mathsf{V},\R)$. Given a point $x\in Q$ in a quadrant, we can define its \emph{index} (or \emph{depth}) as the number $\#\{i:\;\ell_i(x)=0\}$. Quite importantly, this index is invariant under diffeomorphism of quadrants (see \cite[Manifolds with corners]{conrad_differential_nodate}). In fact, this implies that the number of linear inequalities describing a quadrant is also diffeomorphism-invariant. Furthermore, given a diffeomorphism $f\colon Q\to Q'$ between quadrants, its differential at the origin $\mathsf{d}_0f\colon \mathsf{V}\to \mathsf{V}'$ is a linear isomorphism which restricts to the quadrants (as subspaces of the corresponding tangent spaces); in particular, it induces a bijection between the sets $\{\ell_i\}_i$ and $\{\ell'_j\}_j$ of linear forms defining $Q$ and $Q'$ respectively\footnote{Notice that we are not claiming that $\ell'_j\circ \mathsf{d}_0f=\ell_i$ on the nose for the corresponding indices $i,j$.}, and sends one (and hence any) vector $v$ satisfying $\ell_i(v)>0$ for all $i$ to a vector satisfying the corresponding strict inequalities in $Q'$. Indeed, linear isomorphisms which restrict to quadrants are characterized by these conditions, and we will denote by $\mathsf{GL}(\R^{p,q})$ the subgroup of $\mathsf{GL(\R^n)}$ spanned by the linear isomorphisms which restrict to the quadrant $\R^{p,q}$. 

Based on this obvious generalization of smoothness (whose soundness is due to Whitney's extension theorem), one can define algebras of smooth functions on (open subsets of) $\R^{p,q}$, and diffeomorphisms between quadrants \cite{conrad_differential_nodate}. Hence, we can give:
\begin{defn}\label{defn:Mfld with corners} A \emph{premanifold with corners} is a locally $\R$-ringed space $X$ that, around each $x\in X$, is locally isomorphic to an open subset of a quadrant $\R^{p_x,q_x}$ equipped with its sheaf of smooth functions. We say that $X$ is a \emph{manifold with corners} if the underlying space is Hausdorff and second-countable\footnote{These hypotheses imply that $X$ is paracompact and metrizable, and thus it admits smooth partitions of unity subordinated to arbitrary open covers \cite[Paracompactness]{conrad_differential_nodate}.}. 
\end{defn}

Any $n$-dimensional manifold with corners $X$ admits a depth (or index) filtration 
$$
 X_{\leq 0}\subseteq X_{\leq 1}\subseteq \dots\subseteq X_{\leq n}=X,
$$
equiv.\ a continuous map $X\to [n]=\{0<\dots<n\}$. That is, $x\in X$ belongs to $X_{\leq k}$ iff $\mathsf{index}(x)\geq n-k$, where $\mathsf{index}(x)$ is the index of $x$ in some (and hence any) local chart containing $x$. However, it turns out to be more convenient to consider a refinement of this filtration that accounts for the incidence of strata in $X$. 

Consider the following set
$$
\Prect(X)=\bigcup_{k\in [n]}\big\{K\subseteq X_{\leq k}\backslash X_{\leq k-1}\text{ connected component}\big\}\footnote{By convention $X_{\leq -1}=\diameter$.}
$$
ordered by: $K'\leq K$ iff $K'$ is contained in $\overline{K}$, the closure of $K$ in $X$. Then, we define a continuous map $\uppi\colon X\to \Prect(X)$ by sending a point to the connected component of strata where it belongs. This map is continuous because
$
\uppi^{-1}\big(\{K': K'\leq K\}\big)=\overline{K}.
$
Such a closure, of an element $K\in \mathsf{P}(X)$, is called a \emph{face} of $X$.

The poset $\Prect(X)$ encodes the local structure of $X$ and the associated incidence data. One way to interpret this assertion is the following: for $x\in X$, the singularity type of any local chart around $x$ is determined by the image of $x$ along $\uppi\colon X\to \Prect(X)$, i.e.\,  the pair $(p,q)$ defining $\R^{p,q}$. Also, if we consider the equivalence relation between local charts (around possibly different points in the same stratum) generated by inclusions, we have:
$$
(U,x)\sim (V,y)\quad\quad \text{ iff } \quad\quad \uppi(x)=\uppi(y).
$$
These observations will be useful when considering constructible factorization algebras over manifolds with corners. For instance, given $Z\in \Prect(X)$, we will denote by $\mathbbst{R}^{Z}$ the local model $\mathbbst{R}^{p,q}$ around any point $z\in Z$, and by $X_{[Z]}$ the subspace of $X$ given by $Z$. We also use the notation $\square^{Z}=\R^{Z}\cap \square^{\,p,q}$, where $\square^{\, p,q}=(-1,1)^p\times [0,1)^{q}\subseteq \R^{p,q}$, and $\R^Z_{[Z]}$ (resp.\ $\square^{Z}_{[Z]}$) to refer to the subspace corresponding to the smallest element/deepest strata in the poset $\mathsf{P}(\R^{Z})$ (resp.\ $\mathsf{P}(\square^Z)$).  

A relevant example for us is the following:
\begin{example}\label{examp: Conf as mfld with corners} Given a manifold with corners $X$ and a function $Z\colon I\to \mathsf{P}(X)$, where $I$ is a finite set, the (ordered) configuration space 
$$
\Conf_{Z}(X)=\prod_{K\in \mathsf{P}(X)}\Conf_{Z^{-1}\{K\}}\big(X_{[K]}\big)
$$
from Notation \ref{notat: Configuration spaces} is again a manifold with corners, indeed, an ordinary manifold without boundary. $\Conf_{Z}(X)$ is a product of configuration spaces of ordinary smooth manifolds since $X_{[K]}$ is a manifold without boundary for any $K\in \Prect(X)$ (see \cite[\textsection 2.5]{michor_manifolds_1980}).

Compare this fact with the honest manifold with corners given by the usual configuration space of $I$-points
$$
\Conf_{I}(X)=\left\lbrace\underline{x}\in X^{\times I}:\; x_i\neq x_j \text{ if }i\neq j\right\rbrace.
$$
(Notice that $\Conf_I(X)$ is an open subspace of $X^{\times I}$ since $X$ is assumed to be Hausdorff). In fact, $\Conf_Z(X)$ can be seen as a subspace of $\Conf_I(X)$ and their collection as $Z$ ranges over functions $\mathsf{Hom}_{\mathsf{Set}}(I,\mathsf{P}(X))$ is related to the set $\mathsf{P}(\Conf_I(X))$. We will not need this fact, and thus we omit further discussion.
\end{example}

\paragraph{Tangent and frame bundles.} Using Definition \ref{defn:Mfld with corners}, one can redo a considerable amount of basic differential geometric constructions (cf.\ \cite{cerf_topologie_1961,conrad_differential_nodate,michor_manifolds_1980}). In particular, we will need the (inner) tangent and frame spaces associated with a manifold with corners:

Let $X$ be an $n$-manifold with corners. Given a point $x\in X$, we can define in the usual way the $\R$-vector space of tangent vectors $\mathsf{T}_xX$ at $x$ using derivations of the algebra of smooth functions or 1-jets of curves (see \cite[Derivations and vector fields]{conrad_differential_nodate} for a detailed discussion). If $\mathsf{index}(x)=q$, we can extract the convex subspace of inner tangent vectors $\mathsf{T}^{\, i}_{x}X\subseteq \mathsf{T}_xX$ corresponding, under the linear isomorphism $\mathsf{T}_xX\cong \R^{p+q}$ associated to any choice of local coordinates, to the quadrant $\R^{p,q}\subseteq \R^{p+q}$ of inner directions around $x$ (see \cite{michor_manifolds_1980} for an alternative description using jets of curves). Gluing these spaces together we obtain the \emph{tangent bundle} $\mathsf{T}X=\bigcup_{x\in X}\mathsf{T}_xX\to X$, a vector bundle/locally trivial bundle with fiber $\R^{n}$ and structural group $\mathsf{GL}(\R^{n})$ in the category of manifolds with corners, and its subspace $\mathsf{T}^{\, i}X=\bigcup_{x\in X}\mathsf{T}^{\, i}_xX\subseteq \mathsf{T}X$ consisting of \emph{inner tangent vectors}. Notice that $\mathsf{T}^{\,i}X$ may not be a manifold with corners and that $\mathsf{T}^{\,i}X\to X$ is not a locally trivial bundle (see \cite[pag. 20]{michor_manifolds_1980}). However, for any $Z\in \mathsf{P}(X)$, the pullback map
$$
\begin{tikzcd}
    \mathsf{T}^{\,i}_{[Z]}X \ar[r, hookrightarrow] \ar[d]\ar[rd, phantom, "\lrcorner"] & \mathsf{T}^{\,i}X \ar[r, hookrightarrow]\ar[d] & \mathsf{T}X \ar[ld]\\
    X_{[Z]} \ar[r, hookrightarrow] & X
\end{tikzcd}
$$
becomes a locally trivial bundle over $X_{[Z]}$ with fiber $\R^{Z}$ in the category of manifolds with corners. See \cite[Equivalence of bundles and O-modules]{conrad_differential_nodate} for formalities.  

Similarly, given a point $x\in X$, we can consider the $\mathsf{GL}(\R^{n})$-torsor of linear frames at $x$, $\mathsf{Fr}_x(X)=\mathsf{Hom}^{\mathsf{iso}}_{\R}(\R^{n},\mathsf{T}_xX)$. By considering only those framings which restrict to an isomorphism of quadrants $\R^{p,q}\cong \mathsf{T}^{\, i}_xX$, we obtain the subspace of inner frames $\mathsf{Fr}^{\, i}_x(X)\subseteq \mathsf{Fr}_x(X)$ at $x$, which is accordingly a $\mathsf{GL}(\R^Z)$-torsor where $Z\in \mathsf{P}(X)$ denotes the stratum containing $x$. Gluing these spaces together, one gets the \emph{frame bundle} $\mathsf{Fr}(X)=\bigcup_{x\in X}\mathsf{Fr}_x(X)\to X$, a $\mathsf{GL}(\R^{n})$-principal bundle in the category of manifolds with corners, and its subspace $\mathsf{Fr}^{\, i}(X)=\bigcup_{x\in X}\mathsf{Fr}^{\, i}_x(X)\subseteq \mathsf{Fr}(X)$ consisting of \emph{inner frames}. Again, for any $Z\in \mathsf{P}(X)$, the pullback map
$$
\begin{tikzcd}
    \mathsf{Fr}^{\,i}_{[Z]}(X_{[Z]}) \ar[r, hookrightarrow] \ar[d]\ar[rd, phantom, "\lrcorner"] & \mathsf{Fr}^{\,i}(X) \ar[r, hookrightarrow]\ar[d] & \mathsf{Fr}(X) \ar[ld]\\
    X_{[Z]} \ar[r, hookrightarrow] & X
\end{tikzcd}
$$
becomes a $\mathsf{GL}(\R^{Z})$-principal bundle over $X_{[Z]}$ in the category of manifolds with corners.

\begin{remark} It is also possible to define jet spaces and jet bundles for manifolds with corners (\cite[\textsection 2.11]{michor_manifolds_1980} or \cite[\textsection II.3]{cerf_topologie_1961}) and hence to endow sets of differentiable mappings between manifolds with corners, and their subsets, with strong and weak $\mathcal{C}^r$-topologies for $1\leq r\leq \infty$ using these. In the sequel, we will be mostly interested in the weak $\mathcal{C}^{\infty}$-topology because of its compatibility with composition, but we will come back to this point in due time.
\end{remark}

\paragraph{Adapted Riemannian structures.} Recall that a Riemannian metric on a manifold with corners $X$ is just a smooth symmetric $2$-tensor $g\in \Upgamma(X,\mathsf{Sym}^2(\mathsf{T}^*X))$ which is positive definite, i.e.\ $g_x(v_x,v_x)>0$ for any non-zero vector $v_x\in \mathsf{T}X$. Notice the usage of the full tangent bundle over $X$ (see \cite[Operations with metrics]{conrad_differential_nodate}).

A fundamental result for our purposes is the existence of \emph{adapted Riemannian metrics} on manifolds with corners:
\begin{thm}(\cite[\textsection I.3.3]{cerf_topologie_1961}) \label{thm:Adapted Riemannian metrics} Let $X$ be a manifold with corners. Then, there exists a Riemannian metric $g$ on $X$ satisfying:
\begin{itemize}
    \item Every face of $X$ is totally geodesic with respect to $g$ (see \cite[(14)]{cerf_topologie_1961}). 

    \item any two different faces $F_1,F_2$ of $X$ are orthogonal with respect to $g$ at the intersection points, i.e.\ for any $x\in F_1\cap F_2$ and any pair of inner vectors $v_x\in \mathsf{T}^{\,i}_{x}F_1$ and $v'_x\in \mathsf{T}^{\,i}_{x}F_2$, we have $g_x(v_x,v'_x)=0$.
\end{itemize}
\end{thm}

From \cite[\textsection I.3.4]{cerf_topologie_1961} or \cite[Lemma 2.10]{alekseevsky_reflection_2007}, we know that the geodesic exponential map for such a Riemannian metric is well-defined and satisfies:
\begin{cor}\label{cor: Exponential mappings} Let $(X,g)$ be a Riemannian manifold with corners as in Theorem \ref{thm:Adapted Riemannian metrics}. Then, there exists an open neighborhood $W$ in $\mathsf{T}^{\,i}X$ of the zero-section such that $\mathsf{exp}\colon W\to X$ is well-defined and satisfies:
\begin{itemize}
    \item  $\mathsf{exp}(0_x)=x$ for any $x\in X$.
    \item  $\mathsf{Exp}:=(\uppi_X,\mathsf{exp})\colon W\to X\times X$ is a diffeomorphism onto an open neighborhood of the diagonal $X\hookrightarrow X\times X $.
    \item $\mathsf{exp}_x\colon W\cap \mathsf{T}^{\,i}_xX\to X$ is a diffeomorphism onto an open neighborhood of $x\in X$.
    \item $W\cap \mathsf{T}^{\, i}_xX$ is the intersection of an open ball in the inner product space $(\mathsf{T}_xX,g_x)$ with a quadrant $Q_x\subseteq \mathsf{T}_xX$.
    \item $\mathsf{exp}$ restricts to the exponential mapping of the induced Riemannian metric on each face. 
\end{itemize}

Moreover, this neighborhood $W$ can be chosen by fixing an appropriate positive continuous function $\uplambda\colon X\to \R_{>0}$ and defining:
$$
W\equiv W(\uplambda)=\left\lbrace v_{x}\in \mathsf{T}^{\,i}X:\;\;||v_{x}||_{g}=\sqrt{g_x(v_x,v_x)}<\uplambda(x)\right\rbrace.
$$
\end{cor}

\begin{remarks}
The second and third points in Corollary \ref{cor: Exponential mappings} describe how nearby points in \(X\) are connected by geodesics: fixing \(x\), every point \(y\) sufficiently close to \(x\) can be reached by a unique geodesic starting at \(x\), with velocity \(w \in W\cap \mathsf{T}^{\, i}_xX \). Moreover, this correspondence varies smoothly with \(x\). The fourth and fifth points are manifestations of the compatibility of this construction with corners/strata. For example, the fifth point can be interpreted as saying: a geodesic whose velocity is tangent to a face remains in that face.
\end{remarks}

Another useful consequence of Theorem \ref{thm:Adapted Riemannian metrics} is the existence of good Weiss covers for manifolds with corners.
\begin{prop}\label{prop: Good supply of Weiss covers} Any manifold with corners $X$ admits a Weiss cover whose elements are finite disjoint unions of quadrants, or in other words basics, and their finite intersections remain so. Similarly, any manifold with corners $X$ have enough good discs in the sense of \cite[Definition 3.34]{karlsson_assembly_2024}.
\end{prop}
\begin{proof} It suffices to construct a base $\mathfrak{B}$ for the topology of $X$ closed under intersections and such that every element $U\in \mathfrak{B}$ is diffeomorphic to $\R^{p,q}$ for some $(p,q)$. The Weiss cover/base of $X$ can be obtained by considering finite disjoint unions of elements in $\mathfrak{B}$.

As Bott-Tu argued in the case of manifolds without boundary \cite[Theorem 5.1]{bott_differential_1982}, considering (strongly) geodesically convex open subsets in $X$ for an appropriate Riemannian metric does the job. In the corner setting, we need to choose a Riemannian metric, as in Theorem \ref{thm:Adapted Riemannian metrics}, and provide more concrete details regarding which subsets form the base. 

We set $\mathfrak{B}$ to be the closure under finite intersections of 
$$
\widetilde{\mathfrak{B}}=\left\{\mathsf{B}_g(z,r):=\mathsf{exp}_{z}\big(\{v_z\in \mathsf{T}^{\,i}_zX:\;||v_z||_{g}<r\}\big)\;\text{ with }z\in X\text{ and }r<\mathsf{min}\{\uplambda(z),\upbeta_z\}\right\}
$$
where $\uplambda\colon X\to \R_{>0}$ is the continuous function of Corollary \ref{cor: Exponential mappings} and $\upbeta_z>0$ is the bound that appears in \cite[\textsection 3.4, Proposition 4.2]{do_carmo_riemannian_1992}. Then, a geodesic ball $\mathsf{B}_{g}(z,r)\in\widetilde{\mathfrak{B}}$ is (strongly) geodesically convex by the same arguments provided for the cited result in \cite{do_carmo_riemannian_1992}\footnote{Notice that do Carmo does not consider corners, but the arguments employed to obtain this result carry over to our setting since they only involve the local-length-minimizing property of geodesics. 
}
and so does any non-empty $U\in \mathfrak{B}$. In particular, for any $U\in \mathfrak{B}$ and any point $x\in U$ of maximal depth, $\mathsf{exp}_{x}^{-1}(U)$ is star-shaped around $0_x$ in $\mathsf{T}^{\,i}_xX$ and $\mathsf{exp}_x$ yields a diffeomorphism between this star-shaped subset and $U$; the exponential map is bijective, due to the geodesically convex condition, and a local-diffeomorphism.

It remains to check that any $U\in \mathfrak{B}$ is diffeomorphic to a basic $\R^{p,q}$. Recall that any star-shaped (around the origin) open subset $V\subseteq \R^n$ is diffeomorphic to $\R^n$ via a map of the form 
 $$
    \begin{tikzcd}[ampersand replacement=\&]
     V\ar[rr,"\simeq"] \&\& \R^n \\[-6mm]
 v \ar[rr, mapsto] \&\& \left(1+\Big(\displaystyle \int_{0}^{||v||}\frac{\mathsf{d}t}{\upphi(tv/||v||)}\Big)^2\right)\cdot v
    \end{tikzcd},
 $$
where $\upphi\colon V\to \R_{\geq 0}$ is a smooth map satisfying $\upphi^{-1}(0)=\R^n\backslash V$ (see \cite{noauthor_mathoverflow--what_2020}). Since this map is radial, i.e.\ of the form $v\mapsto \upmu(v)\cdot v$ with $\upmu\colon V\to \R_{\geq 1}$, it can be used as well to show that any star-shaped open subset $V'\subseteq \R^{p,q}$ is diffeomorphic to $\R^{p,q}$; the existence of a smooth function $\upphi\colon \R^{p,q}\to \R_{\geq 0}$ such that $\upphi^{-1}(0)=\R^{p,q}\backslash V'$ simply follows from the existence of such functions on $\R^n$ by restriction, since any closed subset in $\R^{p,q}$ can be obtained as the intersection of a closed subset in $\R^n$ and $\R^{p,q}$. Using this digression, we obtain
$$
U\xleftarrow[\;\;\mathsf{exp}_x]{\quad\; \simeq\quad\;} \mathsf{exp}_x^{-1}(U)\xrightarrow[]{\quad\;\simeq\quad\;}\mathsf{T}^{\,i}_{x}X\cong \R^{p,q}.
$$
\end{proof}

\paragraph{Fibration lemmas for embedding spaces.} The set of smooth maps between two manifolds with corners can be endowed with various important topologies (see \cite[\textsection 4]{michor_manifolds_1980}). Our focus will be on the \emph{weak} $\mathcal{C}^{\infty}$-\emph{topology} (or $\mathsf{CO}^{\infty}$-topology in \cite{michor_manifolds_1980}), and so we will denote the resulting space of differentiable mappings by $\mathsf{Map}_{\mathcal{C}^{\infty}}(X, Y)$. We favor this topology over the others because it makes composition
$$
\circ\colon \mathsf{Map}_{\mathcal{C}^{\infty}}(X',X'')\times\mathsf{Map}_{\mathcal{C}^{\infty}}(X,X') \xrightarrow{\;\;\quad\;\;}\mathsf{Map}_{\mathcal{C}^{\infty}}(X,X'')
$$
a continuous function for any $X,X',X''$. This is an obvious necessary ingredient to consider topological categories (and operads) of manifolds with corners. In fact, we will extensively use the subspaces $\Emb(X,Y)\subseteq \mathsf{Map}_{\mathcal{C}^{\infty}}(X,Y)$ spanned by \emph{open embeddings}.

\begin{notation}\label{notat:Topological category of manifolds with corners}
 We denote by $\cMfld$ the topological category of $n$-manifolds with corners and open embeddings between them. It admits a symmetric monoidal structure induced by taking disjoint unions. We will denote by $\cgE$ the full (topological) suboperad of $(\cMfld,\amalg)$ spanned by the set of colors $\{\R^{p,q}\}_{p+q=n}$. Its underlying category, denoted $\Basic_n$, is the topological category of \emph{basic} $n$-\emph{manifolds with corners}, or simply \emph{basics}, and open embeddings. 
\end{notation}

In practice, a more refined notation regarding embedding spaces will be employed. Let $X$ be a manifold with corners and $Z\in \mathsf{P}(X)$ be a connected component of strata. Then,  $\Emb(\R^Z,X)$ will refer to the subspace of $\Emb(\R^{p,q},X)$ given by the (open) embeddings $\varphi$ that preserve the strata, $\varphi(\R^{Z}_{[Z]})\subseteq X_{[Z]}$. In particular and equivalently, $\varphi(0)\in X_{[Z]}$. The obvious generalization will be employed for $\Emb(\R^{Z_1}\sqcup\cdots\sqcup \R^{Z_k},X)$.

In the remainder of this appendix, we collect several technical results concerning these embedding spaces and their weak homotopy type.

\begin{lemma}\label{lem: Evaluation at center is a fibration for Embeddings}
    Let $X$ be a manifold with corners, $I$ a finite set, and let $Z\colon I\to \Prect(X)$ be a function. Evaluation at the points $0\in \R^{Z(i)}$ yields a Serre fibration
$$
 \Emb\big(\bigsqcup_{j\in I}\mathbbst{R}^{Z(j)},X\big)\xrightarrow{\quad\quad} \Conf_Z(X).\footnote{See Notation \ref{notat: Configuration spaces}.}
$$
\end{lemma}

\begin{remark} When the finite set $I$ has only one element, an elementary argument due to Kupers (see \cite{noauthor_mathoverflow--fibrations_2023}) can be adapted to show that $\mathsf{ev}_0\colon\Emb(\R^Z, X)\to X_{[Z]}$ is a Serre fibration. Unfortunately, we haven't been able to extend this simple idea to cover arbitrary finite sets.
\end{remark}

To handle the general case, we will adapt the strategy followed by \cite{palais_local_1960}. In \cite[\textsection II.2-5]{cerf_topologie_1961}, more general fibration lemmas were proven for manifolds with corners. Unfortunately, Cerf equips the spaces of embeddings with the strong (or Whitney) $\mathcal{C}^{\infty}$-topology in loc.cit., while we need to consider the weak (or compact-open) $\mathcal{C}^{\infty}$-topology. See \cite[\textsection 4]{palmer_homological_2021} for a clear exposition.

The crucial criterion that we will use, known as Cerf--Palais Lemma, is:
\begin{lemma}(\cite[Proposition 4.7]{palmer_homological_2021})\label{lem: Cerf-Palais} Let $G$ be a topological group and $B$ a $G$-space. If $B$ admits local cross-sections (i.e.\ for any $b_0\in B$, there is an open neighborhood $b_0\in U\subseteq B$ and a continuous map $\upchi\colon U\to G$ such that 
$$
\begin{tikzcd}
    U\ar[r,"\upchi"]\ar[rr, bend right=25,"\mathsf{inc}"'] &[3mm] G\ar[r,"-\cdot\, b_0"] &[3mm] B 
\end{tikzcd}
$$
commutes), any $G$-equivariant continuous map $E\to B$ is locally trivial if we forget the $G$-action.
\end{lemma}

\begin{lemma}\label{lem: Local-cross sections for Configuration spaces} Let $X$ be a manifold with corners, $I$ a finite set, and $Z\colon I\to \Prect(X)$ a function. Denote by $\mathsf{Diff}_{\mathsf{c}}(X)$ the topological group of compactly supported diffeomorphisms of $X$. Then, the configuration space $\Conf_{Z}(X)$ with its canonical $\mathsf{Diff}_{\mathsf{c}}(X)$-action admits local cross-sections.
\end{lemma}
\begin{proof} Let us fix $\underline{z}\in \Conf_{Z}(X)$. When $X$ is an ordinary manifold without boundary, Palais constructed in \cite[Theorem B]{palais_local_1960} a local cross-section $\upchi$ for $\Conf_Z(X)$ on a neighborhood $\underline{z}\in U\subseteq \Conf_{Z}(X)$ as the composition of three maps
$$
\begin{tikzcd}[ampersand replacement=\&]
    U\ar[rd,"u"']\ar[rrr,"\upchi"] \&\&\& \mathsf{Diff}_{\mathsf{c}}(X)\\
    \& \displaystyle \prod_{j\in I}\mathsf{T}_{z_j}X\ar[r,"k"'] \& \Upgamma_{K}(X,\mathsf{T} X) \ar[ru,"E"'] 
\end{tikzcd},
$$
where $K$ is a compact neighborhood of $\bigsqcup_{j\in I}  \{z_j\}\subset X$. The first map $u$ exploits that $\mathsf{Exp}\colon \bigsqcup_{j\in I}\mathsf{T}_{z_j}X \to \bigsqcup_{j\in I} \{z_j\}\times X$ is a local diffeomorphism. The second one, $k$, is a section of $\Upgamma_{K}(X,\mathsf{T}X)\to \prod_{j\in I}\mathsf{T}_{z_j}X$ built using parallel transport, whose definition determines the compact neighborhood $K$. The last term $E$ is the restriction of
$$
\begin{tikzcd}[ampersand replacement=\&]
    \Upgamma(X,\mathsf{T}X) \ar[rr] \&\& \Map_{\mathcal{C}^{\infty}}(X,X)\\[-8mm]
    v\ar[rr, mapsto] \&\& \big(x\mapsto \mathsf{exp}(v_x)\big)
\end{tikzcd}.
$$
In words, \(u\) identifies a neighborhood \(U\) of \(\underline{z} \in \Conf_{Z}(X)\) with a neighborhood of \((0_{z_j})_j\in \prod_{j\in I}\mathsf{T}_{z_j}X\), where each point in \(U\) is reached by \((\mathsf{exp}_{z_j}(v_{z_j}))_{j \in I}\) for some tangent vector \(v_{z_j} \in \mathsf{T}_{z_j}X\). From there, we extend the tangent data to a smooth, compactly supported vector field, and \(E\) exponentiates it to yield a diffeomorphism.

To replicate this strategy, we must ensure that all steps in the construction can be adapted to our situation. Some care will be required since exponential maps are not globally defined for us (see Corollary \ref{cor: Exponential mappings}).

Assume again that $X$ is a manifold with corners. Let us fix a Riemannian metric $g$ on $X$ and a positive continuous function $\uplambda\colon X\to \R_{>0}$ as in Corollary \ref{cor: Exponential mappings}. Then, the restricted exponential mapping $\mathsf{Exp}\colon \bigsqcup_{j\in I} \mathsf{T}^{\,i}_{z_j}X\to \bigsqcup_{j\in I}\{z_j\}\times X$ yields a diffeomorphism between a neighborhood of the zero-section $\bigsqcup_{j\in I}W_{z_j}$, with $W_{z_j}:=W\cap \mathsf{T}^{\,i}_{z_j}X$, and a neighborhood of the diagonal $O_{\underline{z}}$. Define a neighborhood $ U$ of $\underline{z}\in \Conf_{Z}(X)$ as 
$$
U:=\left\{\underline{x}\in \Conf_Z(X):\;\; (z_j,x_j)\in O_{\underline{z}}\text{ for all }j\in I\right\}
$$
and the map $u\colon U\to \prod_{j\in I}\mathsf{T}^{\,i}_{z_j}X$ by setting $\uppi_ju(\underline{x})=\mathsf{Exp}^{-1}(z_j,x_j)$ for each $j\in I$. Hence,  $u$ lands in $\prod_{j\in I}W_{z_j}$, which by  Corollary \ref{cor: Exponential mappings} is a product of intersections of open balls with quadrants. 

Next, we define $k$. Since $\mathsf{exp}\colon W_{z_j}\to X$ is a diffeomorphism onto an open neighborhood of $z_j$, let us consider relatively compact open subsets $(V_{z_j})_{j\in I}$ such that $z_j\in V_{z_j}$, $\overline{V}_{z_i}\cap\overline{V}_{z_j}=\diameter$ for $i\neq j$, and $\mathsf{exp}^{-1}(V_{z_j})\subseteq W_{z_j}$. 
Then, the exponential map 
$
\mathsf{exp}\colon \bigsqcup_{j\in I}\mathsf{exp}^{-1}(V_{z_j})\to \bigsqcup_{j\in I} V_{z_j}\subseteq X
$
is a diffeomorphism and we set $K:=  \bigsqcup_{j\in I} \overline{V}_{z_j}$. Choose a bump function $\upalpha_j\colon X\to [0,1]$ for each $j\in I$ such that $\upalpha_j(z_j)=1$, which vanishes on a neighborhood of $X\backslash V_{z_j}$, and define 
$$
k\colon \prod_{j\in I}\mathsf{T}^{\,i}_{z_j}X=\Upgamma\big(\bigsqcup_{j\in I}\{z_j\}, \bigsqcup_{j\in I}\mathsf{T}^{\,i}_{z_j}X\big)\xrightarrow{\qquad} \Upgamma_{K}(X,\mathsf{T}^{\,i}X)
$$ as follows:
$$
k(w)\colon x \xmapsto{\qquad} \begin{cases}
    \displaystyle{\sum_{j\in I}}\,\frac{\uplambda(x)\upalpha_j(x)}{\uplambda(z_j)}\mathsf{trans}\big(w_{\uppi\, \mathsf{exp}^{-1}(x)},\mathsf{exp}^{-1}(x)\big) & \quad \text{if }x\in K, \\[8mm]
    0 & \quad \text{if }x\in X\backslash K,
\end{cases}
$$
where $\uppi\colon \mathsf{T}^{\,i}X\to X$ denotes the canonical projection and $\mathsf{trans}(u_x,v_x)$ denotes parallel transport of $u_x$ along $t\mapsto \mathsf{exp}(tv_x)$. Notice that we need a modification of the map appearing in \cite[Lemma c]{palais_local_1960} which takes into account the positive function $\uplambda$ in Corollary \ref{cor: Exponential mappings}. Also, observe that parallel transport is taken with respect to the Levi-Civita connection associated to $(X,g)$; its existence is guaranteed by \cite[Covariant derivatives, Theorem 4.1]{conrad_differential_nodate} and \cite[Linear ODE, Theorem 3.1]{conrad_differential_nodate}. 

The last map, $E$, needs no modification in our case (see \cite[Lemmas a and b]{palais_local_1960}); just observe that we exponentiate inner vector fields $v\in \Upgamma_{K}(X,\mathsf{T}^{\,i}X)$ using Corollary \ref{cor: Exponential mappings} and hence, $v$ must satisfy $||v_{x}||_{g}<\uplambda(x)$ for any $x\in X$.

Let us check that the three defined maps can be composed. That is, given $\underline{y}\in U$, the inequality $||k(u(\underline{y}))_x ||_g<\uplambda(x)$ holds for any $x\in X$. By definition, this is clearly satisfied when $x\in X\backslash K$. Therefore, it suffices to study what happens for $x\in \overline{V}_{z_j}$, since $K=\bigsqcup_{j\in I}\overline{V}_{z_j}$. In this case,
\begin{align*}
    ||k(u(\underline{y}))_x ||_g\; & \overset{(a)}{=}\; \frac{\uplambda(x)\upalpha_j(x)}{\uplambda(z_j)}\left|\left|\mathsf{trans}\big(u(\underline{y})_{z_j}, \mathsf{exp}^{-1}(x)\big) \right|\right|_g\\
    &\overset{(b)}{=} \;  \frac{\uplambda(x)\upalpha_j(x)}{\uplambda(z_j)} ||u(\underline{y})_{z_j} ||_g\\
    & \overset{(c)}{<}\;  \uplambda(x)
\end{align*}
where we have applied: $\upalpha_{i}\vert_{\overline{V}_{z_j}}=0$ if $i\neq j$ for $(a)$; in $(b)$, that parallel transport with respect to the Levi-Civita connection is an isometry (see \cite[Covariant derivatives, Lemma 3.1]{conrad_differential_nodate}); and $0\leq \upalpha_j(x)\leq 1$ together with $  ||u(\underline{y})_{z_j} ||_g<\uplambda(z_j)$, since $\underline{y}\in U$ implies that $u(\underline{y})_{z_j}\in W_{z_j}$, for the last inequality $(c)$.

To finish the proof, observe that the rest of the properties used by Palais in \cite[Lemmas a--d]{palais_local_1960} readily hold for our simple modifications of the maps $u$, $k$, $E$. Henceforth, we obtain a local cross-section $\upchi=E\circ k\circ u$ as desired.
\end{proof}

\begin{proof}[Proof of Lemma \ref{lem: Evaluation at center is a fibration for Embeddings}]
Combining Lemmas \ref{lem: Cerf-Palais} and \ref{lem: Local-cross sections for Configuration spaces}, we obtain that the  map
$$
 \Emb\big(\bigsqcup_{i\in I}\mathbbst{R}^{Z(i)},X\big)\xrightarrow{\quad\quad} \Conf_Z(X)
$$
from Lemma \ref{lem: Evaluation at center is a fibration for Embeddings}, which is clearly $\mathsf{Diff}_{\mathsf{c}}(X)$-equivariant, is a fiber bundle. By \cite[Theorem 6.3.3]{tom_dieck_algebraic_2008}, we conclude the proof.\footnote{Notice that the evaluation map is even a Hurewicz fibration since $\Conf_{Z}(X)$ is paracompact.}
\end{proof}

\begin{prop}\label{prop: Scanning with one basic}
Let $X$ be an $n$-manifold with corners. For any $Z\in\Prect(X)$, differentiating at $0\in \mathbbst{R}^Z$ yields a weak homotopy equivalence 
$$
\Emb(\mathbbst{R}^{Z},X)\xrightarrow{\quad\quad} \Fr^{\,i}_{[Z]}(X_{[Z]}).
$$
\end{prop}
\begin{proof} Explicitly, this map is given by $\varphi\mapsto (\varphi(0),\mathsf{d}_0\varphi)$. Thus, clearly $\varphi(0)\in X_{[Z]}$ and $\mathsf{d}_0\varphi\colon \R^n\cong\mathsf{T}_0\R^{Z}\to \mathsf{T}_{\varphi(0)}X$ is a linear isomorphism which restricts to an isomorphism of quadrants $ \R^Z\cong \mathsf{T}^{\,i}_{\varphi(0)}X$. Therefore, the map is well-defined as displayed and fits into a commutative triangle 
$$
\begin{tikzcd}
    \Emb(\mathbbst{R}^{Z},X)\ar[rr,"\mathsf{d}_0"]\ar[rd,"\mathsf{ev}_0"'] && \Fr^{\,i}_{[Z]}(X_{[Z]})\ar[ld,"\mathsf{p}"] \\
    & X_{[Z]}
\end{tikzcd}.
$$

The diagonal maps are Serre fibrations: the one on the left due to Lemma \ref{lem: Evaluation at center is a fibration for Embeddings} and $\mathsf{p}$ because it is the projection of a principal bundle (see \cite[Theorem 6.3.3]{tom_dieck_algebraic_2008}). Hence, it suffices to look at the induced map on fibers over $z\in X_{[Z]}$. By shrinking in $\mathbbst{R}^{Z}$, we may assume that we land in a fixed chart around $z$, which must be of type $\mathbbst{R}^{Z}$ and such that $0\mapsto z$. Hence, the induced map on fibers over $z$ is a weak homotopy equivalence because so is the map
$
\mathsf{d}_0\colon\Emb_{0\mapsto 0}(\mathbbst{R}^{Z},\mathbbst{R}^{Z})\to \mathsf{GL}(\mathbbst{R}^Z)
$. To see this, observe that $\mathsf{d}_0$ is a left-inverse for the canonical inclusion $\mathsf{GL}(\R^Z)\hookrightarrow \Emb_{0\mapsto 0}(\R^Z,\R^Z)$ and the latter can be shown to be a weak homotopy equivalence as in \cite[Theorem 9.1.1]{kupers_lectures_nodate} by $(i)$ deforming embeddings that fix the origin to be linear near $0$ (with a bump function), and $(ii)$ pushing the non-linearity out to $\infty$ (using dilation maps).   
\end{proof}

Using similar ideas, one can also obtain the following:
\begin{prop}\label{prop: Scanning with several basics} Let $X$ be an $n$-manifold with corners, $I$ be a finite set and let $Z\colon I\to \Prect(X)$ be a function. Differentiating at the origin yields a weak homotopy equivalence
$$
\Emb\big(\bigsqcup_{j\in I}\mathbbst{R}^{Z(j)},X\big)\xrightarrow{\quad\quad} \prod_{K\in \Prect(X)} \Fr^{\,i}_{[K]}\big(\Conf_{Z^{-1}\{K\}}(X_{[K]})\big).
$$
\end{prop}
\begin{proof} Let us begin by explaining more concretely the previous map. First notice that evaluating embeddings $\varphi\colon \bigsqcup_j\R^{Z(j)}\hookrightarrow X$ at the origins $\{0\in \R^{Z(j)}\}_j$ yields a map $\mathsf{ev}_0\colon \Emb(\bigsqcup_j\R^{Z(j)},X)\longrightarrow \Conf_I(X)$ which factors through the subspace $\Conf_Z(X)$, by our refined convention on embedding spaces. Regarding the tangential data, we have a linear isomorphism $\mathsf{d}_{0}(\varphi\vert_{\R^{Z(j)}})\colon \R^n\cong \mathsf{T}_0\R^{Z(j)}\to \mathsf{T}_{\varphi(0)}X$ for each $j\in I$ which restricts to an isomorphism of quadrants $\R^{Z(j)}\cong \mathsf{T}^{\,i}_{\varphi(0)}X $. This is exactly what the notation in the statement signifies; i.e.\ the structural group of the principal bundle $\Fr^{\,i}_{[K]}(\Conf_{t}(X_{[K]}))$ is $\mathsf{GL}(\R^K)^{\times t}$. Slightly more formally, this can be expressed as saying that the canonical mapping
$$
\Emb\big(\bigsqcup_{j\in I}\R^{Z(j)},X\big) \xrightarrow{\quad\quad} \mathsf{Fr}^{\,i}\big(\Conf_I(X)\big),\quad \varphi\longmapsto  \big((\varphi\vert_{\R^{Z(j)}})(0), \mathsf{d}_0(\varphi\vert_{\R^{Z(j)}})\big)_j
$$
lands in the claimed subspace. Therefore, as in Proposition \ref{prop: Scanning with one basic}, we have a commutative triangle with lower tip $\Conf_Z(X)$ that reduces the claim to analyzing the induced map on fibers over points $\underline{z}\in\Conf_Z(X)$. Again, by shrinking in every $\R^{Z(j)}$, we may assume that we land in an open subspace of $X$ diffeomorphic to $\bigsqcup_{j}\R^{Z(j)}$. From here, the result follows from $(i)$ the argument in Proposition \ref{prop: Scanning with one basic} because, by continuity and connectivity, we have $\Emb_{(0_j\mapsto 0_j)_j}(\bigsqcup_j\R^{Z(j)},\bigsqcup_j\R^{Z(j)})= \prod_j\Emb_{0\mapsto 0}(\R^{Z(j)},\R^{Z(j)})$, and $(ii)$ since the fiber of $\prod_K\mathsf{Fr}^{\,i}_{[K]}\big(\Conf_{Z^{-1}\{K\}}(X_{[K]})\big) $ over $\underline{z}\in \Conf_Z(X)$ is $\prod_j\mathsf{GL}(\R^{Z(j)})$. 
\end{proof}

\section{Higher categorical tools}\label{sect: Higher categorical tools}
We collect some abstract results that are used in the main body of the text. Let us begin with two direct consequences of the long exact sequence of homotopy groups associated with a fiber sequence of spaces.

\begin{prop}\label{prop: Categorical triangle to analyze fibers}
   Let us consider a commutative triangle of $\infty$-categories
    \[
    \begin{tikzcd}
        \EuScript{C} \ar[rr, "F"] \ar[dr, "H"'] && \EuScript{D} \ar[dl, "G"] \\
        & \EuScript{E} &
    \end{tikzcd}.
    \]
  Then, the homotopy fiber of \(F\) over an object \(d \in \EuScript{D}\) is contractible if, for \(G(d) = e \in \EuScript{E}\), the functor $F$ induces a weak homotopy equivalence between the homotopy fibers \(\EuScript{C}_{\vert e}\) and \(\EuScript{D}_{\vert e}\). If moreover \(G\) is a right/left fibration, the homotopy fiber $\EuScript{D}_{\vert e}$ coincides with the strict fiber of $G$ over $e$ (see \cite[Corollary 5.3.3]{cisinski_higher_2019}).
\end{prop}

\begin{prop}
\label{prop:localization}
    Let \(F: \EuScript{C} \to \EuScript{D}\) be a functor. Assume \(\EuScript{D}\) is an \(\infty\)-groupoid. Then \(F\) has contractible homotopy fibers if and only if the induced map \( \vert\EuScript{C} \vert \, \simeq \EuScript{D}\) is an equivalence of $\infty$-groupoids.
\end{prop}


\begin{lemma}
\label{lem:strict pullback}
    Let \(\mathsf{C}\) be an ordinary category, and let 
    \(
    F \colon \mathsf{C} \to \EuScript{D}
    \)
    be a functor into an \(\infty\)-category \(\EuScript{D}\). Assume that \(\mathsf{C}\) and \(\EuScript{D}\) admit isofibrations \(\mathsf{C} \xrightarrow[]{} \Fin_{*}\) and \(\EuScript{D} \xrightarrow[]{} \Fin_{*}\) such that the functor \(F\) sits in the commutative diagram 
    \[
    \begin{tikzcd}
    \mathsf{C} \ar[rr,"F"] \ar[dr] && \EuScript{D} \ar[dl, ]  \\
     & \Fin_{*}&
    \end{tikzcd}.
    \]
    Assume that for all \(\langle n \rangle \in \Fin_{*}\), the mapping spaces of \(\EuScript{D}_{\,|\langle n \rangle}\) are either contractible or empty. 
    Then for every object \(d \in \EuScript{D}\) that lives over \(\langle n \rangle \in \Fin_{*}\), the homotopy fiber of \(F\) over \(d\),  \(\mathsf{C}_{|d}\), sits into a strict pullback of ordinary categories:
    \[
    \begin{tikzcd}
    \mathsf{C}_{|d} \ar[r] \ar[d] \ar[rd, phantom, "\lrcorner" description] & \mathsf{C}_{|\langle n \rangle}  \ar[d]  \\
    \ast \ar[r,"d"'] &  h(\EuScript{D}_{\, |\langle n \rangle}),
    \end{tikzcd}
    \]
    where \(h(-)\) denotes the homotopy category.
\end{lemma}

\begin{proof}
    Consider the following diagram
    
    \[
    \begin{tikzcd}
    \mathsf{C}_{|d} \ar[r] \ar[d] \ar[rd, phantom, "\overset{\mathsf{h}\;\;}{\lrcorner}" description] & \mathsf{C}_{|\langle n \rangle} \ar[d] \ar[r,]  &  \mathsf{C}\ar[d] \ar[ld, phantom, "\overset{\mathsf{h}\;\;}{\lrcorner}" description] \\
    \ast \ar[r,"d"']  & \EuScript{D}_{\, |\langle n \rangle} \ar[r]  \ar[d] & \EuScript{D} \ar[d] \ar[ld, phantom, "\lrcorner" description] \ar[d]\\
    & \ast \ar[r,"\langle n \rangle"'] & \Fin_{*}
    \end{tikzcd}.
    \] 
    A priori, we are taking homotopy pullbacks. Since homotopy pullbacks along isofibrations agree with strict pullbacks (see \cite[Proposition A.2.4.4]{lurie_higher_2009}), the lower square can be taken as a strict pullback. The same is true for the pullback along \(\mathsf{C} \xrightarrow[]{} \Fin_{*}\). This justifies all the categories on the vertices of the diagram. 
    
    Due to the assumptions on \(\EuScript{D}_{\,|\langle n \rangle}\) (see \cite[Proposition 2.3.4.5, Example 2.3.4.17 and Proposition 2.3.4.18]{lurie_higher_2009}), we conclude that \(\EuScript{D}_{\, |\langle n \rangle} \simeq \Nerve ( h(\EuScript{D}_{\, |\langle n \rangle}))\), where \(h(\EuScript{D}_{\, |\langle n \rangle})\) is a poset. Therefore, we obtain the following diagram
    \[
    \begin{tikzcd}
    \mathsf{C}_{|d} \ar[r] \ar[d] \ar[rd, phantom,  "\overset{\mathsf{h}\;\;}{\lrcorner}" description] & \mathsf{C}_{|\langle n \rangle} \ar[d] \\
    \ast \ar[r]  \ar[d,"\wr"'] \ar[rd, phantom,  "\overset{\mathsf{h}\;\;}{\lrcorner}" description]  & \EuScript{D}_{\, |\langle n \rangle} \ar[d,"\wr"]  \\
    \ast \ar[r] & h(\EuScript{D}_{\, |\langle n \rangle} )
    \end{tikzcd},
    \]
    whose outer square is the one from the statement. As any functor that maps into a nerve of a poset is an isofibration \cite[Example 01GV]{lurie_kerodon_nodate}, we conclude the proof.
\end{proof}

\begin{remark}
    This last result was applied throughout the text on several occasions to streamline the arguments, namely in the proof of Lemmas \ref{lem:Second condition WeakAprroxofE},  \ref{lem: First condition WreathProdApprox}, \ref{lem: Second condition WreathProdApprox}, and \ref{lem: First condition LocalizationTheorem via WeakApprox}. Notice that the $\infty$-category \(\EuScript{D}=\Eactive_{p,q}\) satisfies the assumptions of Lemma \ref{lem:strict pullback} since 
    \[
    (\Eactive_{p,q})_{|\langle n \rangle} \brbinom{\{Q_{\bullet}\}_{ \langle n \rangle}}{\{R_{\bullet}\}_{\langle n \rangle}}= \prod_{i \in \langle n \rangle^{\circ}} \Rect \big( \square^{\,p, Q_i}, \square^{\,p, R_i} \big).
    \]
    An analogous statement holds for $\EuScript{D}=\gE^{\otimes,\,\act}_{\square^{\,p,q}}$ by the homotopy pullback diagram (\ref{eqt:hopbs showing space of relative embeddings is contractible}).
\end{remark}

\paragraph{A nested formula for homotopy colimits.}

\begin{lemma}\label{lem: colims over cocart fibrations} Let $\int_{i\in \,\EuScript{I}}F(i)\to \EuScript{I}$ be the cocartesian fibration corresponding to a diagram $F\colon \EuScript{I}\to \Cat_{\infty}$, and let $\EuScript{C}$ be a cocomplete \(\infty\)-category. For any diagram $D\colon \int_{i}F(i)\to \EuScript{C}$, there is a natural equivalence
$$
\underset{x\in \int_{i}F(i)}{\hocolim}\,D_x \simeq \underset{i\in \,\EuScript{I}}{\hocolim}\,\underset{z\in F(i)}{\hocolim}\,D_z.
$$
Consequently, given a diagram $\overline{D}\colon \hocolim_{i}F(i)\to \EuScript{C}$, there is a natural equivalence
$$
 \hocolim\big(\underset{i\in\,\EuScript{I}}{\hocolim}\,F(i)\xrightarrow{\;\overline{D}\;}\EuScript{C}\big)\simeq \underset{i\in \,\EuScript{I}}{\hocolim}\,\underset{z\in F(i)}{\hocolim}\,\overline{D}_z.
$$  

Dually, for any diagram $\overline{D}\colon (\hocolim_iF(i))^{\op}\to \EuScript{D}$, where $\EuScript{D}$ is a complete \(\infty\)-category, there is a natural equivalence
$$
 \holim\big(\underset{i\in \,\EuScript{I}}{(\hocolim}\,F(i))^{\op}\xrightarrow{\;\overline{D}\;}\EuScript{D}\big)\simeq \underset{i\in \,\EuScript{I}^{\op}}{\holim}\,\underset{z\in F(i)^{\op}}{\holim}\,\overline{D}_z.
$$
\end{lemma}
\begin{proof} Let us first prove a more general statement: consider a span of \(\infty\)-categories $\EuScript{I}\xleftarrow{\,\uppi\,}\gE\xrightarrow{\,D\,}\EuScript{C}$. Then, we claim that the homotopy colimit of $D$ can be computed as follows
$$
\underset{\gE}{\hocolim}\,D\simeq \underset{i\in\,\EuScript{I}}{\hocolim}\,\underset{\gE_{/i}}{\hocolim}\,D\vert_{\gE_{/i}}. 
$$
In order to show this, consider the Grothendieck construction $\EuScript{M}\to \Delta^1$, i.e.\ the unstraightening, of $\uppi\colon \gE\to \EuScript{I}$. Informally, $\EuScript{M}$ is the \(\infty\)-category with objects $\mathsf{ob}(\gE)\amalg \mathsf{ob}(\EuScript{I})$ and mapping spaces
$$
\Map_{\EuScript{M}}(m,m')=
\begin{cases}
    \Map_{\gE}(m,m') & \quad \text{if } m,m' \in \mathsf{ob}(\gE), \vspace*{2mm}\\ 
    \Map_{\EuScript{I}}(m,m') & \quad \text{if } m,m'\in \mathsf{ob}(\EuScript{I}), \vspace*{2mm}\\ 
    \Map_{\EuScript{I}}(\uppi(m),m') & \quad \text{if } m\in \mathsf{ob}(\gE) \text{ and }m'\in \mathsf{ob}(\EuScript{I}), \vspace*{2mm}\\ 
    \diameter & \quad \text{if } m\in \mathsf{ob}(\EuScript{I}) \text{ and }m'\in \mathsf{ob}(\gE).
\end{cases}
$$
Then, there are canonical inclusions $j_0\colon\gE\hookrightarrow \EuScript{M}\hookleftarrow \EuScript{I}:\! j_1$, which are just the fiber inclusions of $\EuScript{M}\to \Delta^1$. Observe that $\EuScript{M}\hookleftarrow \EuScript{I}$ admits a left adjoint $\widetilde{\uppi}$, essentially given by the identity on $\mathsf{ob}(\EuScript{I})$ and by $\uppi$ on $\mathsf{ob}(\gE)$. Hence, one can (homotopy) left Kan extend $D$ along $j_0$  and notice that $\EuScript{I}\hookrightarrow \EuScript{M}$ is cofinal (since it admits a left adjoint) to obtain
$$
\underset{x\in \gE}{\hocolim}\,D(x)\simeq \underset{m\in \EuScript{M}}{\hocolim}\,(\mathsf{Lan}_{j_0}D)(m)\simeq \underset{i\in \,\EuScript{I}}{\hocolim}\,(\mathsf{Lan}_{j_0}D)(i).
$$
We conclude our first claim by applying the pointwise formula for left Kan extensions due to the identification  
$
\gE_{/i}:= \EuScript{I}_{/i}\times_{\EuScript{I}}\gE \simeq \EuScript{M}_{/i}\times_{\EuScript{M}}\gE 
$
for any $i\in \mathsf{ob}(\EuScript{I})\subseteq\mathsf{ob}(\EuScript{M})$.

For the first statement in the lemma, we consider the particular case when $\uppi\colon \gE\to \EuScript{I}$ is a cocartesian fibration. Thus, it suffices to observe that $\gE_i=\{i\}\times_{\EuScript{I}}\gE\to \EuScript{I}_{/i}\times_{\EuScript{I}}\gE=\gE_{/i}$ is cofinal in this case, because $F(i)\simeq \gE_i$ when $\uppi$ is the cocartesian fibration associated with $F$. For this, note that for any $(e,l\colon \uppi(e)\to i)\in \gE_{/i}$, the slice $(\gE_i)_{(e,l)/}$ is contractible since it has an initial object: pick a cocartesian lift $e\to l_!e$ over $l$. 

For the second statement, we use that $\infty$-localizations are cofinal \cite[Proposition 7.1.10]{cisinski_higher_2019} and that $\int_{i}F(i)\to \hocolim_iF(i)$ is the $\infty$-localization at the cocartesian morphisms since $\int_iF(i)$ is the oplax colimit of $F\colon \EuScript{I}\to \Cat_{\infty}$ (see \cite[Theorem 7.4]{gepner_lax_2017}).

 The last statement follows from the previous one by looking at $\overline{D}$ as a diagram $\hocolim_iF(i)\to \EuScript{D}^{\op}$. 
\end{proof}

\paragraph{The Boardman-Vogt tensor product for \(\infty\)-operads.}

Lurie showed that $\Opd_{\infty}$ can be equipped with a symmetric monoidal structure, $\Opd_{\infty}^{\otimes}\to \Fin_{*}$, called Boardman-Vogt tensor product \cite[\textsection 2.2.5]{lurie_higher_nodate}. Such monoidal product $\boxtimes \colon \Opd_{\infty}\times\Opd_{\infty}\to \Opd_{\infty}$ can be obtained from a left Quillen bifunctor \cite[Proposition 2.2.5.7]{lurie_higher_nodate} and hence it commutes with homotopy colimits in each variable. Therefore, for any \(\infty\)-operad $\EuScript{O}$, we obtain a right adjoint functor $\ALG_{\EuScript{O}}(-)\colon \Opd_{\infty}\to \Opd_{\infty}$ of $\EuScript{-\boxtimes O}$. Actually, this functor is characterized by the existence of a natural equivalence of \(\infty\)-categories
\begin{equation}\label{eqt: Algebras over BV-tensor}
    \Alg_{\EuScript{B}\boxtimes\EuScript{O}}(\EuScript{P})\simeq \Alg_{\EuScript{B}}(\ALG_{\EuScript{O}}(\EuScript{P}))
\end{equation}
for any pair $\EuScript{B}, \EuScript{P}\in \Opd_{\infty}$. For instance, using $\Map_{\Opd_{\infty}}(\EuScript{O},\EuScript{P})\simeq \Alg_{\EuScript{O}}(\EuScript{P})^{\simeq}$, we obtain
$$
\Map_{\Opd_{\infty}}(\EuScript{B}\boxtimes \EuScript{O},\EuScript{P})\simeq \Map_{\Opd_{\infty}}(\EuScript{B},\ALG_{\EuScript{O}}(\EuScript{P}))
$$
from the previous equivalence between \(\infty\)-categories of algebras.

In \cite[\textsection 2.2.5]{lurie_higher_nodate}, Lurie also worked out an explicit description of the $\infty$-category $\Alg_{\EuScript{B}}(\ALG_{\EuScript{O}}(\EuScript{P}))$ for any triple of $\infty$-operads $\EuScript{B}$, $\EuScript{O}$, $\EuScript{P}$. Its objects are identified with the so called \emph{bifunctors}, i.e.\ a bifunctor is just a functor $\upvarphi\colon \EuScript{B}^{\otimes}\times \EuScript{O}^{\otimes}\to \EuScript{P}^{\otimes}$ such that: $(i)$ the diagram 
$$
\begin{tikzcd}
\EuScript{B}^{\otimes}\times \EuScript{O}^{\otimes} \ar[d]\ar[r,"\upvarphi"] & \EuScript{P}^{\otimes}\ar[d]\\
\Fin_*\times \Fin_* \ar[r,"\wedge"'] & \Fin_*
\end{tikzcd}
$$
commutes (where $\wedge$ denotes the smash product of finite sets), and $(ii)$ $\upvarphi$ sends pairs of inert maps to inert maps. By considering the full subcategory spanned by them, $\BiFun(\EuScript{B},\EuScript{O};\EuScript{P})$, in the $\infty$-category of functors $\EuScript{B}^{\otimes}\times \EuScript{O}^{\otimes}\to \EuScript{P}^{\otimes}$ satisfying $(i)$, Lurie established a canonical equivalence of $\infty$-categories
$$
\BiFun(\EuScript{B},\EuScript{O};\EuScript{P})\simeq \Alg_{\EuScript{B}}(\ALG_{\EuScript{O}}(\EuScript{P})).
$$
Combined with Equation (\ref{eqt: Algebras over BV-tensor}), we arrive at:
\begin{defn}\label{defn: Exhibits a tensor product} A bifunctor $\upvarphi\colon \EuScript{B}_1^{\otimes}\times \EuScript{B}_2^{\otimes}\to \EuScript{B}^{\otimes}$ exhibits $\EuScript{B}$ as a tensor product of $\EuScript{B}_1$ and $\EuScript{B}_2$ if, for any $\infty$-operad $\EuScript{P}$, precomposition with $\upvarphi$ induces an equivalence of $\infty$-categories
$
\Alg_{\EuScript{B}}(\EuScript{P})\simeq 
\BiFun(\EuScript{B}_1,\EuScript{B}_2;\EuScript{P}).
$    
\end{defn}

In the remainder of this appendix, we collect a couple of simple facts about the functor $\ALG$ that we apply throughout the text.

\begin{lemma}{\cite[Proposition 3.2.4.3 and Example 3.2.4.4]{lurie_higher_nodate}} If $\EuScript{V}^{\otimes}\to \Fin_{*}$ is a symmetric monoidal \(\infty\)-category, so is $\ALG_{\EuScript{O}}(\EuScript{V})\to \Fin_{*}$. Informally, the tensor product of two $\EuScript{O}$-algebras $A$ and $B$ is given by the pointwise tensor product $A\otimes B\colon o\mapsto A(o)\otimes B(o) $.
\end{lemma}

\begin{remark} The previous lemma just asserts that $\Alg_{\EuScript{O}}(\EuScript{V})$ admits a symmetric monoidal structure by tensoring $\EuScript{O}$-algebras in $\EuScript{V}$ with the monoidal structure of $\EuScript{V}$.
\end{remark}

\begin{lemma}\label{lem:ALG commutes with lims} The functor $\ALG\colon \Opd^{\op}_{\infty}\times \Opd_{\infty}\to \Opd_{\infty}$ commutes with homotopy limits in both variables. In other words, 
$$
\ALG_{(\underset{i}{\hocolim}\, \EuScript{O}_i)}(\EuScript{P})\simeq \underset{i}{\holim}\ALG_{\EuScript{O}_i}(\EuScript{P})  \quad \text{and} \quad \ALG_{\EuScript{O}}(\underset{j}{\holim}\,\EuScript{P}_j)\simeq \underset{j}{\holim}\ALG_{\EuScript{O}}(\EuScript{P}_j).
$$

The same holds for $\Alg\colon \Opd_{\infty}^{\op}\times \Opd_{\infty}\to \Cat_{\infty}$. 
\end{lemma}
\begin{proof}
    The commutation of homotopy limits in the second variable of $\ALG$ comes from the fact that $\ALG_{\EuScript{O}}(-)$ is right adjoint to $-\boxtimes\EuScript{O}$ for any \(\infty\)-operad $\EuScript{O}$. The commutation in the first variable follows from the fact that the Boardman-Vogt tensor product $\boxtimes\colon \Opd_{\infty}\times \Opd_{\infty}\to \Opd_{\infty}$ preserves homotopy colimits in each variable by Yoneda lemma applied to the chain of equivalences
    \begin{align*}
          \Map_{\Opd_{\infty}}\big(\EuScript{B}, \ALG_{(\underset{i}{\hocolim}\, \EuScript{O}_i)}(\EuScript{P})\big) &\simeq \Map_{\Opd_{\infty}}\big(\EuScript{B}\boxtimes (\underset{i}{\hocolim}\,\EuScript{O}_i),\EuScript{P}\big)\\
          &\simeq \Map_{\Opd_{\infty}}\big(\underset{i}{\hocolim}\,(\EuScript{B}\boxtimes \EuScript{O}_i),\EuScript{P}\big)\\
          &\simeq \underset{i}{\holim}\Map_{\Opd_{\infty}}\big(\EuScript{B}\boxtimes \EuScript{O}_i,\EuScript{P}\big)\\
          &\simeq \underset{i}{\holim}\Map_{\Opd_{\infty}}\big(\EuScript{B},\ALG_{\EuScript{O}_i}(\EuScript{P})\big)\\
           &\simeq \Map_{\Opd_{\infty}}\big(\EuScript{B},\underset{i}{\holim}\ALG_{\EuScript{O}_i}(\EuScript{P})\big).
    \end{align*}
    
    The claims for $\Alg$ reduce to those of $\ALG$ since $\Alg=\ALG_{\langle 1\rangle}$.
\end{proof} 

\begin{remark}
    Recall that the ``forgetful" functors  
    $
    \mathsf{fgt}\colon\mathsf{sm}\!\Cat_{\infty}^{\mathsf{strong}}\xrightarrow{\quad} \Opd_{\infty}
    $  
    and 
    $
    (-)_{\langle 1\rangle}\colon \Opd_{\infty}\xrightarrow{\quad} \Cat_{\infty}
    $
    preserve homotopy limits. 
\end{remark}

\section{Weak and strong approximations}
\label{sect:weak operads}

This last appendix gathers the key results from \cite{harpaz_little_nodate}, which serve as the foundation for our arguments in \textsection \ref{sect: Additivity for generalized swiss-cheese operads} on the additivity of generalized Swiss-cheese operads.

\begin{defn}{\cite[Definition 4.2.4]{harpaz_little_nodate}}
    A \emph{weak \(\infty\)-operad} consists of an \(\infty\)-category \(\EuScript{C}\) together with:  
    \begin{itemize}
        \item a \emph{factorization system} \((\EuScript{C}^{\inert}, \EuScript{C}^{\act})\) (see \cite[Definition 4.2.1]{harpaz_little_nodate}), 
        whose morphisms are called \emph{inert} and \emph{active} respectively, and 

        \item a full \emph{subcategory of basics} \(\EuScript{C}_0 \subseteq \EuScript{C}\), such that for every object \(x \in \EuScript{C}\), there exist: $(1)$ a finite set \(I_x\), and $(2)$ a family of inert morphisms \(\{f_i : x \to x_i\}_{i \in I_x}\) with \(x_i \in \EuScript{C}_{0}^{\inert}\),
        so that the induced functor
          \(
          I_x \;\xrightarrow[]{}\; (\EuScript{C}^{\inert})_{x/} \times_{\EuScript{C}^{\inert}} \EuScript{C}_{0}^{\inert} 
          \)
          is coinitial.
    \end{itemize}

  A morphism of weak \(\infty\)-operads $(\EuScript{C}, \EuScript{C}^{\inert},\EuScript{C}^{\act},\EuScript{C}_0)\to (\EuScript{D}, \EuScript{D}^{\inert},\EuScript{D}^{\act},\EuScript{D}_0)$ is just a functor $\upvarphi\colon \EuScript{C}\to \EuScript{D}$ which respects the three subcategories and such that it induces a coinitial functor $(\EuScript{C}^{\inert})_{x/} \times_{\EuScript{C}^{\inert}} \EuScript{C}_{0}^{\inert} \to (\EuScript{D}^{\inert})_{\upvarphi(x)/} \times_{\EuScript{D}^{\inert}} \EuScript{D}_{0}^{\inert} $ for any $x\in \EuScript{C}$. 
\end{defn}

\begin{example}
\label{ex:weak structure on infinity cat}
    In any \(\infty\)-category \( \EuScript{C} \), there exists a canonical factorization system in which the inert morphisms are precisely the equivalences in \( \EuScript{C} \), and the active morphisms are all morphisms in \( \EuScript{C} \). Moreover, given any full subcategory \( \EuScript{C}_0 \subset \EuScript{C} \), one may regard \( \EuScript{C}_0 \) as the subcategory of basics for this factorization system. 
\end{example}

\begin{example}
\label{ex:weak structure on infinity operads}
    For any \(\infty\)-operad \(\EuScript{O}^{\otimes} \xrightarrow[]{} \Fin_{*}\), there is a canonical weak \(\infty\)-operad structure on \(\EuScript{O}^{\otimes}\) in which the factorization system is given by exactly the inert and active maps, and the basic objects are those lying over \(\langle 1 \rangle\).
\end{example}

\begin{defn}\label{defn: weak and strong approximation}
    Let $\upgamma: \EuScript{C} \to \EuScript{D}$ be a morphism of weak \(\infty\)-operads. Then $\upgamma$ is a \emph{weak approximation} if the following conditions hold:
    \begin{itemize}
       \item The map $\upgamma^{-1} \; \EuScript{D}^{\inert}_0 \to \EuScript{D}^{\inert}_0 $ is an $\infty$-localization map.
       \item For every $y \in \EuScript{C}$, the homotopy fibers of $(\EuScript{C}^{\,\act})_{/y} \to (\EuScript{D}^{\,\act})_{/ \psi(y)}$ are contractible.
    \end{itemize}
    We say that $\upgamma\colon \EuScript{C}\to \EuScript{D}$ is a \emph{strong approximation} if it is a weak approximation and $\upgamma^{-1} \; \EuScript{D}^{\inert}_0 \to \EuScript{D}^{\inert}_0 $ is an equivalence of \(\infty\)-categories.
\end{defn}

\begin{remark} 
\label{rem:WeakApproxReduction}
    Let $\upgamma\colon \EuScript{O}^{\otimes}\to \EuScript{P}^{\otimes}$ be a map of \(\infty\)-operads. Then, the second condition for $\upgamma$ to be a weak approximation can be checked only for $y\in \EuScript{O}^{\otimes}_{\langle 1\rangle}$. This fact appears and it is justified explicitly in the proof of \cite[Proposition 5.2.4]{harpaz_little_nodate}. Also, observe that in this case   $(\EuScript{P}^{\otimes})^{\inert}_0= (\EuScript{P}^{\otimes})^{\inert}_{\langle 1\rangle}$.
\end{remark}

\begin{remark}
    \label{rmk:ProductOfWeakApprox}
     \cite[Remark 4.2.16]{harpaz_little_nodate}
        Let $\upgamma_1: \EuScript{C}_1 \to \EuScript{D}_1$ and $\upgamma_1: \EuScript{C}_1 \to \EuScript{D}_1$ be two weak approximations of weak $\infty$-operads. Then $ ( \upgamma_1 \times \upgamma_2) : \EuScript{C}_1 \times \EuScript{C}_2 \to \EuScript{D}_1 \times \EuScript{D}_2$ is a weak approximation.
\end{remark}

\begin{remark}
\label{rem: weak operad structure on the pullback}
    \cite[Remark 4.2.17]{harpaz_little_nodate}
    Let
    \[
    \begin{tikzcd}
    \EuScript{P} \ar[r, "\upvarphi'"] \ar[d, "\uppsi'"'] \ar[dr, phantom, "\overset{\mathsf{h}\;\;\;}{\lrcorner}"] &
    \EuScript{C} \ar[d, "\uppsi"] \\
    \EuScript{D} \ar[r, "\upvarphi"'] &
    \EuScript{E}
    \end{tikzcd}
    \]
    be a homotopy pullback diagram of \(\infty\)-categories. Assume that \( \EuScript{C} \), \( \EuScript{D} \), and \( \EuScript{E} \) are equipped with weak \( \infty \)-operad structures, and that both \( \upvarphi \) and \( \uppsi \) are morphisms of weak \( \infty \)-operads. Then the pullback \( \EuScript{P} \) naturally inherits a weak \( \infty \)-operad structure characterized as follows: a morphism in \( \EuScript{P} \) is inert (resp.\ active) if and only if its images under \( \upvarphi' \) and \( \uppsi' \) are inert (resp.\ active) in \( \EuScript{C} \) and \( \EuScript{D} \). Furthermore, the subcategory of basic objects in \( \EuScript{P} \) is given by the homotopy fiber product
    \[
    \EuScript{P}_0 := \EuScript{D}_0 \times^{\mathsf{h}}_{\EuScript{E}_0} \EuScript{C}_0.
    \]
    Under these assumptions, if the map \( \uppsi \colon \EuScript{C} \to \EuScript{E} \) is a strong (resp.\ weak) approximation, the induced map \( \uppsi' \colon \EuScript{P} \to \EuScript{D} \) likewise defines a strong (resp.\ weak) approximation.
\end{remark}

\begin{defn}
    Let \((\EuScript{C},\EuScript{C}^{\inert},\EuScript{C}^{\act},\EuScript{C}_0)\) be a weak \(\infty\)-operad. Let \(\EuScript{E}\) be an \(\infty\)-category that admits finite products. Then a \(\EuScript{C}\)-\emph{monoid object} in \(\EuScript{E}\) is a functor \(A : \EuScript{C} \to \EuScript{E}\) such that for every \(x \in \EuScript{C}\) the collection \(\{ A(f_i) :  A(x) \to A(x_i)\}_{i \in I_x}\) exhibits \(A(x)\) as the product of the objects \(\{A(x_i)\}_{i \in I_x}\) in \(\EuScript{E}\). The full subcategory of \(\Fun (\EuScript{C},\EuScript{E})\) spanned by \(\EuScript{C}\)-monoid objects is denoted with \(\Mon_{\EuScript{C}}(\EuScript{E})\).
\end{defn}

\begin{remark}
    If \(\EuScript{O}^{\otimes}\) is an \(\infty\)-operad, we can consider both the $\infty$-category $\Alg_{\EuScript{O}}(\EuScript{E})$ and $\Mon_{\EuScript{O}^{\otimes}} (\EuScript{E})$. \cite[Proposition 4.1.21]{harpaz_little_nodate} shows that they are equivalent.
\end{remark}

\begin{prop}{\cite[Proposition 4.1.23]{harpaz_little_nodate}}
    \label{prop:weak operad 1}
    Let  $\upgamma: \EuScript{C} \to \EuScript{D}$ be a map of \(\infty\)-operads. Then the following conditions are equivalent:
    \begin{enumerate}
        \item The map $ \upgamma $ is an operadic equivalence.
        \item The map $ \upgamma $ is essentially surjective, and restriction along $ \upgamma $ induces an equivalence of \(\infty\)-categories
        $$
            \upgamma^* \colon \Mon_{\EuScript{D}}(\Spc) \xrightarrow{\;\;\;\sim\;\;\;} \Mon_{\EuScript{C}}(\Spc).
        $$
    \end{enumerate}
\end{prop}

\begin{thm}{\cite[Theorem 4.3.11]{harpaz_little_nodate}}
    \label{thm:wreathprod}
    Let $ \EuScript{E} $ be an \(\infty\)-category that admits small limits. Then restriction along the map of weak $\infty$-operads
    $
         \upgamma \colon \EuScript{C}^{\otimes} \times \EuScript{D}^{\otimes} \to \EuScript{C}^{\otimes} \wr \EuScript{D}^{\otimes}
    $
    induces an equivalence of \(\infty\)-categories
    $$
        \Mon_{\EuScript{C} \,\wr\, \EuScript{D}}(\EuScript{E}) \xrightarrow{\;\;\;\sim\;\;\;} \Mon_{\EuScript{C} \times \EuScript{D}}(\EuScript{E}) \simeq \Mon_{\EuScript{C}}(\Mon_{\EuScript{D}}(\EuScript{E})).
    $$
\end{thm}

\begin{defn}
    Let $\upgamma: \EuScript{C} \to \EuScript{D}$  be a morphism of weak \(\infty\)-operads. A \emph{locally constant monoid with respect to $\upgamma$} is a monoid object $A: \EuScript{C} \to \EuScript{E}$ such that $A|_{\upgamma^{-1} \EuScript{D}^{\inert}_0}$ factors through $\EuScript{D}^{\inert}_0$. We denote by $\Mon^{\mathsf{lc}}_{\EuScript{C}}(\EuScript{E})$ the full subcategory of $\Mon_{\EuScript{C}}(\EuScript{E})$ spanned by locally constant monoids.
\end{defn}

\begin{remark}
\label{rmk:LocConstMonoid}
    Let $\upgamma: \EuScript{C} \to \EuScript{D}$ be a weak approximation and let \(W\) be a wide subcategory with arrows in \({\upgamma^{-1} \EuScript{D}^{\inert}_0}\) that are mapped to equivalences in \( \EuScript{D}^{\inert}_0\). Let $A: \EuScript{C} \to \EuScript{E}$ be a monoid object. Then \(A\) is locally constant with respect to \(\upgamma\) if and only if arrows in \(W\) are sent to equivalences in \(\EuScript{E}\).
\end{remark}

\begin{lemma}{\cite[Proposition 4.2.18]{harpaz_little_nodate}}\label{lem:Harpaz weak approx}
    Let $\upgamma: \EuScript{C} \to \EuScript{D}$ be a weak approximation map of weak \(\infty\)-operads and let $\EuScript{E}$ be an \(\infty\)-category which admits small limits. Then the restriction
$$\upgamma^{*}: \Mon_{\EuScript{D}}(\EuScript{E}) \xrightarrow{\;\;\sim\;\;} \Mon_{\EuScript{C}}^{\mathsf{lc}}(\EuScript{E})$$ is an equivalence of \(\infty\)-categories.
\end{lemma}

\begin{lemma}{\cite[Proposition 4.2.18]{harpaz_little_nodate}}\label{lem:Harpaz strong approx}
    Let $\upgamma: \EuScript{C} \to \EuScript{D}$ be a strong approximation map of weak \(\infty\)-operads and let $\EuScript{E}$ be an \(\infty\)-category which admits small limits. Then the restriction
$$\upgamma^{*}: \Mon_{\EuScript{D}}(\EuScript{E}) \xrightarrow{\;\;\sim\;\;} \Mon_{\EuScript{C}}(\EuScript{E})$$ is an equivalence of \(\infty\)-categories.
\end{lemma}

\begin{cor}\label{cor: Harpaz approxs as operadic notions} Let $\upgamma\colon \EuScript{O}^{\otimes}\to \EuScript{P}^{\otimes}$ be a morphism of \(\infty\)-operads. Then, 
\begin{itemize}
    \item if $\upgamma$ is a weak approximation, it exhibits an \(\infty\)-localization of \(\infty\)-operads;
    \item if $\upgamma$ is a strong approximation, it is an equivalence of \(\infty\)-operads.
\end{itemize}
\end{cor}
\begin{proof} Observe that a locally constant monoid $A\in \Mon^{\mathsf{lc}}_{\EuScript{O}}(\Spc)$ is a monoid which sends the morphisms inverted by $\upgamma^{-1}(\EuScript{P}^{\otimes})^{\inert}_{\langle 1\rangle}\to (\EuScript{P}^{\otimes})^{\inert}_{\langle 1\rangle}$ to equivalences, when $\upgamma$ is a weak approximation. Through the equivalence of \(\infty\)-categories given in \cite[Proposition 4.1.21]{harpaz_little_nodate}, those correspond to $\EuScript{O}$-algebras in $\Spc$ sending the unary operations inverted by the previous functor to equivalences. Therefore, the conclusions follow by combining Lemma \ref{lem:Harpaz weak approx} (resp.\ Lemma \ref{lem:Harpaz strong approx}) with \cite[Proposition 4.1.23]{harpaz_little_nodate} and noticing that the hypothesis on $\upgamma^{-1}(\EuScript{P}^{\otimes})^{\inert}_{\langle 1\rangle}\to (\EuScript{P}^{\otimes})^{\inert}_{\langle 1\rangle}$ implies that $\upgamma$ is essentially surjective.
\end{proof}


\bibliographystyle{plain}
\bibliography{References}

$\,$\\
{V.C.: \small\textsc{Max-Planck Institut for Mathematics in the Sciences, Leipzig, Germany}}\\
\textit{Email address:} \texttt{vcarmonamath@gmail.com}
\vspace*{2mm}

$\,$\\
{A.Š.: \small\textsc{Technical University Munich, Munich, Germany}}\\
\textit{Email address:} \texttt{svr@cit.tum.de} 
 \end{document}

%% file: figure4.tex
\begin{tikzpicture}[scale=0.9]

\begin{scope}
\draw[ blue!50!black] (-2,0) rectangle (4,3);

\draw[green!40!blue!45!black, dashed] (-0.5,-0.7) -- (-0.5,3);
\draw[green!40!blue!45!black, dashed] (-2,0.8) -- (4,0.8);
\draw[green!40!blue!45!black, dashed] (-2,2.2) -- (4,2.2);
\draw[green!40!blue!45!black, dashed] (3,-0.7) -- (3,3);

\draw[green!40!blue!45!black, dashed] (0.3,0) -- (0.3,3);
\draw[green!40!blue!45!black, dashed] (0.9,0) -- (0.9,3);
\draw[green!40!blue!45!black, dashed] (1.3,0) -- (1.3,3);
\draw[green!40!blue!45!black, dashed] (1.9,0) -- (1.9,3);
\draw[green!40!blue!45!black, dashed] (2.2,0) -- (2.2,3);
\draw[green!40!blue!45!black, dashed] (2.8,0) -- (2.8,3);
\draw[green!40!blue!45!black, dashed] (3,0) -- (3,3);

\node[left] at (-2.1,3) {$Y$};
\node[right] at (4,-0.2) {$X$};

\draw[blue!50!black, line width=0.4mm, dashed] (-0.5,0.8) rectangle (3,2.2);
\fill[blue!20] (-0.5,0.8) rectangle (3,2.2);
\foreach \x in {0.3,1.3,2.2}{
  \draw[green!40!blue!45!black, thick, dashed] (\x,0.8) rectangle ++(0.6,1.4);
  \fill[green!40!blue!30] (\x,0.8) rectangle ++(0.6,1.4);
}

\draw[line width=0.4mm,orange!80!black] (-2,0.8) -- (-2,2.2);
\node[thick,orange!80!black] at (-2.3,1.5) {$V$};
\draw[line width=0.4mm,orange!80!black] (-0.5,-0.7) -- (3,-0.7);
\node[thick,orange!80!black] at (-0.9,-0.5) {$U$};
\draw[line width=0.6mm,green!40!blue!45!black] (0.3,0) -- (0.9,0);
\node[thick,teal] at (0.6,-0.4) {$U_1$};
\draw[line width=0.6mm,green!40!blue!45!black] (1.3,0) -- (1.9,0);
\node[thick,teal] at (1.6,-0.4) {$U_2$};
\draw[line width=0.6mm,green!40!blue!45!black] (2.2,0) -- (2.8,0);
\node[thick,teal] at (2.5,-0.4) {$U_3$};

\node[thick] at (1,-1.7) {$\displaystyle \bigotimes_{1\leq i\leq 3}\mathdutchcal{F}(\textcolor{teal}{U_i}\times \textcolor{orange!80!black}{V})\longrightarrow \mathdutchcal{F}(\textcolor{orange!80!black}{U}\times \textcolor{orange!80!black}{V})$};

\end{scope}

\begin{scope}[xshift=6cm]
\draw[ blue!50!black] (0,0) rectangle (6,3);

\node[left] at (-0.1,3) {$Y$};

\node[right] at (6,-0.2) {$X$};

\draw[green!40!blue!45!black, dashed] (1.5,0) -- (1.5,3);
\draw[green!40!blue!45!black, dashed] (5,0) -- (5,3);

\draw[green!40!blue!45!black, dashed] (-0.7,0.8) -- (6,0.8);
\draw[green!40!blue!45!black, dashed] (0,1.2) -- (6,1.2);
\draw[green!40!blue!45!black, dashed] (0,1.4) -- (6,1.4);
\draw[green!40!blue!45!black, dashed] (0,1.6) -- (6,1.6);
\draw[green!40!blue!45!black, dashed] (0,2.1) -- (6,2.1);
\draw[green!40!blue!45!black, dashed] (-0.7,2.2) -- (6,2.2);

\draw[blue!50!black, line width=0.4mm, dashed] (1.5,0.8) rectangle (5,2.2);
\fill[blue!20] (1.5,0.8) rectangle (5,2.2);

\draw[green!40!blue!45!black, dashed, thick] (1.5,1.2) rectangle (5,1.4);
\fill[green!40!blue!30] (1.5,1.2) rectangle (5,1.4);
\draw[green!40!blue!45!black, thick, dashed] (1.5,1.6) rectangle (5,2.1);
\fill[green!40!blue!30] (1.5,1.6) rectangle (5,2.1);

\draw[line width=0.4mm,orange!80!black] (1.5,0) -- (5,0);
\node[thick,orange!80!black] at (3.25,-0.4) {$U$};
\draw[line width=0.4mm,orange!80!black] (-0.7,0.8) -- (-0.7,2.2);
\node[thick,orange!80!black] at (-0.6,0.3) {$V$};
\draw[line width=0.6mm,green!40!blue!45!black] (0,1.2) -- (0,1.4);
\node[thick,teal] at (-0.3,1.3) {$V_1$};
\draw[line width=0.6mm,green!40!blue!45!black] (0,1.6) -- (0,2.1);
\node[thick,teal] at (-0.3,1.85) {$V_2$};

\node[thick] at (3,-1.7) {$\displaystyle \bigotimes_{1\leq j\leq 2}\mathdutchcal{F}(\textcolor{orange!80!black}{U}\times \textcolor{teal}{V_j})\longrightarrow \mathdutchcal{F}(\textcolor{orange!80!black}{U}\times \textcolor{orange!80!black}{V})$};

\end{scope}

\end{tikzpicture}

%% file: figure1.tex
\begin{tikzpicture}[scale = 0.8]
\coordinate (A1) at (-7,0,0);
\coordinate (B1) at (-3,0,0);
\coordinate (C1) at (-3,0,-4);
\coordinate (D1) at (-7,0,-4);
\coordinate (P1) at (-5,0,-1);
\coordinate (Q1) at (-5,0,-3);
\coordinate (R1) at (-6,0,-1);
\coordinate (S1) at (-6,0,-3);

\coordinate (P3) at (-3.5,0,-3);
\coordinate (Q3) at (-3.5,0,-4);
\coordinate (R3) at (-4.5,0,-3);
\coordinate (S3) at (-4.5,0,-4);
\node[dashed,thick,green!40!blue!45!black] at (-5,0.5,-4) {$c$};

\draw[dashed,thick] (D1) -- (A1) -- (B1) -- (C1);
\draw[line width=0.7mm,blue!50!black] (C1) -- (D1);

\fill[blue!10] (D1) -- (A1) -- (B1) -- (C1) -- (D1);
\draw[dashed,thick,green!50!black] (P1) -- (Q1) -- (S1) -- (R1) -- (P1);
\fill[dashed,thick,green!40!blue!30] (P1) -- (Q1) -- (S1) -- (R1) -- (P1);
\node[dashed,thick,green!40!blue!45!black] at (-6.5,0,-2) {\small{$U$}};
\draw[dashed,thick,green!50!black] (P3) -- (Q3) -- (S3) -- (R3) -- (P3);
\fill[dashed,thick,green!40!blue!30] (P3) -- (Q3) -- (S3) -- (R3) -- (P3);
\node[dashed,thick,green!40!blue!45!black] at (-3.5,0,-2) {\small{$V$}};
\end{tikzpicture}


%% file: figure2.tex
\begin{tikzpicture}[scale=0.6] 
    \begin{scope}[yshift=0cm]
        \coordinate (A1) at (-7,0,0);
        \coordinate (B1) at (-3,0,0);
        \coordinate (C1) at (-3,0,-4);
        \coordinate (D1) at (-7,0,-4);
        \coordinate (P1) at (-5,0,-1);
        \coordinate (Q1) at (-5,0,-3);
        \coordinate (R1) at (-6,0,-1);
        \coordinate (S1) at (-6,0,-3);
        \coordinate (P3) at (-3.5,0,-3);
        \coordinate (Q3) at (-3.5,0,-4);
        \coordinate (R3) at (-4.5,0,-3);
        \coordinate (S3) at (-4.5,0,-4);
        \node[dashed,thick,green!40!blue!45!black] at (-2,-0.75,-4) {$\square ^{\, 1,Q_3}$};
        \draw[dashed,thick] (D1) -- (A1) -- (B1) -- (C1);
        \draw[line width=0.7mm,blue!50!black] (C1) -- (D1);
        \fill[blue!10] (D1) -- (A1) -- (B1) -- (C1) -- (D1);
    \end{scope}

    \begin{scope}[yshift=2cm]
        \coordinate (A1) at (-7,0,0);
        \coordinate (B1) at (-3,0,0);
        \coordinate (C1) at (-3,0,-4);
        \coordinate (D1) at (-7,0,-4);
        \coordinate (P1) at (-3.7,0,-1);
        \coordinate (Q1) at (-3.7,0,-3);
        \coordinate (R1) at (-6,0,-1);
        \coordinate (S1) at (-6,0,-3);
        \coordinate (P3) at (-3.5,0,-3);
        \coordinate (Q3) at (-3.5,0,-4);
        \coordinate (R3) at (-4.5,0,-3);
        \coordinate (S3) at (-4.5,0,-4);
        \node[dashed,thick,green!40!blue!45!black] at (-2,-0.75,-4) {$\square ^{\, 1,Q_2}$};
        \draw[dashed,thick] (D1) -- (A1) -- (B1) -- (C1);
        \draw[dashed,thick] (C1) -- (D1);
        \fill[blue!10] (D1) -- (A1) -- (B1) -- (C1) -- (D1);
        \draw[dashed,thick,green!50!black] (P1) -- (Q1) -- (S1) -- (R1) -- (P1);
        \fill[dashed,thick,green!40!blue!30] (P1) -- (Q1) -- (S1) -- (R1) -- (P1);
    \end{scope}

    \begin{scope}[yshift=4cm]
        \coordinate (A1) at (-7,0,0);
        \coordinate (B1) at (-3,0,0);
        \coordinate (C1) at (-3,0,-4);
        \coordinate (D1) at (-7,0,-4);
        \coordinate (P1) at (-5,0,-1);
        \coordinate (Q1) at (-5,0,-3);
        \coordinate (R1) at (-6,0,-1);
        \coordinate (S1) at (-6,0,-3);
        \coordinate (P3) at (-3.5,0,-3);
        \coordinate (Q3) at (-3.5,0,-4);
        \coordinate (R3) at (-4.5,0,-3);
        \coordinate (S3) at (-4.5,0,-4);
        \node[dashed,thick,green!40!blue!45!black] at (-2,-0.75,-4) {$\square ^{\, 1,Q_1}$};
        \draw[dashed,thick] (D1) -- (A1) -- (B1) -- (C1);
        \draw[line width=0.7mm,blue!50!black] (C1) -- (D1);
        \fill[blue!10] (D1) -- (A1) -- (B1) -- (C1) -- (D1);
        \draw[dashed,thick,green!50!black] (P1) -- (Q1) -- (S1) -- (R1) -- (P1);
        \fill[dashed,thick,green!40!blue!30] (P1) -- (Q1) -- (S1) -- (R1) -- (P1);
        \draw[dashed,thick,green!50!black] (P3) -- (Q3) -- (S3) -- (R3) -- (P3);
        \fill[dashed,thick,green!40!blue!30] (P3) -- (Q3) -- (S3) -- (R3) -- (P3);
    \end{scope}
    \end{tikzpicture}

%% file: figure5.tex
\begin{tikzpicture}[scale = 0.63]
\coordinate (O) at (0,0,0);
\coordinate (A) at (4,0,0);
\coordinate (B) at (4,4,0);
\coordinate (C) at (0,4,0);
\coordinate (D) at (0,0,-4);
\coordinate (E) at (4,0,-4);
\coordinate (F) at (4,4,-4);
\coordinate (G) at (0,4,-4);
\coordinate (P) at (2,0,-1);
\coordinate (Q) at (2,0,-3);
\coordinate (R) at (1,0,-1);
\coordinate (S) at (1,0,-3);
\coordinate (P2) at (2,4,-1);
\coordinate (Q2) at (2,4,-3);
\coordinate (R2) at (1,4,-1);
\coordinate (S2) at (1,4,-3);
\coordinate (P4) at (3.5,0,-3);
\coordinate (Q4) at (3.5,0,-4);
\coordinate (R4) at (2.5,0,-3);
\coordinate (S4) at (2.5,0,-4);
\coordinate (P5) at (3.5,4,-3);
\coordinate (Q5) at (3.5,4,-4);
\coordinate (R5) at (2.5,4,-3);
\coordinate (S5) at (2.5,4,-4);
\node at (2,4.7,-4) {$\upmu_{\scalebox{0.8}{$\diameter$}}(c)$};

\fill[blue!20] (O) -- (D) -- (E) -- (A) -- cycle;
\fill[blue!20] (D) -- (E) -- (F) -- (G) -- cycle;

\fill[dashed,thick,green!40!blue!30] (R) -- (P) -- (P2) -- (R2);
\fill[dashed,thick,green!40!blue!30] (P) -- (Q) -- (Q2)--(P2);
\fill[dashed,thick,green!40!blue!30] (P2) -- (Q2) -- (S2)--(R2);
\fill[dashed,thick,green!40!blue!30] (R4) -- (P4) -- (P5) -- (R5);
\fill[dashed,thick,green!40!blue!30] (P4) -- (Q4) -- (Q5)--(P5);
\fill[dashed,thick,green!40!blue!30] (P5) -- (Q5) -- (S5)--(R5);
\fill[dashed,thick,green!40!blue!45] (S4) -- (Q4) -- (Q5) -- (S5);
\fill[ pattern={Lines[
                  distance=2mm,
                  angle=-45,
                  line width=0.7mm
                 ]},
        pattern color=green!30!black!90] (S4) -- (Q4) -- (Q5) -- (S5);
\draw[dashed,thick] (O) -- (D);
\draw[dashed,thick] (C) -- (G);
\draw[dashed,thick] (A) -- (B) -- (C) -- (O) -- (A);
\draw[dashed,thick] (A) -- (E) -- (F) -- (B) ;
\draw[dashed,thick] (D) -- (G) -- (F);

\draw[dashed,thick,green!50!black] (P2) -- (Q2) -- (S2) -- (R2) -- (P2);
\draw[dashed,thick,green!50!black] (P) -- (Q) -- (S) -- (R) -- (P);
\draw[dashed,thick,green!50!black] (P) -- (P2);
\draw[dashed,thick,green!50!black] (Q) -- (Q2);
\draw[dashed,thick,green!50!black] (R) -- (R2);
\draw[dashed,thick,green!50!black] (S) -- (S2);
\draw[dashed,thick,green!50!black] (P5) -- (Q5) -- (S5) -- (R5) -- (P5);
\draw[dashed,thick,green!50!black] (P4) -- (Q4) -- (S4) -- (R4) -- (P4);
\draw[dashed,thick,green!50!black] (P4) -- (P5);
\draw[dashed,thick,green!50!black] (Q4) -- (Q5);
\draw[dashed,thick,green!50!black] (R4) -- (R5);
\draw[dashed,thick,green!50!black] (S4) -- (S5);
\draw[line width=0.3mm,blue!50!black] (D) -- (E);
\end{tikzpicture}
\hspace{2cm} 
\begin{tikzpicture}[scale = 0.63]
\coordinate (O) at (0,0,0);
\coordinate (A) at (4,0,0);
\coordinate (B) at (4,4,0);
\coordinate (C) at (0,4,0);
\coordinate (D) at (0,0,-4);
\coordinate (E) at (4,0,-4);
\coordinate (F) at (4,4,-4);
\coordinate (G) at (0,4,-4);
\coordinate (P) at (2,0,-1);
\coordinate (Q) at (2,0,-3);
\coordinate (R) at (1,0,-1);
\coordinate (S) at (1,0,-3);
\coordinate (P2) at (2,4,-1);
\coordinate (Q2) at (2,4,-3);
\coordinate (R2) at (1,4,-1);
\coordinate (S2) at (1,4,-3);
\coordinate (P4) at (3.5,0,-3);
\coordinate (Q4) at (3.5,0,-4);
\coordinate (R4) at (2.5,0,-3);
\coordinate (S4) at (2.5,0,-4);
\coordinate (P5) at (3.5,4,-3);
\coordinate (Q5) at (3.5,4,-4);
\coordinate (R5) at (2.5,4,-3);
\coordinate (S5) at (2.5,4,-4);
\node at (2,4.7,-4) {$\upmu_{\{1\}}(c)$};
\fill[blue!20] (O) -- (D) -- (E) -- (A) -- cycle;
\fill[blue!20] (D) -- (E) -- (F) -- (G) -- cycle;
\fill[dashed,thick,green!40!blue!30] (R) -- (P) -- (P2) -- (R2);
\fill[dashed,thick,green!40!blue!30] (P) -- (Q) -- (Q2)--(P2);
\fill[dashed,thick,green!40!blue!30] (P2) -- (Q2) -- (S2)--(R2);
\fill[dashed,thick,green!40!blue!45] (P) -- (Q) -- (S)--(R);
\fill[ pattern={Lines[
                  distance=2mm,
                  angle=-45,
                  line width=0.7mm
                 ]},
        pattern color=green!30!black!90] (P) -- (Q) -- (S)--(R);
\fill[dashed,thick,green!40!blue!30] (R4) -- (P4) -- (P5) -- (R5);
\fill[dashed,thick,green!40!blue!30] (P4) -- (Q4) -- (Q5)--(P5);
\fill[dashed,thick,green!40!blue!30] (P5) -- (Q5) -- (S5)--(R5);
\fill[dashed,thick,green!40!blue!45] (P4) -- (Q4) -- (S4)--(R4);
\fill[dashed,thick,green!40!blue!45] (S4) -- (Q4) -- (Q5) -- (S5);
\fill[ pattern={Lines[
                  distance=2mm,
                  angle=-45,
                  line width=0.7mm
                 ]},
        pattern color=green!30!black!90] (S4) -- (Q4) -- (Q5) -- (S5);
\fill[ pattern={Lines[
                  distance=2mm,
                  angle=-45,
                  line width=0.7mm
                 ]},
        pattern color=green!30!black!90] (P4) -- (Q4) -- (S4) -- (R4);
\draw[dashed,thick] (O) -- (D);
\draw[dashed,thick] (C) -- (G);
\draw[dashed,thick] (A) -- (B) -- (C) -- (O) -- (A);
\draw[dashed,thick] (A) -- (E) -- (F) -- (B) ;
\draw[dashed,thick] (D) -- (G) -- (F);

\draw[dashed,thick,green!50!black] (P2) -- (Q2) -- (S2) -- (R2) -- (P2);
\draw[dashed,thick,green!50!black] (P) -- (Q) -- (S) -- (R) -- (P);
\draw[dashed,thick,green!50!black] (P) -- (P2);
\draw[dashed,thick,green!50!black] (Q) -- (Q2);
\draw[dashed,thick,green!50!black] (R) -- (R2);
\draw[dashed,thick,green!50!black] (S) -- (S2);
\draw[dashed,thick,green!50!black] (P5) -- (Q5) -- (S5) -- (R5) -- (P5);
\draw[dashed,thick,green!50!black] (P4) -- (Q4) -- (S4) -- (R4) -- (P4);
\draw[dashed,thick,green!50!black] (P4) -- (P5);
\draw[dashed,thick,green!50!black] (Q4) -- (Q5);
\draw[dashed,thick,green!50!black] (R4) -- (R5);
\draw[dashed,thick,green!50!black] (S4) -- (S5);
\draw[line width=0.3mm,blue!50!black] (D) -- (E);
\end{tikzpicture}.

%% file: figure6.tex
\begin{figure}[htp]
    \centering
\begin{tikzpicture}[scale = 0.5]
\coordinate (O) at (0,0,0);
\coordinate (A) at (4,0,0);
\coordinate (B) at (4,4,0);
\coordinate (C) at (0,4,0);
\coordinate (D) at (0,0,-4);
\coordinate (E) at (4,0,-4);
\coordinate (F) at (4,4,-4);
\coordinate (G) at (0,4,-4);
\coordinate (P) at (2,0,-1);
\coordinate (Q) at (2,0,-3);
\coordinate (R) at (1,0,-1);
\coordinate (S) at (1,0,-3);
\coordinate (P2) at (2,4,-1);
\coordinate (Q2) at (2,4,-3);
\coordinate (R2) at (1,4,-1);
\coordinate (S2) at (1,4,-3);
\coordinate (P4) at (3.5,0,-3);
\coordinate (Q4) at (3.5,0,-4);
\coordinate (R4) at (2.5,0,-3);
\coordinate (S4) at (2.5,0,-4);
\coordinate (P5) at (3.5,4,-3);
\coordinate (Q5) at (3.5,4,-4);
\coordinate (R5) at (2.5,4,-3);
\coordinate (S5) at (2.5,4,-4);
\node at (2,4.7,-4) {$\upmu_{\{1'\}}(c)$};
\node at (6.5,2.6,0) {$\circ$};
\node at (7,2.6,0) {$\Bigg($};

\fill[blue!20] (O) -- (D) -- (E) -- (A) -- cycle;
\fill[blue!20] (D) -- (E) -- (F) -- (G) -- cycle;
\draw[line width=0.4mm,blue!50!black] (0,0,-4) -- (4,0,-4);
\fill[dashed,thick,green!40!blue!30] (R) -- (P) -- (P2) -- (R2);
\fill[dashed,thick,green!40!blue!30] (P) -- (Q) -- (Q2)--(P2);

\fill[dashed,thick,green!40!blue!30] (P2) -- (Q2) -- (S2)--(R2);
\fill[dashed,thick,green!40!blue!45] (P) -- (Q) -- (S)--(R);

\fill[dashed,thick,green!40!blue!30] (R4) -- (P4) -- (P5) -- (R5);
\fill[dashed,thick,green!40!blue!30] (P4) -- (Q4) -- (Q5)--(P5);
\fill[dashed,thick,green!40!blue!30] (P5) -- (Q5) -- (S5)--(R5);
\fill[dashed,thick,green!40!blue!45] (P4) -- (Q4) -- (S4)--(R4);
\fill[dashed,thick,green!40!blue!45] (S4) -- (Q4) -- (Q5) -- (S5);
\draw[dashed,thick] (O) -- (D);
\draw[dashed,thick] (C) -- (G);
\draw[dashed,thick] (A) -- (B) -- (C) -- (O) -- (A);
\draw[dashed,thick] (A) -- (E) -- (F) -- (B) ;
\draw[dashed,thick] (D) -- (G) -- (F);

\draw[dashed,thick,green!50!black] (P2) -- (Q2) -- (S2) -- (R2) -- (P2);
\draw[dashed,thick,green!50!black] (P) -- (Q) -- (S) -- (R) -- (P);
\draw[dashed,thick,green!50!black] (P) -- (P2);
\draw[dashed,thick,green!50!black] (Q) -- (Q2);
\draw[dashed,thick,green!50!black] (R) -- (R2);
\draw[dashed,thick,green!50!black] (S) -- (S2);
\draw[dashed,thick,green!50!black] (P5) -- (Q5) -- (S5) -- (R5) -- (P5);
\draw[dashed,thick,green!50!black] (P4) -- (Q4) -- (S4) -- (R4) -- (P4);
\draw[dashed,thick,green!50!black] (P4) -- (P5);
\draw[dashed,thick,green!50!black] (Q4) -- (Q5);
\draw[dashed,thick,green!50!black] (R4) -- (R5);
\draw[dashed,thick,green!50!black] (S4) -- (S5);
\draw[line width=0.4mm,blue!50!black] (S4) -- (Q4);
\end{tikzpicture}
\begin{tikzpicture}[scale = 0.5]
\coordinate (O) at (0,0,0);
\coordinate (A) at (4,0,0);
\coordinate (B) at (4,4,0);
\coordinate (C) at (0,4,0);
\coordinate (D) at (0,0,-4);
\coordinate (E) at (4,0,-4);
\coordinate (F) at (4,4,-4);
\coordinate (G) at (0,4,-4);
\node at (3,6.2,0) {$\upmu_{\scalebox{0.8}{$\diameter$}}(\upiota')$};
\node at (6.5,1,0) {,};

\fill[blue!20] (O) -- (D) -- (E) -- (A) -- cycle;

\begin{scope} [shift={(1,0,-1)}]
    \coordinate (P6) at (3,0,1);
    \coordinate (Q6) at (3,0,-3);
    \coordinate (R6) at (-1,0,1);
    \coordinate (S6) at (-1,0,-3);
    \coordinate (P7) at (3,1,1);
    \coordinate (Q7) at (3,1,-3);
    \coordinate (R7) at (-1,1,1);
    \coordinate (S7) at (-1,1,-3);
    \fill[dashed,thick,black!20] (R6) -- (P6) -- (P7) -- (R7);
    \fill[dashed,thick,black!20] (P6) -- (Q6) -- (Q7)--(P7);
    \fill[dashed,thick,black!20] (P7) -- (Q7) -- (S7)--(R7);
    \fill[dashed,thick,black!20] (S6) -- (Q6) -- (Q7) -- (S7);
    \fill[dashed,thick,black!50] (P6) -- (Q6) -- (S6) -- (R6);
    
    \draw[dashed,thick,black!90] (P7) -- (Q7) -- (S7) -- (R7) -- (P7);
    \draw[dashed,thick,black!90] (P6) -- (Q6) -- (S6) -- (R6) -- (P6);
    \draw[dashed,thick,black!90] (P6) -- (P7);
    \draw[dashed,thick,black!90] (Q6) -- (Q7);
    \draw[dashed,thick,black!90] (R6) -- (R7);
    \draw[dashed,thick,black!90] (S6) -- (S7);
\end{scope}

\begin{scope} [shift={(0.2,2.8,-1.2)}, scale=0.2]
    \coordinate (P6) at (19,0,6);
    \coordinate (Q6) at (19,0,-14);
    \coordinate (R6) at (-1,0,6);
    \coordinate (S6) at (-1,0,-14);
    \coordinate (P7) at (19,4,6);
    \coordinate (Q7) at (19,4,-14);
    \coordinate (R7) at (-1,4,6);
    \coordinate (S7) at (-1,4,-14);
    \fill[dashed,thick,black!20] (R6) -- (P6) -- (P7) -- (R7);
    \fill[dashed,thick,black!20] (P6) -- (Q6) -- (Q7)--(P7);
    \fill[dashed,thick,black!20] (P7) -- (Q7) -- (S7)--(R7);
    \fill[dashed,thick,black!20] (S6) -- (Q6) -- (Q7) -- (S7);
    \fill[dashed,thick,black!20] (P6) -- (Q6) -- (S6) -- (R6);
    
    \draw[dashed,thick,black!90] (P7) -- (Q7) -- (S7) -- (R7) -- (P7);
    \draw[dashed,thick,black!90] (P6) -- (Q6) -- (S6) -- (R6) -- (P6);
    \draw[dashed,thick,black!90] (P6) -- (P7);
    \draw[dashed,thick,black!90] (Q6) -- (Q7);
    \draw[dashed,thick,black!90] (R6) -- (R7);
    \draw[dashed,thick,black!90] (S6) -- (S7);
\end{scope}

\draw[dashed,thick] (O) -- (D);
\draw[dashed,thick] (C) -- (G);
\draw[dashed,thick] (A) -- (B) -- (C) -- (O) -- (A);
\draw[dashed,thick] (A) -- (E) -- (F) -- (B) ;
\draw[dashed,thick] (D) -- (G) -- (F);
\end{tikzpicture}
\hspace{0.2 cm}
\begin{tikzpicture}[scale = 0.5]
\coordinate (O) at (0,0,0);
\coordinate (A) at (4,0,0);
\coordinate (B) at (4,4,0);
\coordinate (C) at (0,4,0);
\coordinate (D) at (0,0,-4);
\coordinate (E) at (4,0,-4);
\coordinate (F) at (4,4,-4);
\coordinate (G) at (0,4,-4);
\node at (2,4.7,-4) {$\upmu_{\{1\}}(\upiota')$};
\node at (6.5,3,0) {$\Bigg)$};
\node at (7.2,3,0) {$=$};

\fill[blue!20] (O) -- (D) -- (E) -- (A) -- cycle;
\fill[blue!20] (D) -- (E) -- (F) -- (G) -- cycle;
\draw[line width=0.4mm,blue!50!black] (0,0,-4) -- (4,0,-4);

\begin{scope}
\coordinate (P6) at (4,2.7,0);
\coordinate (Q6) at (4,2.7,-4);
\coordinate (R6) at (0,2.7,0);
\coordinate (S6) at (0,2.7,-4);
\coordinate (P7) at (4,3.7,0);
\coordinate (Q7) at (4,3.7,-4);
\coordinate (R7) at (0,3.7,0);
\coordinate (S7) at (0,3.7,-4);
\fill[dashed,thick,black!20] (R6) -- (P6) -- (P7) -- (R7);
\fill[dashed,thick,black!20] (P6) -- (Q6) -- (Q7)--(P7);
\fill[dashed,thick,black!20] (P7) -- (Q7) -- (S7)--(R7);
\fill[dashed,thick,black!50] (S6) -- (Q6) -- (Q7) -- (S7);
\fill[dashed,thick,black!20] (P6) -- (Q6) -- (S6) -- (R6);

\draw[dashed,thick,black!90] (P7) -- (Q7) -- (S7) -- (R7) -- (P7);
\draw[dashed,thick,black!90] (P6) -- (Q6) -- (S6) -- (R6) -- (P6);
\draw[dashed,thick,black!90] (P6) -- (P7);
\draw[dashed,thick,black!90] (Q6) -- (Q7);
\draw[dashed,thick,black!90] (R6) -- (R7);
\draw[dashed,thick,black!90] (S6) -- (S7);
\end{scope}

\begin{scope}
\coordinate (P6) at (4,0,0);
\coordinate (Q6) at (4,0,-4);
\coordinate (R6) at (0,0,0);
\coordinate (S6) at (0,0,-4);
\coordinate (P7) at (4,1,0);
\coordinate (Q7) at (4,1,-4);
\coordinate (R7) at (0,1,0);
\coordinate (S7) at (0,1,-4);
\fill[dashed,thick,black!20] (R6) -- (P6) -- (P7) -- (R7);
\fill[dashed,thick,black!20] (P6) -- (Q6) -- (Q7)--(P7);
\fill[dashed,thick,black!20] (P7) -- (Q7) -- (S7)--(R7);
\fill[dashed,thick,black!50] (S6) -- (Q6) -- (Q7) -- (S7);
\fill[dashed,thick,black!50] (P6) -- (Q6) -- (S6) -- (R6);

\draw[dashed,thick,black!90] (P7) -- (Q7) -- (S7) -- (R7) -- (P7);
\draw[dashed,thick,black!90] (P6) -- (Q6) -- (S6) -- (R6) -- (P6);
\draw[dashed,thick,black!90] (P6) -- (P7);
\draw[dashed,thick,black!90] (Q6) -- (Q7);
\draw[dashed,thick,black!90] (R6) -- (R7);
\draw[dashed,thick,black!90] (S6) -- (S7);
\draw[line width=0.4mm,blue!50!black] (S6) -- (Q6);
\end{scope}

\draw[dashed,thick] (O) -- (D);
\draw[dashed,thick] (C) -- (G);
\draw[dashed,thick] (A) -- (B) -- (C) -- (O) -- (A);
\draw[dashed,thick] (A) -- (E) -- (F) -- (B) ;
\draw[dashed,thick] (D) -- (G) -- (F);
\draw[dashed,thick] (0,0,-4) -- (4,0,-4);
\end{tikzpicture}
 \begin{tikzpicture}[scale = 0.55]
 
\coordinate (O) at (0,0,0);
\coordinate (A) at (4,0,0);
\coordinate (B) at (4,4,0);
\coordinate (C) at (0,4,0);
\coordinate (D) at (0,0,-4);
\coordinate (E) at (4,0,-4);
\coordinate (F) at (4,4,-4);
\coordinate (G) at (0,4,-4);
\coordinate (P) at (2,0,-1);
\coordinate (Q) at (2,0,-3);
\coordinate (R) at (1,0,-1);
\coordinate (S) at (1,0,-3);
\coordinate (P2) at (2,4,-1);
\coordinate (Q2) at (2,4,-3);
\coordinate (R2) at (1,4,-1);
\coordinate (S2) at (1,4,-3);
\coordinate (P4) at (3.5,0,-3);
\coordinate (Q4) at (3.5,0,-4);
\coordinate (R4) at (2.5,0,-3);
\coordinate (S4) at (2.5,0,-4);
\coordinate (P5) at (3.5,4,-3);
\coordinate (Q5) at (3.5,4,-4);
\coordinate (R5) at (2.5,4,-3);
\coordinate (S5) at (2.5,4,-4);

\fill[blue!20] (O) -- (D) -- (E) -- (A) -- cycle;
\fill[blue!20] (D) -- (E) -- (F) -- (G) -- cycle;
\fill[dashed,thick,green!40!blue!30] (R) -- (P) -- (P2) -- (R2);
\fill[dashed,thick,green!40!blue!30] (P) -- (Q) -- (Q2)--(P2);
\fill[dashed,thick,green!40!blue!30] (P2) -- (Q2) -- (S2)--(R2);
\fill[dashed,thick,green!40!blue!45] (P) -- (Q) -- (S)--(R);

\fill[dashed,thick,green!40!blue!30] (R4) -- (P4) -- (P5) -- (R5);
\fill[dashed,thick,green!40!blue!30] (P4) -- (Q4) -- (Q5)--(P5);
\fill[dashed,thick,green!40!blue!30] (P5) -- (Q5) -- (S5)--(R5);
\fill[dashed,thick,green!40!blue!45] (P4) -- (Q4) -- (S4)--(R4);
\fill[dashed,thick,green!40!blue!45] (S4) -- (Q4) -- (Q5) -- (S5);

\draw[line width=0.4mm,blue!50!black] (0,0,-4) -- (4,0,-4);
  
\begin{scope}
\coordinate (P6) at (3.5,2.7,-3);
\coordinate (Q6) at (3.5,2.7,-4);
\coordinate (R6) at (2.5,2.7,-3);
\coordinate (S6) at (2.5,2.7,-4);
\coordinate (P7) at (3.5,3.7,-3);
\coordinate (Q7) at (3.5,3.7,-4);
\coordinate (R7) at (2.5,3.7,-3);
\coordinate (S7) at (2.5,3.7,-4);
\fill[dashed,thick,black!20] (R6) -- (P6) -- (P7) -- (R7);
\fill[dashed,thick,black!20] (P6) -- (Q6) -- (Q7)--(P7);
\fill[dashed,thick,black!20] (P7) -- (Q7) -- (S7)--(R7);
\fill[dashed,thick,black!50] (S6) -- (Q6) -- (Q7) -- (S7);
\fill[dashed,thick,black!20] (P6) -- (Q6) -- (S6) -- (R6);

\draw[dashed,thick,black!90] (P7) -- (Q7) -- (S7) -- (R7) -- (P7);
\draw[dashed,thick,black!90] (P6) -- (Q6) -- (S6) -- (R6) -- (P6);
\draw[dashed,thick,black!90] (P6) -- (P7);
\draw[dashed,thick,black!90] (Q6) -- (Q7);
\draw[dashed,thick,black!90] (R6) -- (R7);
\draw[dashed,thick,black!90] (S6) -- (S7);
\end{scope}

\begin{scope}
\coordinate (P6) at (3.5,0,-3);
\coordinate (Q6) at (3.5,0,-4);
\coordinate (R6) at (2.5,0,-3);
\coordinate (S6) at (2.5,0,-4);
\coordinate (P7) at (3.5,1,-3);
\coordinate (Q7) at (3.5,1,-4);
\coordinate (R7) at (2.5,1,-3);
\coordinate (S7) at (2.5,1,-4);
\fill[dashed,thick,black!20] (R6) -- (P6) -- (P7) -- (R7);
\fill[dashed,thick,black!20] (P6) -- (Q6) -- (Q7)--(P7);
\fill[dashed,thick,black!20] (P7) -- (Q7) -- (S7)--(R7);
\fill[dashed,thick,black!50] (S6) -- (Q6) -- (Q7) -- (S7);
\fill[dashed,thick,black!50] (P6) -- (Q6) -- (S6) -- (R6);

\draw[dashed,thick,black!90] (P7) -- (Q7) -- (S7) -- (R7) -- (P7);
\draw[dashed,thick,black!90] (P6) -- (Q6) -- (S6) -- (R6) -- (P6);
\draw[dashed,thick,black!90] (P6) -- (P7);
\draw[dashed,thick,black!90] (Q6) -- (Q7);
\draw[dashed,thick,black!90] (R6) -- (R7);
\draw[dashed,thick,black!90] (S6) -- (S7);
\draw[line width=0.3mm,blue!50!black] (S6) -- (Q6);
\end{scope}

\begin{scope} [shift={(1.25,0,-1.5)}]
\coordinate (P6) at (0.75,0,0.5);
\coordinate (Q6) at (0.75,0,-1.5);
\coordinate (R6) at (-0.25,0,0.5);
\coordinate (S6) at (-0.25,0,-1.5);
\coordinate (P7) at (0.75,1,0.5);
\coordinate (Q7) at (0.75,1,-1.5);
\coordinate (R7) at (-0.25,1,0.5);
\coordinate (S7) at (-0.25,1,-1.5);
\fill[dashed,thick,black!20] (R6) -- (P6) -- (P7) -- (R7);
\fill[dashed,thick,black!20] (P6) -- (Q6) -- (Q7)--(P7);
\fill[dashed,thick,black!20] (P7) -- (Q7) -- (S7)--(R7);
\fill[dashed,thick,black!20] (S6) -- (Q6) -- (Q7) -- (S7);
\fill[dashed,thick,black!50] (P6) -- (Q6) -- (S6) -- (R6);

\draw[dashed,thick,black!90] (P7) -- (Q7) -- (S7) -- (R7) -- (P7);
\draw[dashed,thick,black!90] (P6) -- (Q6) -- (S6) -- (R6) -- (P6);
\draw[dashed,thick,black!90] (P6) -- (P7);
\draw[dashed,thick,black!90] (Q6) -- (Q7);
\draw[dashed,thick,black!90] (R6) -- (R7);
\draw[dashed,thick,black!90] (S6) -- (S7);
\end{scope}

\begin{scope} [shift={(1.1,2.8,-1.2)}, scale=0.2]
\coordinate (P6) at (4.5,0,1);
\coordinate (Q6) at (4.5,0,-9);
\coordinate (R6) at (-0.5,0,1);
\coordinate (S6) at (-0.5,0,-9);
\coordinate (P7) at (4.5,4,1);
\coordinate (Q7) at (4.5,4,-9);
\coordinate (R7) at (-0.5,4,1);
\coordinate (S7) at (-0.5,4,-9);
\fill[dashed,thick,black!20] (R6) -- (P6) -- (P7) -- (R7);
\fill[dashed,thick,black!20] (P6) -- (Q6) -- (Q7)--(P7);
\fill[dashed,thick,black!20] (P7) -- (Q7) -- (S7)--(R7);
\fill[dashed,thick,black!20] (S6) -- (Q6) -- (Q7) -- (S7);
\fill[dashed,thick,black!20] (P6) -- (Q6) -- (S6) -- (R6);

\draw[dashed,thick,black!90] (P7) -- (Q7) -- (S7) -- (R7) -- (P7);
\draw[dashed,thick,black!90] (P6) -- (Q6) -- (S6) -- (R6) -- (P6);
\draw[dashed,thick,black!90] (P6) -- (P7);
\draw[dashed,thick,black!90] (Q6) -- (Q7);
\draw[dashed,thick,black!90] (R6) -- (R7);
\draw[dashed,thick,black!90] (S6) -- (S7);
\end{scope}

\draw[dashed,black!90] (R5)--(P5)--(Q5)--(S5)--cycle; 
\draw[dashed,black!90] (R2)--(P2)--(Q2)--(S2)--cycle;  
\draw[dashed,black!90] (R5)--(R4) (P5)--(P4) (Q5)--(Q4) (S5)--(S4); 
\draw[dashed,black!90] (R)--(R2) (P)--(P2) (Q)--(Q2) (S)--(S2); 
\draw[dashed,thick] (O) -- (D);
\draw[dashed,thick] (C) -- (G);
\draw[dashed,thick] (A) -- (B) -- (C) -- (O) -- (A);
\draw[dashed,thick] (A) -- (E) -- (F) -- (B);
\draw[dashed,thick] (D) -- (G) -- (F);
\end{tikzpicture}
   \caption{Illustration of the $(j,k)$-component of a morphism in the image of the bifunctor $\upmu\colon \E_{1,1}\times \E_{0,1}\longrightarrow \E_{1,2}$. The accent in $\{1'\}$ denotes that the object belongs to the second factor $\E_{0,1}$.}
    \label{fig: figure 6}
\end{figure}

%% file: figure3.tex
 \begin{tikzpicture}[scale=0.6]
    \begin{scope}[yshift=6cm]
        \coordinate (A1) at (-7,0,0);
        \coordinate (B1) at (-3,0,0);
        \coordinate (C1) at (-3,4,0);
        \coordinate (D1) at (-7,4,0);
        \draw[dashed,thick] (D1) -- (A1) -- (B1) -- (C1);
        \draw[line width=0.7mm,blue!50!black] (C1) -- (D1);
        \fill[blue!10] (D1) -- (A1) -- (B1) -- (C1) -- (D1);
        \coordinate (E) at (-4,4,0);
        \coordinate (F) at (-5,4,0);
        \coordinate (G1) at (-4,3,0);
        \coordinate (H1) at (-5,3,0);
        \coordinate (E2) at (-4,2.5,0);
        \coordinate (F2) at (-5,2.5,0);
        \coordinate (G2) at (-4,0.5,0);
        \coordinate (H2) at (-5,0.5,0);
        \node[dashed,thick,green!40!blue!45!black] at (-1,2,0) {$U \times V_z$};
         \node[dashed,thick,green!40!blue!45!black] at (-15,0.5,0) {$V_z$};
        \draw[dashed,thick,green!50!black] (F) -- (E) -- (G1) -- (H1) -- (F);
        \fill[dashed,thick,green!40!blue!30] (F) -- (E) -- (G1) -- (H1) -- (F);
        \draw[dashed,thick,green!50!black] (F2) -- (E2) -- (G2) -- (H2) -- (F2);
        \fill[dashed,thick,green!40!blue!30] (F2) -- (E2) -- (G2) -- (H2) -- (F2);
        \node[circle, fill=black, inner sep=2pt] at (-4.5,4,0) {};
        \node[dashed,thick,green!40!blue!45!black] at (-4.5,5,0.55) {$x_3$};
        \node[circle, fill=black, inner sep=2pt] at (-4.5,1.5,0) {};
        \node[dashed,thick,green!40!blue!45!black] at (-4.3,1,0.55) {$x_1$};
        \coordinate (D5) at (-16,4,0);       
        \coordinate (A5) at (-16,0,0);
        \draw[line width=0.7mm,blue!50!black] (A5) -- (D5);
        \draw[line width=0.7mm,dashed,orange!80!black] (-18.5,4,0) -- (-18.5,3,0);
         \draw[line width=0.2mm,dashed,orange!80!black] (-18.5,4,0) -- (-16,4,0);
          \draw[line width=0.2mm,dashed,orange!80!black] (-16,3,0) -- (-18.5,3,0);
        \node[thick,orange!80!black] at (-19.5,3.5,0) {$C_3$};
        \draw[line width=0.7mm,dashed,orange!80!black] (-18.5,2.5,0) -- (-18.5,0.5,0);
        \draw[line width=0.2mm,dashed,orange!80!black] (-18.5,2.5,0) -- (-16,2.5,0);
        \draw[line width=0.2mm,dashed,orange!80!black] (-18.5,0.5,0) -- (-16,0.5,0);
        \node[thick,orange!80!black] at (-19.5,1.5,0)  {$C_1$};
        \coordinate (E3) at (-5.5,3,0);
        \coordinate (F3) at (-6.5,3,0);
        \coordinate (G13) at (-5.5,0.5,0);
        \coordinate (H13) at (-6.5,0.5,0);
        \draw[dashed,thick,green!50!black] (F3) -- (E3) -- (G13) -- (H13) -- (F3);
        \fill[dashed,thick,green!40!blue!30] (F3) -- (E3) -- (G13) -- (H13) -- (F3);
        \node[circle, fill=black, inner sep=2pt] at (-6,2,0) {};
        \node[dashed,thick,green!40!blue!45!black] at (-5.7,1.4,0.55) {$x_2$}; 
    
        \draw[line width=0.7mm,dashed,orange!80!black] (-17,0.5,0) -- (-17,3,0);
        \draw[line width=0.2mm,dashed,orange!80!black] (-17,0.5,0) -- (-16,0.5,0);
        \draw[line width=0.2mm,dashed,orange!80!black] (-16,3,0) -- (-17,3,0);
        \node[thick,orange!80!black] at (-17.8,1.75,0) {$C_2$};
        \draw[->] (-8,2,0) -- (-14,2,0);
        \node at (-11,2.5,0) {$\uppi_{V_z}$};
    \end{scope}

    \begin{scope}[yshift=-0.8cm,xshift=-1.5cm]
        \draw[->] (-5,4.5,-4) -- (-5,2,-4);
        \node at (-4.25,3.25,-4) {$\uppi_U$};
        \coordinate (C1) at (-3,0,-4);
        \coordinate (D1) at (-7,0,-4);
        \draw[line width=0.7mm,blue!50!black] (C1) -- (D1);
        \node[dashed,thick,green!40!blue!45!black] at (-0.5,0,-4) {$U$};
        \draw[line width=0.7mm,orange!80!black] (-5.5,0,-4) -- (-6.5,0,-4);
        \node[thick,orange!80!black] at (-6,0.75,-4) {$D_2$};
        \node[thick,orange!80!black] at (-5.5,0,-4) {)};
        \node[thick,orange!80!black] at (-6.5,0,-4) {(};
        \draw[line width=0.7mm,orange!80!black] (-4,0,-4) -- (-5,0,-4);
        \node[thick,orange!80!black] at (-4.5,0.75,-4) {$D_1$}; \node[thick,orange!80!black] at (-4,0,-4) {)};
        \node[thick,orange!80!black] at (-5,0,-4) {(};
    \end{scope}
    \end{tikzpicture}